\documentclass[table,x11names]{article} 

\providecommand{\ip}[2]{\langle #1, #2 \rangle}

  \usepackage{authblk}
  \usepackage[margin=1in]{geometry}
\usepackage{silence}
\usepackage[T1]{fontenc}
\usepackage[utf8]{inputenc}
\usepackage{csquotes}
\usepackage[UKenglish]{babel}
\usepackage{paralist}
\usepackage[shortlabels]{enumitem}
\usepackage{xcolor}
\usepackage{graphicx}
\usepackage{bbm}
\usepackage{amsmath}
\usepackage{amssymb}
\usepackage{amsfonts}
\usepackage{stmaryrd}
\usepackage{mathrsfs}
\usepackage{mathtools}
\usepackage{epstopdf}
\usepackage{comment}
\usepackage{verbatim}
\usepackage[color=yellow]{todonotes}
\usepackage{booktabs}
\usepackage{pifont}
\usepackage{makecell}
\usepackage{upgreek}
\usepackage[dvipsnames]{xcolor}

\usepackage{tabularx}
\usepackage{array}
\usepackage{booktabs}
\usepackage{makecell}
\usepackage{placeins}

\usepackage{xspace}
\usepackage{float}
\usepackage{stackengine}
  \usepackage{amsthm}
  
  \usepackage{hyperref}
  \hypersetup{%
    linktocpage=true,
    colorlinks=true,
    citecolor=darkgreen,
    linkcolor=blue,
    urlcolor=blue,
  }
  \usepackage[nameinlink,capitalise, noabbrev]{cleveref}
  \usepackage{autonum}
  \newcommand{\email}[1]{\href{mailto:#1}{#1}}
  \newenvironment{keywords}{\small \noindent\begin{quote}\textbf{Keywords.}}{\end{quote}}
  
\input{widebar.tex}

\theoremstyle{plain}
  
  \newtheorem{theorem}{Theorem}
 \newtheorem{lemma}[theorem]{Lemma}
  \newtheorem{corollary}[theorem]{Corollary}
  \newtheorem{proposition}[theorem]{Proposition}
  
  \newtheorem{remark}{Remark}

  \newtheorem{assumption}{Assumption}
  
  \numberwithin{equation}{section}
  \numberwithin{theorem}{section}
  
  \crefname{lemma}{Lemma}{Lemmas}
  \crefname{remark}{Remark}{Remarks}
  \crefname{proposition}{Proposition}{Propositions}
  \crefname{section}{Section}{Sections}
  \crefname{subsection}{Subsection}{Subsections}
  \crefname{equation}{}{}
  \Crefname{equation}{Equation}{Equations}
  \Crefname{figure}{Figure}{Figures}
  \usepackage{setspace}
\crefname{remark}{Remark}{Remarks}
\crefname{assumption}{Assumption}{Assumptions}

\crefname{enumi}{}{}

\newcommand{\proofstep}[1]{\par\medskip\noindent\textbf{#1}\quad}

\DeclarePairedDelimiter{\norm}{\lVert}{\rVert}

\DeclarePairedDelimiterXPP\set[1]{}\{\}{}{

#1}
\DeclarePairedDelimiter{\scp}{\langle}{\rangle}

\DeclarePairedDelimiterXPP\condexpect[1]{\mathbb{E}}[]{}{
  
  #1}
\DeclarePairedDelimiterXPP\condexpectwithinitial[2]{\mathbb{E}_{#1}}[]{}{
  
  #2}

\DeclareMathOperator{\diag}{diag}

\DeclareMathOperator{\Law}{Law}
\DeclareMathOperator{\e}{e}

\renewcommand{\d}{\mathrm d}

\newcommand{\R}{\mathbf R}
\newcommand{\real}{\R}

\renewcommand{\leq}{\leqslant}
\renewcommand{\geq}{\geqslant}
\renewcommand{\le}{\leqslant}
\renewcommand{\ge}{\geqslant}

\newcommand{\wasserstein}{\mathcal W}

\newcommand{\allx}[1]{\mathcal{X}^{#1}}

\renewcommand{\d}{\mathrm d}

\newcommand{\emp}[1]{\mu_{\mathcal X^J_{#1}}}
\newcommand{\empbar}[1]{\mu_{\overline{\mathcal X}^J_{#1}}}
\newcommand{\empv}[1]{\mu_{\mathcal V^J_{#1}}}
\newcommand{\empbarv}[1]{\mu_{\overline{\mathcal V}^J_{#1}}}

\newcommand{\proba}{\mathbf P}
\newcommand{\expect}{\mathbf{E}}
\newcommand{\E}{\mathbf{E}}

\newcommand{\meanx}[1]{\mathcal M\bigl(\mu_{\allx{J}_{#1}}\bigr)}

\newcommand{\momentp}[2]{\mathfrak M_{#1}(#2)}

\newcommand{\edecay}[1]{\lambda_{#1}}

\newcommand{\mfldis}{\widebar{\rho}}

\usepackage{tikz}
\usetikzlibrary{shapes.misc}
\usetikzlibrary{positioning}

\cslet{blx@noerroretextools}\empty 
\usepackage[
  backend=biber,
  style=numeric,
  sorting=none,
  maxbibnames=99
]{biblatex}
\DeclareFieldFormat{volume}{volume \textbf{#1}}
\DeclareFieldFormat[article]{volume}{\textbf{#1}}

\newcommand{\eps}{\varepsilon}

\definecolor{darkred}{rgb}{.7,0,0}
\definecolor{darkgreen}{rgb}{.1,.7,0}
\definecolor{SkyBlue}{HTML}{46C5DD}
\newcommand{\dk}[1]{{\color{NavyBlue}{#1}}}

\newcommand{\cmz}[1]{C_{{\rm MZ},#1}}
\newcommand{\cwm}[1]{C_{{\rm WM},#1}}

\AddToHook{env/lemma/begin}{\crefalias{theorem}{lemma}}
\AddToHook{env/proposition/begin}{\crefalias{theorem}{proposition}}
\AddToHook{env/assumption/begin}{\crefalias{theorem}{Assumption}}
\AddToHook{env/remark/begin}{\crefalias{theorem}{remark}}
\AddToHook{env/corollary/begin}{\crefalias{theorem}{corollary}}
\AddToHook{env/result/begin}{\crefalias{theorem}{result}}
\AddToHook{env/definition/begin}{\crefalias{theorem}{definition}}
\AddToHook{env/example/begin}{\crefalias{theorem}{example}}

\makeatletter
\AddToHook{cmd/appendix/before}{\def\cref@section@alias{appendix}}
\makeatother

\title{Uniform-in-time propagation of chaos for \\ Second-Order Consensus-Based Optimization}

  \author[1]{Seung-Yeal Ha$^{a,}$}
  \author[2]{Franca Hoffmann$^{b,}$}
  \author[2]{Dohyeon Kim$^{c, }$}
  \affil[ ]{\footnotesize%
    $^a$\email{syha@snu.ac.kr},
    $^b$\email{franca.hoffmann@caltech.edu},
    $^c$\email{dohyeon@caltech.edu}.
  }
  \affil[1]{\footnotesize Department of Mathematical Sciences, Seoul National University, Republic of Korea}
  \affil[2]{\footnotesize Department of Computing and Mathematical Sciences, Caltech, USA}
  \date{}

\begin{document}
\maketitle

\begin{abstract}
We study second-order Consensus-Based Optimization (CBO), a derivative-free global optimization
algorithm in which the consensus force and the multiplicative exploratory noise act on particle
velocities. We prove quantitative uniform-in-time propagation of chaos for the unmodified
second-order CBO dynamics, together with an almost uniform-in-time stability estimate for the
microscopic particle system.

The proof is not a direct adaptation of the first-order CBO argument. Although both first- and
second-order CBO have multiplicative noise that degenerates near consensus and a shift-invariant
weighted interaction, the kinetic model has an additional structural obstruction: the consensus
mechanism and the stochastic forcing act only on the velocity variable, while the position variable
evolves by transport. Thus spatial concentration has to be recovered indirectly through velocity
dissipation. Moreover, the shift-invariant interaction leaves a translation mode that is not
directly damped by the consensus force, so a standard synchronous coupling in the Euclidean
phase-space distance does not close uniformly in time.

The main idea of the paper is to introduce shifted internal variables that separate the contracting
fluctuation modes from the undamped translation mode. In these variables we build a Lyapunov functional with a position--velocity cross term and prove exponential decay of centered
moments. This decay is the mechanism that makes the time-dependent coupling coefficient
integrable. Combining it with uniform-in-time raw moment bounds, concentration inequalities,
stability estimates for the weighted mean, and a Monte Carlo estimate, we obtain the classical
Monte Carlo rate for propagation of chaos uniformly in time. The system-to-system stability
estimate avoids the mean-field sampling error and yields the faster rate \(O(J^{-q})\).
\end{abstract}


\tableofcontents
\newpage


\section{Introduction}

\subsection{Global Optimization via Interacting Particle Systems}

We consider the global optimization problem
\begin{align}\label{eq:global-opt}
    x_\ast \in \arg\min_{x \in \R^d} f(x),
\end{align}
where \(f:\R^d\to\R\) may be non-convex. Such problems arise in machine learning,
scientific computing, uncertainty quantification, and inverse problems. For instance, in a Bayesian
inverse problem with forward map \(\mathcal G\), observation \(y=\mathcal G(x)+\eta\), noise
covariance \(\Gamma\), and regularizer \(\mathcal R\), the MAP estimator minimizes
\[
    f(x)=\frac12 |y-\mathcal G(x)|_\Gamma^2+\mathcal R(x).
\]
This variational viewpoint is standard in Bayesian inverse problems and uncertainty quantification; see, for instance, \cite{stuart2010inverse,kaipio2005statistical,MR3041539}. In many applications, however, the objective function is available only through pointwise evaluations, or its derivatives are too expensive, noisy, or unreliable to compute. This motivates the use of derivative-free and ensemble-based optimization methods; see \cite{conn2009introduction,rios2013derivative} for general references on derivative-free optimization. Among derivative-free methods, interacting particle algorithms have been widely used in moderate to high dimensions. A swarm of particles explores the objective landscape through stochastic dynamics combining noise-driven exploration and collective exploitation through consensus, alignment, or ensemble-based updates. This viewpoint includes classical metaheuristics such as Particle Swarm Optimization (PSO)~\cite{488968,Huang_2023}, as well as more recent mean-field-oriented methods such as Consensus-Based Optimization (CBO)~\cite{CBO,carrillo2018analytical}, Consensus-Based Sampling~\cite{CBS-Carrillo2021}, and Ensemble Kalman methods~\cite{MR3041539,MR3988266,MR4199469,MR4234152}. As the number of particles tends to infinity, the empirical measure is expected to converge to a nonlinear mean-field law. This mean-field description connects the algorithmic dynamics to tools from stochastic analysis, nonlocal PDEs, and propagation of chaos, and provides a framework for studying both finite-particle error estimates and long-time behavior.

\subsection{First-Order CBO and Known Results}\label{sec:intro-first-order}

The standard first-order, or overdamped, CBO model evolves \(J\) particles according to
\begin{align}\label{eq:first-order-cbo}
    \d X_t^j
    =
    -\bigl(X_t^j-\mathcal M_\alpha(\mu_{X_t^J})\bigr)\,\d t
    +
    \sigma S\bigl(X_t^j-\mathcal M_\alpha(\mu_{X_t^J})\bigr)\,\d W_t^j,
    \qquad j\in[J],
\end{align}
where
\[
    \mu_{X_t^J}:=\frac1J\sum_{j=1}^J\delta_{X_t^j},
    \qquad
    \mathcal M_\alpha(\mu)
    :=
    \frac{\int x\,e^{-\alpha f(x)}\,\mu(\d x)}
         {\int e^{-\alpha f(x)}\,\mu(\d x)}.
\]
Here \(\alpha>0\) is an inverse-temperature parameter. The noise operator is either isotropic,
\[
    S(x)=|x|I_d,
\]
or anisotropic,
\[
    S(x)=\diag(x).
\]
The associated mean-field process when $J \rightarrow \infty$ is
\begin{align}\label{eq:first-order-cbo-mfl}
    \d\overline X_t
    =
    -\bigl(\overline X_t-\mathcal M_\alpha(\overline\rho_t)\bigr)\,\d t
    +
    \sigma S\bigl(\overline X_t-\mathcal M_\alpha(\overline\rho_t)\bigr)\,\d\overline W_t,
    \qquad
    \overline\rho_t=\Law(\overline X_t).
\end{align}
Equivalently, \(\overline\rho_t\) solves the nonlinear nonlocal Fokker--Planck equation
\begin{align}\label{eq:first-order-cbo-fp}
    \partial_t \overline\rho_t
    =
    \nabla_x\cdot
    \Bigl(
        \bigl(x-\mathcal M_\alpha(\overline\rho_t)\bigr)\overline\rho_t
    \Bigr)
    +
    \frac{\sigma^2}{2}
    \nabla_x\cdot\nabla_x\cdot
    \Bigl(
        S\bigl(x-\mathcal M_\alpha(\rho)\bigr)
    S\bigl(x-\mathcal M_\alpha(\rho)\bigr)^\top \overline\rho_t
    \Bigr) \,.
\end{align}
The theory of first-order CBO is now relatively well developed. Long-time convergence of the
mean-field equation toward a consensus point near the global minimizer was studied in
\cite{carrillo2018analytical,fornasier2024consensus}. Particle-level convergence was analyzed in
\cite{MR4179193}, with time-discrete and random-batch variants treated in
\cite{MR4215338,jin2018random}. Quantitative finite-time propagation of chaos was obtained in
\cite{gerber2023meanfield}. Most relevant for the present paper is~\cite{uit_cbo}, which proves
uniform-in-time propagation of chaos for the unmodified first-order system using synchronous
coupling, concentration estimates, and stability bounds for the weighted mean. Before this result,
uniform-in-time estimates were obtained for modified CBO dynamics. The work
\cite{huang2024uniformintimemeanfieldlimitestimate} considers a rescaled/convexified drift, while
\cite{bayraktar2025uniformintimeweakpropagationchaos} proves a uniform-in-time weak
propagation-of-chaos estimate on a bounded search domain by introducing a smooth cutoff in the
volatility.

\subsection{Second-Order CBO from PSO}\label{sec:intro-second-order}

CBO is closely related to Particle Swarm Optimization (PSO), introduced by Kennedy and
Eberhart~\cite{488968}. In its classical discrete form, PSO evolves positions and velocities
\((X_n^j,V_n^j)\) through updates of the schematic form
\begin{align}
    V_{n+1}^j
    &=
    \omega V_n^j
    +
    c_1 r_{1,n}^j\bigl(P_n^j-X_n^j\bigr)
    +
    c_2 r_{2,n}^j\bigl(G_n-X_n^j\bigr),\\
    X_{n+1}^j
    &=
    X_n^j+V_{n+1}^j.
\end{align}
Here \(P_n^j\) denotes the personal best position of particle \(j\), \(G_n\) denotes a global or
neighborhood best position, and \(r_{1,n}^j,r_{2,n}^j\) are random exploration coefficients. Thus,
unlike first-order CBO, PSO is intrinsically inertial: particles move through velocities, and the
optimization mechanism acts through acceleration.

A second-order, or kinetic, CBO model is obtained by replacing the best-particle information in PSO
by the weighted consensus point \(\mathcal M_\alpha\), and then passing to a continuous-time limit.
In this direction, Grassi and Pareschi~\cite{sara_pso} derived the second-order CBO system
\begin{align}\label{eq:intro-k-cbo}
    \d X_t^j &= V_t^j\,\d t, \notag\\
    m\,\d V_t^j
    &=
    -\gamma V_t^j\,\d t
    -
    \bigl(X_t^j-\mathcal M_\alpha(\mu_{X_t^J})\bigr)\,\d t
    +
    \sigma S\bigl(X_t^j-\mathcal M_\alpha(\mu_{X_t^J})\bigr)\,\d W_t^j,
    \qquad j\in[J].
\end{align}
Here \(m>0\) is the mass, or inertia, and \(\gamma>0\) is the friction. Both the consensus force
and the multiplicative noise act on the velocity, while the position evolves by transport. The
zero-inertia connection between PSO-type dynamics and first-order CBO has been rigorously studied
in~\cite{CBO-PSO-zero-inertia-2}. The corresponding synchronously coupled mean-field system for $J \rightarrow \infty$ is
\begin{align}\label{eq:intro-k-cbo-mfl}
    \d\overline X_t^j &= \overline V_t^j\,\d t, \notag\\
    m\,\d\overline V_t^j
    &=
    -\gamma\overline V_t^j\,\d t
    -
    \bigl(\overline X_t^j-\mathcal M_\alpha(\overline\rho_t^X)\bigr)\,\d t
    +
    \sigma S\bigl(\overline X_t^j-\mathcal M_\alpha(\overline\rho_t^X)\bigr)\,\d W_t^j,
\end{align}
where \(\overline\rho_t=\Law(\overline X_t^1,\overline V_t^1)\) and
\(\overline\rho_t^X\) denotes its spatial marginal. The law \(\overline\rho_t\) solves the kinetic
nonlinear Fokker--Planck equation
\begin{align}\label{eq:intro-k-cbo-fp}
    \partial_t\overline\rho_t
    +
    v\cdot\nabla_x\overline\rho_t
    &=
    \frac{\gamma}{m}\nabla_v\cdot\bigl(v\overline\rho_t\bigr)
    +
    \frac1m\nabla_v\cdot
    \Bigl(
        \bigl(x-\mathcal M_\alpha(\overline\rho_t^X)\bigr)\overline\rho_t
    \Bigr) \notag\\
    &\quad+
    \frac{\sigma^2}{2m^2}
    \nabla_v\cdot\nabla_v\cdot
    \Bigl(
        S\bigl(x-\mathcal M_\alpha(\rho^X)\bigr)
    S\bigl(x-\mathcal M_\alpha(\rho^X)\bigr)^\top\overline\rho_t
    \Bigr) \,.
\end{align}
The inertial formulation is closer to the velocity-based update rule of PSO used in practice and is often used to enhance exploration through momentum effects. From the analytical viewpoint,
however, this same kinetic structure makes the mean-field analysis more challenging compared to the overdamped case: the consensus force and the stochastic forcing act on velocity, while spatial contraction must be recovered indirectly through a hypocoercive mechanism.

\subsection{Why the First-Order Argument Does Not Directly Extend}
\label{sec:intro-challenge}

The second-order model inherits several difficulties already present in first-order CBO, including
multiplicative noise that vanishes near consensus and the shift-invariant structure of the weighted
interaction. The new difficulty is that these features now appear in a kinetic system: the
consensus force and the stochastic forcing act only on the velocity variable, while the position
variable evolves only by transport.

This creates two coupled obstructions.

\emph{Kinetic hypocoercivity.}
The position equation has no direct drift or diffusion toward the consensus point. A coercive
estimate in the standard phase-space distance therefore misses the mechanism by which the
dissipative structure in the velocity equation controls spatial concentration. The relevant decay
must instead be proved through a nonlinear functional containing position--velocity cross terms.
In particular, one has to show that the internal spatial fluctuations decay even though the
dissipative terms first appear in the velocity equation.

\emph{Shift invariance and the center mode.}
The CBO interaction depends only on displacements from the weighted mean. This shift-invariant
structure is already present in first-order CBO and is one of the reasons why uniform-in-time
estimates require centered moment decay rather than a direct globally contractive coupling. In the
second-order model, the same obstruction interacts with the kinetic transport structure: the
spatial center-of-mass component of the coupling error has no direct restoring force, and the
velocity fluctuation has to be coupled to the spatial fluctuation in the correct tilted variables.
Thus the first-order centered-moment argument is no longer sufficient by itself. One must identify
a shifted velocity variable so that the velocity equation can transfer dissipation to the spatial
fluctuations and the synchronous coupling can be closed uniformly in time.

A qualitative difference from the first-order theory is that the contracting quantity is no longer purely spatial. In the first-order model, after removing the empirical mean, the centered spatial fluctuations already carry the relevant dissipative structure: the consensus drift acts directly on $X_t^j - \mathcal{M}_\alpha(\mu_{X_t^J})\,,$ so the decay of centered spatial moments holds without any cross-variable mechanism and can be inserted directly into the coupling argument. In the second-order model, the consensus force acts first on the velocity and reaches the position only through transport. The spatial decay is therefore not captured by a purely spatial centered moment; it must instead be extracted from a nonlinear energy that couples position and velocity, so that the dissipation generated in the velocity equation is transferred to the spatial fluctuations. The shifted velocity variable introduced in \cref{sec:setting-transformation} is designed to make this transfer explicit, once the undamped translation mode has been removed.

This shifted formulation is the main structural point of the proof. The shift is not a cosmetic
rewriting of the dynamics; it separates the contracting internal modes from the undamped
translation mode and reveals the dissipative structure of the kinetic system. Without this shift,
the natural unshifted centered functional leaves non-integrable or wrong-sign terms that prevent a
uniform-in-time closure.

\subsection{Literature Context}

\paragraph{Interacting particle systems and propagation of chaos.}
The mean-field formalism and propagation of chaos go back to Kac~\cite{kac1956foundations} and
McKean--Sznitman~\cite{MR1108185,McKean1967PropagationChaos}; see
\cite{ReviewChaintronI,ReviewChaintronII} for modern reviews. Classical finite-time estimates
typically grow exponentially in time. Uniform-in-time estimates are needed when one wants to
combine mean-field limits with long-time asymptotics. Recent results cover McKean--Vlasov
equations~\cite{10.1214/18-ECP150,MR4163850}, singular fluid models~\cite{guillin2024uniform},
unconfined systems~\cite{sticky_coupling}, and entropy-based sharp approaches~\cite{MR4634344}.
General criteria for upgrading finite-time bounds to uniform-in-time bounds are discussed
in~\cite{schuh2024conditionsuniformtimeconvergence}. In second-order settings, such questions are
typically more delicate because spatial coercivity has to be recovered through hypocoercive
mechanisms; see, for instance,
\cite{villani2009hypocoercivity,Jabin2017MeanFL,JABIN20163588,
EberleGuillinZimmer2019,MR2731396,MR3646428,monmarche:hal-01716329,MR4333408}.

\paragraph{Flocking and second-order multi-agent systems.}
Second-order CBO is structurally related to alignment models such as Cucker--Smale
\cite{cucker2007emergent}, where agents evolve through position--velocity dynamics and interact
through collective alignment. Mean-field limits and long-time behavior for such models have been
studied in \cite{MR2425606,MR2536440,MR2596552,MR4349805,MR4405302}. A recent uniform-in-time
propagation-of-chaos result for a first-order Cucker--Smale model appears in
\cite{gerber2025uniformintimepropagationchaoscuckersmale}. Unlike deterministic alignment models,
second-order CBO contains multiplicative stochastic forcing acting only on velocity.

\paragraph{CBO and related particle methods.}
For CBO itself, first-order results were recalled above. Uniform-in-time results are also available
for modified CBO dynamics~\cite{choi2025modifiedconsensusbasedoptimizationmodel}, convexifying
rescalings~\cite{huang2024uniformintimemeanfieldlimitestimate}, and weak convergence
approaches~\cite{bayraktar2025uniformintimeweakpropagationchaos}. Related high-dimensional
particle methods include PSO~\cite{Huang_2023}, Consensus-Based Sampling~\cite{CBS-Carrillo2021},
and Ensemble Kalman methods~\cite{MR3041539,MR3988266,MR4199469,MR4234152}; for the Ensemble
Kalman Sampler, uniform-in-time mean-field estimates were obtained
in~\cite{vaes2024sharppropagationchaosensemble}.

\subsection{Our Contributions}

We prove uniform-in-time mean-field limits for the unmodified second-order CBO system
\(\cref{eq: k-cbo}\), without adding convexifying drift terms, confinement, or cutoffs. The result
extends the uniform-in-time theory of first-order CBO to a genuinely kinetic model derived from
PSO, where the consensus force and multiplicative noise act on velocity rather than directly on
position.

Our main results, stated in \cref{sec: main-results}, are:
\begin{enumerate}[leftmargin=*]
    \item A quantitative \emph{uniform-in-time propagation of chaos} estimate between the
    interacting particle system \(\cref{eq: k-cbo}\) and the McKean--Vlasov system
    \(\cref{eq: k-cbo-mfl}\):
    \[
        \sup_{t\ge0}
        \E\left[
            |X_t^j-\overline X_t^j|^2
            +
            |V_t^j-\overline V_t^j|^2
        \right]
        \le
        \frac{C}{J}.
    \]
    This recovers the classical Monte Carlo rate uniformly over infinite time horizons.

    \item An \emph{almost uniform-in-time stability estimate} between two copies of the interacting
    particle system driven by the same Brownian motions. Since this estimate does not involve the
    empirical sampling error of the mean-field law, it yields the faster rate \(O(J^{-q})\), for any
    admissible \(q>\frac12\), under the corresponding moment assumptions.
\end{enumerate}

The main technical contribution is a shifted hypocoercive framework for the internal fluctuation
dynamics. The first-order proof relies on exponential decay of centered spatial moments. In the
second-order model, the corresponding decay is not visible in the naive centered variables because
position is controlled only through velocity. We therefore introduce a tilted velocity
fluctuation, obtained by subtracting a multiple of the centered position. In these shifted
variables, a Lyapunov functional yields exponential decay of centered moments. This
decay is the ingredient that makes the uniform-in-time synchronous coupling close.

Once this decay is established, the rest of the uniform-in-time argument can be organized around
the first-order strategy of~\cite{uit_cbo}: the decay of internal fluctuations makes the
time-dependent coupling coefficient integrable, while uniform raw moment bounds, concentration
inequalities, stability estimates for the weighted mean, and a Monte Carlo estimate control the
remaining nonlinear terms. The final coupling energy contains both a position--velocity cross term
and a correction removing the undamped center-of-mass spatial mode.

The assumptions used in this work are sufficient conditions under which the shifted hypocoercive mechanism can be isolated and quantified without
additional localization or truncation of the algorithm. In particular, the restrictions on the
parameters ensure that the dissipation in the centered--shifted variables absorbs the contributions
from the multiplicative noise and the weighted-mean perturbation. We do not expect these
conditions to be optimal. Rather, they reflect the present proof strategy, whose goal is to
establish a first uniform-in-time propagation-of-chaos estimate for the unmodified second-order CBO
dynamics while keeping the main algebraic mechanism transparent. Weakening these assumptions, improving the dependence on the inverse-temperature parameter \(\alpha\), and treating broader
classes of objective functions are left as future directions.

\paragraph{Future directions.}
In light of the above discussion, several directions remain open. First, the present analysis is carried out under structural
conditions on the objective and on the parameters that allow the shifted moment
estimates to close. Relaxing these assumptions, in particular for unbounded objectives used in
applications, is an important next step. One possible route is to combine the present argument with
localization or Lyapunov estimates adapted to the objective landscape, rather than relying on
global boundedness assumptions.

Second, the constants in the current estimates still have unfavorable dependence on the
inverse-temperature parameter \(\alpha\). Since large \(\alpha\) is the regime in which the
weighted mean better approximates the global minimizer, improving this dependence would sharpen
the connection between the mean-field limit and the optimization performance of the algorithm.

Third, recent sharp uniform-in-time propagation-of-chaos frameworks based on entropy or BBGKY
hierarchies have been developed for more regular McKean--Vlasov diffusions
\cite{MR4634344,grass2024sharppropagationchaosmckeanvlasov}. Extending such sharp methods to CBO
remains nontrivial because the diffusion is multiplicative and degenerate, the weighted mean is
not globally Lipschitz in the usual Wasserstein distance, and the interaction is shift-invariant.
Understanding whether the present shifted hypocoercive framework can be combined with these
sharper approaches is a natural problem.

Finally, the model studied here is continuous in time. Since PSO and CBO are implemented as
time-discrete algorithms, a uniform-in-time discretization theory for second-order CBO would be a
substantial step toward a complete finite-particle, finite-step convergence theory.

\paragraph{Plan of the paper.}
In \cref{sec:math-setting} we introduce the particle system, the synchronously coupled
mean-field dynamics, and the notation used throughout the paper. We then define the centered and
shifted internal variables that isolate the contracting fluctuation modes from the translation
mode, and we introduce the nonlinear energies used in the proof. The main uniform-in-time
propagation-of-chaos and stability results are stated in \cref{sec: main-results}. Their proofs are
given in \cref{sec: main-results-proof}: the argument first derives a differential inequality for
the modified coupling energy and then uses the exponential decay of the shifted internal moments
to obtain an integrable coupling coefficient. The auxiliary estimates needed in this argument,
including centered moment decay, raw moment bounds, concentration inequalities, stability of the
weighted mean, and the Monte Carlo estimate, are collected in \cref{sec: auxiliary-results}.
Appendix~A explains the role of the shifted hypocoercive formulation, while Appendix~B records why
the unshifted centered functional does not provide a closed uniform-in-time estimate.


\section{Mathematical Setting}\label{sec:math-setting}

The second-order CBO particle system is
\begin{align}
    \d X_t^j &= V_t^j\,\d t, \notag\\
    m\,\d V_t^j
    &=
    -\gamma V_t^j\,\d t
    -
    \bigl(X_t^j-\mathcal M_\alpha(\mu_{X_t^J})\bigr)\,\d t
    +
    \sigma S\bigl(X_t^j-\mathcal M_\alpha(\mu_{X_t^J})\bigr)\,\d W_t^j,
    \qquad j\in[J], \label{eq: k-cbo}
\end{align}
where
\[
    \mu_{X_t^J}:=\frac1J\sum_{j=1}^J\delta_{X_t^j},
    \qquad
    \mathcal M_\alpha(\mu)
    :=
    \frac{\int x\,e^{-\alpha f(x)}\,\mu(\d x)}
         {\int e^{-\alpha f(x)}\,\mu(\d x)}.
\]
The noise operator is either \(S(x)=|x|I_d\) or \(S(x)=\diag(x)\). The synchronously coupled mean-field system for $J \rightarrow \infty$ is
\begin{align}
    \d\overline X_t^j &= \overline V_t^j\,\d t, \notag\\
    m\,\d\overline V_t^j
    &=
    -\gamma \overline V_t^j\,\d t
    -
    \bigl(\overline X_t^j-\mathcal M_\alpha(\overline\rho_t^X)\bigr)\,\d t
    +
    \sigma S\bigl(\overline X_t^j-\mathcal M_\alpha(\overline\rho_t^X)\bigr)\,\d W_t^j,
    \label{eq: k-cbo-mfl}
\end{align}
where \(\overline\rho_t=\Law(\overline X_t^1,\overline V_t^1)\) and
\(\overline\rho_t^X\) denotes its spatial marginal. The particle and mean-field systems are
initialized at the same positions and velocities and driven by the same Brownian motions.

The law \(\overline\rho_t\) solves the kinetic nonlinear Fokker--Planck equation
\begin{align}\label{eq:k-cbo-fp}
    \partial_t\overline\rho_t
    +
    v\cdot\nabla_x\overline\rho_t
    &=
    \frac{\gamma}{m}\nabla_v\cdot\bigl(v\overline\rho_t\bigr)
    +
    \frac1m\nabla_v\cdot
    \Bigl(
        \bigl(x-\mathcal M_\alpha(\overline\rho_t^X)\bigr)\overline\rho_t
    \Bigr) \notag\\
    &\quad+
    \frac{\sigma^2}{2m^2}
    \nabla_v\cdot\nabla_v\cdot
    \Bigl(
        S\bigl(x-\mathcal M_\alpha(\rho^X)\bigr)
    S\bigl(x-\mathcal M_\alpha(\rho^X)\bigr)^\top \overline\rho_t
    \Bigr)\,.
\end{align}

\subsection{Notation}\label{sec:notations}

We write \([J]:=\{1,\dots,J\}\), use \(|\cdot|\) for the Euclidean norm, and
\(\|\cdot\|_{\rm F}\) for the Frobenius norm. For a probability measure \( \rho \) on \(\R^{2d}\),
with generic point \((x,v)\), set
\[
    \mathcal P_p(\R^{2d})
    :=
    \left\{
        \rho \in \mathcal{P}(\R^{2d}) :\int_{\R^{2d}}(|x|^p+|v|^p)\,\rho(\d x,\d v)<\infty
    \right\}.
\]
The Wasserstein-\(p\) distance is denoted by \(\mathcal W_p\). 
For $\mu \in \mathcal P (\R^{d}) \,,$ We define position and velocity means
\[
    m_X(\mu)
    :=
    \int_{\R^d} x\, \mu (\d x),
    \qquad
    m_V(\mu )
    :=
    \int_{\R^d} v\, \mu (\d v).
\]
For the particle system,
\[
    \emp{t}:=\frac1J\sum_{j=1}^J\delta_{X_t^j},
    \qquad
    \empv{t}:=\frac1J\sum_{j=1}^J\delta_{V_t^j},
\]
with means $m_X(\emp{t})\,, m_V(\empv{t})\,.$
For the mean-field samples,
\[
    \empbar{t}:=\frac1J\sum_{j=1}^J\delta_{\overline X_t^j},
    \qquad
    \empbarv{t}:=\frac1J\sum_{j=1}^J\delta_{\overline V_t^j}.
\]
The mean-field law is \(\overline\rho_t\), with spatial and velocity marginals
\(\overline\rho_t^X,\overline\rho_t^V\), and means \(m_X(\overline\rho_t^X)\),
\(m_V(\overline\rho_t^V)\).
For a spatial probability measure \(\mu \in \mathcal{P}(\R^d) \), define
\[
    \Delta_\alpha(\mu):=m_X(\mu)-\mathcal M_\alpha(\mu),
    \qquad
    \momentp{p}{\mu}:=\int |x-m_X(\mu)|^p\,\mu(\d x).
\]
The same notation is used for empirical measures.
We also introduce following notation for estimating coupling error in synchronous coupling argument,
\[
    \delta X_t^j:=X_t^j-\overline X_t^j,
    \qquad
    \delta V_t^j:=V_t^j-\overline V_t^j,
\]
with averages
\[
    \overline{\delta X}_t:=\frac1J\sum_j\delta X_t^j,
    \qquad
    \overline{\delta V}_t:=\frac1J\sum_j\delta V_t^j,
\]
and centered fluctuations
\[
    \widetilde{\delta X}_t^j:=\delta X_t^j-\overline{\delta X}_t,
    \qquad
    \widetilde{\delta V}_t^j:=\delta V_t^j-\overline{\delta V}_t.
\]
Finally,
\[
    \tau(S):=
    \begin{cases}
        d, & S(x)=|x|I_d,\\
        1, & S(x)=\diag(x),
    \end{cases}
    \qquad
    \chi_S:=\tau(S)+p-2.
\]

\subsection{Assumptions}\label{sec:assumptions}
We assume following two \cref{assumption:bounded}, and \cref{assumption:lip} on the objective function $f: \R^d \rightarrow \R\,.$

\begin{assumption}\label{assumption:bounded}
The objective \(f:\R^d\to\R\) is bounded:
\[
    \underline f\le f(x)\le \overline f.
\]
\end{assumption}

\begin{assumption}\label{assumption:lip}
The objective \(f:\R^d\to\R\) is globally Lipschitz:
\[
    |f(x)-f(y)|\le L_f|x-y|.
\]
\end{assumption}
We assume following \cref{assump:admissibility} on the parameters $m, \gamma, \sigma\,.$
\begin{assumption}[Parameter admissibility]\label{assump:admissibility}
For each moment order \(p\ge2\) used below, the parameters \(m,\gamma,\sigma\)
satisfy the following conditions.

\begin{enumerate}[label=(\roman*)]
    \item \textbf{The quadratic case \(p=2\).}
    Define
    \begin{equation}\label{eq:Ksigma-def}
        K_\sigma
        :=
        \frac{2\sigma^2\tau(S)(1+C_{\mathrm{lem}})}{m^2}.
    \end{equation}
    We assume
    \begin{equation}\label{eq:p2-friction-condition}
        \gamma^2>\frac{3m}{2},
        \qquad
        K_\sigma<\frac1{m\gamma}.
    \end{equation}
    Equivalently,
    \begin{equation}\label{eq:p2-small-noise}
        \sigma^2\gamma
        <
        \frac{m}{2\tau(S)(1+C_{\mathrm{lem}})}.
    \end{equation}

    \item \textbf{The high-moment case \(p>2\).}
    Define
    \begin{equation}\label{eq:mu-gap-p-def}
        \mu_{\mathrm{gap},p}
        :=
        2-(p-1)^{-\frac{p-2}{p}}
        \exp\!\left(
            \frac{p-1}{p}\alpha(\overline f-\underline f)
        \right),
    \end{equation}
    and assume
    \begin{equation}\label{eq:mu-gap-condition}
        \mu_{\mathrm{gap},p}>0.
    \end{equation}
    We also assume
    \begin{equation}\label{eq:high-p-mass-bound}
        \gamma^2>m(p-2),
        \qquad
        m<
        \frac{p(16-7\mu_{\mathrm{gap},p})}
        {16(p-1)(p-2)}.
    \end{equation}

    Finally, for the high-moment Lyapunov weight
    \[
        a_p:=\frac1p\left(1-\frac{m(p-2)}{\gamma^2}\right),
    \]
    we assume the strict coercivity condition
    \begin{equation}\label{eq:high-p-coercivity-gamma}
        a_p
        >
        2(p-1)
        \left(\frac{m}{p\gamma}\right)^{\frac{p}{p-1}}.
    \end{equation}
    Equivalently, \(c_1(a_p,p)>0\) in the norm-equivalence estimate of
    \cref{lem:hypercoercivity}. A sufficient lower bound for
    \eqref{eq:high-p-coercivity-gamma} is
    \begin{equation}\label{eq:gamma-coercivity}
        \gamma
        \ge
        \max\left\{
            \sqrt{2m(p-2)},
            \frac{m}{p}\bigl(8p(p-1)\bigr)^{\frac{p-1}{p}}
        \right\}.
    \end{equation}
\end{enumerate}
\end{assumption}
\begin{remark}[Inverse temperature, moment order, and admissibility]
\label{rem:alpha-moment-admissibility}
For a fixed high-moment order \(p>2\), the positivity condition
\(\mu_{\mathrm{gap},p}>0\) is equivalent to
\[
    \alpha(\overline f-\underline f)
    <
    \frac{p\log 2+(p-2)\log(p-1)}{p-1}.
\]
The right-hand side grows like \(\log(2p)\) as \(p\to\infty\). Thus, for a fixed
value of \(\alpha(\overline f-\underline f)\), the positivity condition
\(\mu_{\mathrm{gap},p}>0\) can be enforced by taking the moment order \(p\) large
enough.

This is the basic inverse-temperature/moment-order tradeoff in the present
hypocoercive argument. A larger inverse temperature requires control of higher
centered moments. This should be distinguished from the first-order uniform-in-time
CBO argument, where high moments also enter through the good--bad event decomposition
and the desired stability rate, but not through the same high-moment admissibility
gap. Here the high-moment Lyapunov estimate itself requires
\(\mu_{\mathrm{gap},p}>0\).

There is a cost to increasing \(p\). The remaining admissibility conditions become
stronger:
\[
    \gamma^2>m(p-2),
    \qquad
    m<
    \frac{p(16-7\mu_{\mathrm{gap},p})}
    {16(p-1)(p-2)}
    =
    O(p^{-1})
    \quad\text{as }p\to\infty,
\]
and the coercivity condition \eqref{eq:high-p-coercivity-gamma} imposes an additional
lower bound on \(\gamma\). Thus, using higher moment orders pushes the sufficient
parameter regime toward a higher-friction/lower-mass regime. This is natural for the
second-order dynamics: the consensus force and the multiplicative noise act in the
velocity equation, while spatial contraction is recovered through the 
coupling between position and velocity.

In the main propagation-of-chaos and stability results below, the eighth moment plays
a special role because the bad-set estimate uses the decay rate
\(\kappa=\edecay{8}/8\). Hence the present statements assume admissibility at \(p=8\).
If \(\alpha(\overline f-\underline f)\) is too large for \(\mu_{\mathrm{gap},8}>0\),
one could instead rerun the good--bad event argument with a higher base moment order,
replacing the eighth-moment rate by the corresponding higher-moment rate. This would
require higher finite moments of the initial data and would change the constants and
decay rates in the concentration estimates.

The additional comparisons between different high-moment gaps, such as
\[
    \mu_{\mathrm{gap},2r}>\frac12\mu_{\mathrm{gap},8},
    \qquad
    \mu_{\mathrm{gap},p_\ast}>\frac12\mu_{\mathrm{gap},8},
\]
come from the concentration estimates used in the propagation-of-chaos and stability
proofs; see \cref{rem:concentration-gap-compatibility}. They are separate from the
small-noise restrictions. In particular, decreasing \(\sigma\) helps with the
quadratic and noise-absorption conditions, but it does not by itself change these
high-moment gap comparisons. Finally, note that the conditions in \cref{assump:admissibility} should be read as sufficient conditions
for closing the present high-moment estimates. They are not the necessary conditions for the second-order CBO dynamics.
\end{remark}

\subsection{Centered--Shifted Variables and Decoupling}
\label{sec:setting-transformation}

We now introduce the internal variables used in the estimates. The purpose of this
change of variables is twofold. First, since the CBO interaction is invariant under translations
of the swarm only through relative displacements from the consensus point, the relevant
contractive quantities are not the raw positions and velocities, but their fluctuations around the
empirical mean. Second, in the kinetic system, the naive centered velocity
\(V_t^j-m_V(\empv{t})\) is not the correct variable for closing the moment estimates. The spatial
fluctuation \(X_t^j-m_X(\emp{t})\) and the velocity fluctuation must be tilted against each other
so that the dissipative structure in the velocity equation can be transferred to the spatial
variable.

This is the role of the shifted velocity below. After subtracting the empirical mean, the
deterministic drift in the internal variables becomes a damped linear kinetic system. In
particular, the empirical mean velocity and the weighted-mean displacement cancel from the
deterministic internal drift. The weighted-mean displacement remains only in the martingale
coefficient, where it can later be controlled by centered moment estimates.

Define
\[
    Y_t^j:=X_t^j-m_X(\emp{t}),
    \qquad
    Z_t^j:=V_t^j-\frac1\gamma Y_t^j,
    \qquad
    \hat Z_t^j:=Z_t^j-\frac1J\sum_{k=1}^J Z_t^k .
\]
Since \(\sum_jY_t^j=0\), we have
\[
    \frac1J\sum_{k=1}^J Z_t^k
    =
    \frac1J\sum_{k=1}^J V_t^k
    =
    m_V(\empv{t}),
\]
and therefore
\[
    \hat Z_t^j
    =
    \bigl(V_t^j-m_V(\empv{t})\bigr)
    -
    \frac1\gamma Y_t^j .
\]
In particular,
\[
    \sum_{j=1}^J Y_t^j=0,
    \qquad
    \sum_{j=1}^J \hat Z_t^j=0.
\]
We also write
\[
    \Delta_\alpha(\emp{t})
    :=
    m_X(\emp{t})-\mathcal M_\alpha(\emp{t}).
\]
Then
\[
    X_t^j-\mathcal M_\alpha(\emp{t})
    =
    Y_t^j+\Delta_\alpha(\emp{t}).
\]
Since \(\d m_X(\emp{t})=m_V(\empv{t})\,\d t\), we obtain
\[
    \d Y_t^j
    =
    \bigl(V_t^j-m_V(\empv{t})\bigr)\,\d t
    =
    \left(\hat Z_t^j+\frac1\gamma Y_t^j\right)\d t .
\]

We next compute the dynamics of \(Z_t^j\). From the particle system,
\[
    \d V_t^j
    =
    -\frac\gamma m V_t^j\,\d t
    -
    \frac1m\bigl(Y_t^j+\Delta_\alpha(\emp{t})\bigr)\,\d t
    +
    \frac\sigma m
    S\bigl(Y_t^j+\Delta_\alpha(\emp{t})\bigr)\,\d W_t^j .
\]
Using \(Z_t^j=V_t^j-\gamma^{-1}Y_t^j\) and
\[
    V_t^j
    =
    \hat Z_t^j
    +
    m_V(\empv{t})
    +
    \frac1\gamma Y_t^j,
\]
we find
\begin{align}
    \d Z_t^j
    &=
    \left(
        -K_Z\hat Z_t^j
        -K_Y Y_t^j
        -\frac1m\Delta_\alpha(\emp{t})
        -\frac\gamma m m_V(\empv{t})
    \right)\d t
    \notag\\
    &\qquad
    +
    \frac\sigma m
    S\bigl(Y_t^j+\Delta_\alpha(\emp{t})\bigr)\,\d W_t^j,
    \label{eq:sde-Z-uncentered}
\end{align}
where
\[
    K_Z:=\frac\gamma m+\frac1\gamma,
    \qquad
    K_Y:=\frac2m+\frac1{\gamma^2}.
\]

Averaging \(\cref{eq:sde-Z-uncentered}\) over \(j\), and using
\(\sum_jY_t^j=\sum_j\hat Z_t^j=0\), gives
\[
    \d m_Z(t)
    =
    \left(
        -\frac1m\Delta_\alpha(\emp{t})
        -\frac\gamma m m_V(\empv{t})
    \right)\d t
    +
    \d\overline M_t,
\]
where
\[
    m_Z(t):=\frac1J\sum_{k=1}^J Z_t^k=m_V(\empv{t}),
    \qquad
    \d\overline M_t
    :=
    \frac1J\sum_{k=1}^J
    \frac\sigma m
    S\bigl(Y_t^k+\Delta_\alpha(\emp{t})\bigr)\,\d W_t^k .
\]
Subtracting this mean equation from \(\cref{eq:sde-Z-uncentered}\) yields
\begin{equation}\label{eq:sde-Z-hat}
    \d\hat Z_t^j
    =
    (-K_Z\hat Z_t^j-K_Y Y_t^j)\d t
    +
    \d M_t^{j,\rm center},
\end{equation}
where
\[
    \d M_t^{j,\rm center}
    :=
    \frac\sigma m
    S\bigl(Y_t^j+\Delta_\alpha(\emp{t})\bigr)\,\d W_t^j
    -
    \frac1J\sum_{k=1}^J
    \frac\sigma m
    S\bigl(Y_t^k+\Delta_\alpha(\emp{t})\bigr)\,\d W_t^k.
\]

Thus the centered--shifted variables satisfy the internal system
\[
    \begin{cases}
    \d Y_t^j
    =
    \left(\hat Z_t^j+\dfrac1\gamma Y_t^j\right)\d t,\\[0.8em]
    \d\hat Z_t^j
    =
    (-K_Z\hat Z_t^j-K_Y Y_t^j)\d t
    +
    \d M_t^{j,\rm center}.
    \end{cases}
\]
The deterministic part is now a closed damped linear system in
\((Y_t^j,\hat Z_t^j)\): the empirical mean velocity and the weighted-mean displacement have
canceled from the drift after centering. The nonlinear weighted-mean shift
\(\Delta_\alpha(\emp{t})\) remains only inside the centered martingale coefficient, where it is
controlled later by moment estimates for the internal fluctuations. This algebraic decoupling is
the reason for using the shifted velocity \(\hat Z_t^j\), rather than the naive centered velocity
\(V_t^j-m_V(\empv{t})\).

\subsection{Lyapunov Functionals and Norm Equivalence}\label{sec:lyapunov}

We use modified energy functionals to transfer velocity dissipation into spatial control.
This is needed in two places: first, to control the centered moments of the particle cloud
in \cref{thm:exp-decay-centered-moments}; second, to control the synchronous coupling error
in \cref{thm:uit-poc-kinetic}. The two constructions use the same hypocoercive idea, but
they are not identical. In particular, the \(p=2\) centered-moment estimate uses a dedicated
quadratic form \(\psi_2\), while the coupling estimate uses the quadratic specialization of
a universal polynomial form \(\phi_{a,p}\).

\paragraph{Algebraic forms}
For \(p\ge2\) and a spatial weight \(a>0\), define the polynomial form
\begin{equation}\label{eq:phi-ap}
    \phi_{a,p}(x,v)
    :=
    a|x|^p+|v|^p+\frac m\gamma |x|^{p-2}\langle x,v\rangle.
\end{equation}
The quadratic specialization \(\phi_{a,2}\) is used for the coupling energy. The cross
coefficient \(m/\gamma\) is chosen to match the high-friction scaling in the dynamics.

For the \(p=2\) centered-moment estimate, we use a dedicated quadratic form
\begin{equation}\label{eq:psi2-centered}
    \psi_2(x,v)
    :=
    a_2|x|^2+|v|^2+\frac5\gamma\langle x,v\rangle,
    \qquad
    a_2:=\frac9{2m}+\frac1{\gamma^2}.
\end{equation}

\begin{lemma}[Norm equivalence of the algebraic forms]\label{lem:hypercoercivity}
For \(\phi_{a,p}\), define the constants \(c_1(a,p)\) and \(c_2(a,p)\) by
\begin{equation}\label{eq:c-phi}
    c_1(a,p) :=
    \begin{cases}
        \frac{a+1-\sqrt{(a-1)^2+(m/\gamma)^2}}{2}, & p=2,\\[1.2em]
        \min\!\left\{ a-2(p-1)\left(\frac{m}{p\gamma}\right)^{\frac p{p-1}}, \, 1-\frac1{2^{p-1}} \right\}, & p>2,
    \end{cases}
\end{equation}
and
\begin{equation}
    c_2(a,p) :=
    \begin{cases}
        \frac{a+1+\sqrt{(a-1)^2+(m/\gamma)^2}}{2}, & p=2,\\[1.2em]
        \max\!\left\{ a+2(p-1)\left(\frac{m}{p\gamma}\right)^{\frac p{p-1}}, \, 1+\frac1{2^{p-1}} \right\}, & p>2.
    \end{cases}
\end{equation}
If \(c_1(a,p)>0\), then
\begin{equation}\label{eq:phi-norm-equivalence}
    c_1(a,p)(|x|^p+|v|^p) \le \phi_{a,p}(x,v) \le c_2(a,p)(|x|^p+|v|^p).
\end{equation}

For \(\psi_2\), define
\begin{equation}\label{eq:c-psi}
    c_1^{(2)} := \frac{a_2+1-\sqrt{(a_2-1)^2+(5/\gamma)^2}}{2},
    \qquad
    c_2^{(2)} := \frac{a_2+1+\sqrt{(a_2-1)^2+(5/\gamma)^2}}{2}.
\end{equation}
If \(c_1^{(2)}>0\), then
\begin{equation}\label{eq:psi2-norm-equivalence}
    c_1^{(2)}(|x|^2+|v|^2) \le \psi_2(x,v) \le c_2^{(2)}(|x|^2+|v|^2).
\end{equation}
In particular, the quadratic admissibility condition \(\gamma^2>3m/2\) implies \(\gamma^2>7m/6\), which guarantees \(c_1^{(2)}>0\).
\end{lemma}

\begin{proof}
For \(p=2\), the constants for \(\phi_{a,2}\) are the eigenvalues of the symmetric matrix
\(\begin{pmatrix} a & m/(2\gamma) \\ m/(2\gamma) & 1 \end{pmatrix}\), which directly yields the \(p=2\) cases in \eqref{eq:c-phi}.

For \(p>2\), applying Young's inequality with parameter \(\delta>0\) yields
\[
    \frac m\gamma |x|^{p-2}|\langle x,v\rangle|
    \le
    \frac m\gamma
    \left(
        \frac{p-1}{p}\delta |x|^p + \frac1p\delta^{-(p-1)}|v|^p
    \right).
\]
Choosing \(\delta = 2\left(\frac{m}{p\gamma}\right)^{\frac1{p-1}}\) yields the \(p>2\) cases in \eqref{eq:c-phi} and establishes \eqref{eq:phi-norm-equivalence}.

For \(\psi_2\), the eigenvalues of the associated matrix \(\begin{pmatrix} a_2 & 5/(2\gamma) \\ 5/(2\gamma) & 1 \end{pmatrix}\) are \(c_1^{(2)}\) and \(c_2^{(2)}\). The positivity condition \(c_1^{(2)}>0\) is equivalent to \(a_2 > \frac{25}{4\gamma^2}\). Substituting \(a_2 = \frac{9}{2m} + \frac{1}{\gamma^2}\) gives \(\gamma^2 > 7m/6\).
\end{proof}

\paragraph{Centered Lyapunov functionals}
To analyze the internal variance, we evaluate the algebraic forms on the centered--shifted variables \((Y_t^j,\hat Z_t^j)\) introduced in \cref{sec:setting-transformation}. Define the spatial weight for \(p>2\) as
\begin{equation}\label{eq:a-p-centered}
    a_p := \frac1p\left(1-\frac{m(p-2)}{\gamma^2}\right).
\end{equation}
We define the pathwise centered Lyapunov functional by
\begin{equation}\label{eq:pathwise-Lp}
    L_p(t)
    :=
    \begin{cases}
        \displaystyle \frac1J\sum_{j=1}^J \psi_2(Y_t^j,\hat Z_t^j), & p=2,\\[1.5em]
        \displaystyle \frac1J\sum_{j=1}^J \phi_{a_p,p}(Y_t^j,\hat Z_t^j), & p>2.
    \end{cases}
\end{equation}
The ensemble-averaged Lyapunov functional \(\mathcal L_p(t)\) is defined as
\begin{equation}\label{eq:ensemble-Lp}
    \mathcal L_p(t) := \E[L_p(t)].
\end{equation}
The pathwise form \(L_p(t)\) is used in the concentration estimates, while \(\mathcal L_p(t)\) is used in the centered-moment decay theorem.

\begin{corollary}[Equivalence between centered Lyapunov functionals and moments]
\label{cor:centered-lyapunov-moment-equivalence}
Let \(p\ge2\). Under the parameter conditions of \cref{assump:admissibility}, the strict positivity of the equivalence constants \(c_1^{(2)}>0\) and \(c_1(a_p,p)>0\) is guaranteed. Define the constant \(K_{p,\gamma} := \max\{1+2^{p-1}\gamma^{-p}, \, 2^{p-1}\}\), and set
\begin{equation}
    C_{p,1} := \frac{1}{K_{p,\gamma}} \begin{cases} c_1^{(2)}, & p=2,\\ c_1(a_p,p), & p>2 \end{cases}, \qquad
    C_{p,2} := K_{p,\gamma} \begin{cases} c_2^{(2)}, & p=2,\\ c_2(a_p,p), & p>2 \end{cases}.
\end{equation}
Then the pathwise functional satisfies
\begin{equation}\label{eq:pathwise-Lp-centered-moment-equiv}
    C_{p,1} \left( \momentp{p}{\emp{t}} + \momentp{p}{\empv{t}} \right)
    \le L_p(t) \le
    C_{p,2} \left( \momentp{p}{\emp{t}} + \momentp{p}{\empv{t}} \right).
\end{equation}
Consequently, taking expectations yields
\begin{equation}\label{eq:expected-Lp-centered-moment-equiv}
    C_{p,1} \E\left[ \momentp{p}{\emp{t}} + \momentp{p}{\empv{t}} \right]
    \le \mathcal L_p(t) \le
    C_{p,2} \E\left[ \momentp{p}{\emp{t}} + \momentp{p}{\empv{t}} \right].
\end{equation}
\end{corollary}

\begin{proof}
By \cref{lem:hypercoercivity}, \(L_p(t)\) is bounded above and below by the sum of \(\frac1J\sum |Y_t^j|^p\) and \(\frac1J\sum |\hat Z_t^j|^p\) using the respective equivalence constants.
Since \(V_t^j-m_V(\empv{t}) = \hat Z_t^j+\gamma^{-1} Y_t^j\), the elementary inequalities \(|u+v|^p\le 2^{p-1}(|u|^p+|v|^p)\) and \(|u|^p\le 2^{p-1}(|u+v|^p+|v|^p)\) yield
\[
    \momentp{p}{\empv{t}} \le 2^{p-1} \left( \frac1J\sum_{j=1}^J|\hat Z_t^j|^p + \gamma^{-p}\momentp{p}{\emp{t}} \right) \quad \text{and} \quad \frac1J\sum_{j=1}^J|\hat Z_t^j|^p \le 2^{p-1} \left( \momentp{p}{\empv{t}} + \gamma^{-p}\momentp{p}{\emp{t}} \right).
\]
Combining these inequalities bounds the sum of \(Y\) and \(\hat Z\) moments by the physical centered moments with factor \(K_{p,\gamma}\). Absorbing this into the equivalence constants gives \eqref{eq:pathwise-Lp-centered-moment-equiv}.
\end{proof}

\paragraph{Coupling energy}
Recap that for the synchronous coupling differences \((\delta X_t^j,\delta V_t^j)\), we defined their empirical means \(\overline{\delta X}_t := \frac1J\sum_{j=1}^J\delta X_t^j\) and \(\overline{\delta V}_t := \frac1J\sum_{j=1}^J\delta V_t^j\), and the centered fluctuations \(\widetilde{\delta X}_t^j := \delta X_t^j-\overline{\delta X}_t\) and \(\widetilde{\delta V}_t^j := \delta V_t^j-\overline{\delta V}_t\).

The modified coupling energy \(\mathcal E_t\) is defined using \(\phi_{a,2}\):
\begin{equation}\label{eq:coupling-energy-defn}
    \mathcal E_t := \frac1J\sum_{j=1}^J \phi_{a,2}(\delta X_t^j,\delta V_t^j) - \frac1m|\overline{\delta X}_t|^2, \qquad a:=\frac12+\frac1m.
\end{equation}
The corresponding Euclidean coupling error is
\begin{equation}\label{eq:hat-E-def}
    \widehat{\mathcal E}_t := \frac1J\sum_{j=1}^J \left( |\delta X_t^j|^2 + |\delta V_t^j|^2 \right).
\end{equation}

\begin{corollary}[Equivalence of coupling energies]\label{cor:coupling-energy-equivalence}
Let \(a=1/2+1/m\). Define the global coupling constants
\begin{equation}
    c_{\mathcal{E},1} := \min\left\{ c_1(a,2), \, c_1\left(a-\frac1m,2\right) \right\}, \qquad c_{\mathcal{E},2} := \max\left\{ c_2(a,2), \, c_2\left(a-\frac1m,2\right) \right\}.
\end{equation}
If \(\gamma>m/\sqrt2\), then \(c_{\mathcal{E},1} > 0\) and
\begin{equation}\label{eq:coupling-energy-equivalence}
    c_{\mathcal{E},1} \widehat{\mathcal E}_t \le \mathcal E_t \le c_{\mathcal{E},2} \widehat{\mathcal E}_t.
\end{equation}
The same estimate holds after taking expectations.
\end{corollary}

\begin{proof}
Since \(\frac1J\sum \widetilde{\delta X}_t^j = 0\) and \(\frac1J\sum \widetilde{\delta V}_t^j = 0\), the cross terms cancel, yielding the orthogonal decompositions:
\[
    \widehat{\mathcal E}_t = \frac1J\sum_{j=1}^J \left( |\widetilde{\delta X}_t^j|^2 + |\widetilde{\delta V}_t^j|^2 \right) + |\overline{\delta X}_t|^2 + |\overline{\delta V}_t|^2,
\]
and
\[
    \mathcal E_t = \frac1J\sum_{j=1}^J \phi_{a,2}(\widetilde{\delta X}_t^j,\widetilde{\delta V}_t^j) + \phi_{a-\frac1m,2}(\overline{\delta X}_t,\overline{\delta V}_t).
\]
Since \(a-1/m=1/2\), the center-of-mass block is positive definite if \(1/2 > m^2/(4\gamma^2)\), which is equivalent to \(\gamma>m/\sqrt2\). Under this condition, the fluctuation block is also positive definite. Applying \eqref{eq:phi-norm-equivalence} to both blocks gives \eqref{eq:coupling-energy-equivalence}.
\end{proof}

\begin{remark}[Terminology]
The quantities \(L_p(t)\), \(\mathcal L_p(t)\), and \(\mathcal E_t\) are called Lyapunov functionals in the stochastic sense: their It\^o differentials contain local martingale terms, and the decay estimates hold after taking expectations. The pathwise form \(L_p(t)\) is retained because it is necessary for the concentration estimates.
\end{remark}


\subsection{Proof Strategy}
\label{sec:proof-overview}

We prove
\[
    \sup_{t\ge0}\E[\widehat{\mathcal E}_t]\lesssim J^{-1},
    \qquad
    \widehat{\mathcal E}_t
    :=
    \frac1J\sum_{j=1}^J
    \bigl(|\delta X_t^j|^2+|\delta V_t^j|^2\bigr),
\]
by combining five ingredients: centered-moment decay, raw moment bounds,
concentration estimates, weighted-mean stability, and a synchronous coupling
estimate.

\paragraph{Step 1: Centered moment decay.}
The first ingredient is the decay of internal fluctuations. Using the
centered--shifted variables \((Y,\hat Z)\) from
\cref{sec:setting-transformation}, we use the pathwise Lyapunov functional
\(L_p(t)\), built from \(\phi_{a_p,p}\) for \(p>2\), and the dedicated quadratic
functional \(L_2(t)\), built from \(\psi_2\). Their expected counterparts are denoted
by \(\mathcal L_p(t):=\E[L_p(t)]\). The shifted variable \(\hat Z\) is chosen so that
the deterministic drift of \((Y,\hat Z)\) is closed and linear, while the nonlinear
weighted-mean shift enters only through the noise. Under the admissibility and
small-noise conditions, \cref{thm:exp-decay-centered-moments} gives
\[
    \E\left[
        \momentp{p}{\emp{t}}
        +
        \momentp{p}{\empv{t}}
    \right]
    \le
    C e^{-\edecay{p} t}.
\]
The corresponding estimate for the synchronously coupled mean-field samples is
proved in the same subsection. These decay estimates provide the integrable
coefficients used later in the coupling argument.

\paragraph{Step 2: Raw moment bounds.}
The centered estimates do not by themselves prevent the whole swarm from
drifting. We therefore prove uniform-in-time raw moment bounds. First,
\cref{prop:raw-velocity-decay} controls the raw velocity by using
\[
    X_t^j-\mathcal M_\alpha(\emp{t})
    =
    Y_t^j+\Delta_\alpha(\emp{t}),
\]
together with the centered-moment decay from Step 1. Then
\cref{lem:uit-raw-moments-particle,lem:uit-raw-moments-mf} control the raw
positions by integrating the velocity equation:
\[
    X_t^j=X_0^j+\int_0^t V_s^j\,\d s.
\]
For the propagation-of-chaos proof, the important point is the eighth-order
spatial bound. It is needed on the bad set after applying Cauchy--Schwarz to
terms of the form
\[
    \bigl(\momentp{2}{\emp{t}}+\momentp{2}{\empbar{t}}\bigr)
    \mathcal W_2^2(\emp{t},\empbar{t}).
\]

\paragraph{Step 3: Concentration of centered moments.}
The weighted-mean stability estimate depends on the centered spatial moments,
so we need a high-probability version of Step 1. For \(q>2\) and
\[
    \kappa<\min\{\edecay{2},\edecay{2q}/q\},
\]
\cref{lem:concentration-ineq} gives
\[
    \mathbb P\left[
        \sup_{t\ge0} e^{\kappa t}L_2(t)
        \ge
        \mathcal L_2(0)+A
    \right]
    \le
    \frac{C_{\rm Bad}}{A^q}J^{-q/2}\mathcal L_{2q}(0).
\]
By norm equivalence, this yields the same type of bound for
\(\momentp{2}{\emp{t}}\). The analogous estimate for the empirical measure of
the independent mean-field samples is \cref{lem:concentration-mf}. In the proof
of \cref{thm:uit-poc-kinetic}, we apply this with \(q=r\). After the good--bad
decomposition and one Cauchy--Schwarz step, the bad-set contribution scales as
\(J^{-r/4}\), hence the assumption \(2r\ge8\) ensures an \(O(J^{-1})\) remainder.

\paragraph{Step 4: Weighted-mean stability and Monte Carlo sampling.}
The nonlinear part of the weighted mean is controlled by the Monte-Carlo estimate
\(\cref{lem:stability_est}\):
\[
    |\Delta_\alpha(\mu)-\Delta_\alpha(\nu)|
    \le
    C_M
    \bigl(
        \sqrt{\momentp{2}{\mu}}
        +
        \sqrt{\momentp{2}{\nu}}
    \bigr)
    \mathcal W_2(\mu,\nu),
    \qquad
    \Delta_\alpha(\mu):=m_X(\mu)-\mathcal M_\alpha(\mu).
\]
Applied to \(\mu=\emp{t}\) and \(\nu=\empbar{t}\), this produces the stability
error
\[
    \E|\Delta_\alpha(\emp{t})-\Delta_\alpha(\empbar{t})|^2.
\]
The good event from Step 3 gives an integrable coefficient times the coupling
energy, while the bad event is controlled using the raw eighth moment from
Step 2.

There is also a sampling error between the empirical measure of the mean-field
particles and the true mean-field law:
\[
    D_t
    :=
    \mathcal M_\alpha(\empbar{t})
    -
    \mathcal M_\alpha(\overline\rho_t^X).
\]
By \cref{lem:stability_est},
\[
    \E|D_t|^2
    \le
    \frac{C_{{\rm WM},2}}{J}
    \E\bigl|\overline X_t^1-m_X(\overline\rho_t^X)\bigr|^2
    \le
    \frac{C}{J}e^{-\edecay{2}t}.
\]
This is the term that fixes the propagation-of-chaos rate at the classical
Monte Carlo order \(J^{-1}\). It is absent in the particle-to-particle
stability estimate, which is why \cref{thm:stab_particle} can yield the faster
\(J^{-q}\) remainder.

\paragraph{Step 5: Synchronous coupling and Gr\"onwall closure.}
The final step is to differentiate the modified coupling energy
\[
    \mathcal E_t
    =
    \frac1J\sum_{j=1}^J
    \phi_{a,2}(\delta X_t^j,\delta V_t^j)
    -
    \frac1m|\overline{\delta X}_t|^2.
\]
The subtraction of the center-of-mass spatial mode is used to align the
cross-term coefficients in the fluctuation and center-of-mass blocks. With
\(a=1/2+1/m\), these cross terms cancel, and the remaining linear part gives
dissipation on the spatial fluctuations and velocity errors. The stochastic
trace is absorbed by the small-noise condition.

Combining the bounds from Steps 3--4 yields
\[
    \frac{\d}{\d t}\E[\mathcal E_t]
    \le
    \beta(t)\E[\mathcal E_t]
    +
    \frac1J\alpha(t),
    \qquad
    \alpha,\beta\in L^1(0,\infty).
\]
Since the synchronous coupling starts from identical initial data,
\(\mathcal E_0=0\). Gr\"onwall's inequality gives
\[
    \sup_{t\ge0}\E[\mathcal E_t]\lesssim J^{-1}.
\]
The norm equivalence between \(\mathcal E_t\) and
\(\widehat{\mathcal E}_t\) then gives the desired phase-space estimate, and the
spatial Wasserstein estimate follows from the synchronous empirical coupling.

\subsection{Summary of Structural Constants}
\label{sec:summary-constants}
We collect the constants needed in the main statements and estimates. Longer bookkeeping constants from concentration and stability estimates are recorded where they are introduced.

Set
\[
    a_2:=\frac9{2m}+\frac1{\gamma^2},
    \qquad
    c_\pm^{(2)}
    :=
    \frac{a_2+1\pm\sqrt{(a_2-1)^2+(5/\gamma)^2}}{2},
\]
and write \(c_1^{(2)}:=c_-^{(2)}\), \(c_2^{(2)}:=c_+^{(2)}\). These constants belong to
\(\psi_2\), not to \(\phi_{a,2}\). 

\begin{table}[!htbp]
    \footnotesize
    \centering
    \setlength{\tabcolsep}{3pt}
    \renewcommand{\arraystretch}{1.18}
    \begin{tabularx}{\textwidth}{
        >{\raggedright\arraybackslash}p{2.1cm}
        >{\raggedright\arraybackslash}p{2.7cm}
        >{\raggedright\arraybackslash}p{2.8cm}
        X
    }
        \toprule
        \textbf{Constant}
        & \textbf{Related result}
        & \textbf{Depends on}
        & \textbf{Mathematical expression} \\
        \midrule

        \(K_Z,K_Y\)
        & Variable decoupling
        & \(m,\gamma\)
        & \(\displaystyle
            K_Z:=\frac{\gamma}{m}+\frac1\gamma,\qquad
            K_Y:=\frac2m+\frac1{\gamma^2}.
        \)
        \\

        \(c_1(a,p),c_2(a,p)\)
        & \(\cref{lem:hypercoercivity}\)
        & \(a,p,m,\gamma\)
        & Norm-equivalence constants for \(\phi_{a,p}\); see
        \(\cref{lem:hypercoercivity}\).
        \\

        \(c_1^{(2)},c_2^{(2)}\)
        & \(\cref{thm:exp-decay-centered-moments}\)
        & \(m,\gamma\)
        & Dedicated quadratic norm-equivalence constants for \(\psi_2\).
        \\

        \(\mu_{\mathrm{gap},p}\)
        & \(\cref{thm:exp-decay-centered-moments}\)
        & \(p,\alpha,\overline f,\underline f\)
        & \(\displaystyle
            2
            -
            (p-1)^{-\frac{p-2}{p}}
            \exp\!\left(
                \frac{p-1}{p}\alpha(\overline f-\underline f)
            \right).
        \)
        \\

        \(\Lambda_p\)
        & \(\cref{thm:exp-decay-centered-moments}\)
        & \makecell{\(p,m,\chi_S,\)\\ \(\mu_{\mathrm{gap},p},C_{\rm lem}\)}
        & \(\displaystyle
            \frac{2(p-2)\chi_S}{m^2}
            \left(
                \frac{\mu_{\mathrm{gap},p}m^2}{64\chi_S}
            \right)^{-\frac{2}{p-2}}
            \left[
                1+\left(\frac{C_{\rm lem}}2\right)^{\frac{2}{p-2}}
            \right].
        \)
        \\

        \(K_\sigma\)
        & \(\cref{thm:exp-decay-centered-moments}\)
        & \(\sigma,m,\tau(S),C_{\rm lem}\)
        & \(\displaystyle
            K_\sigma
            :=
            \frac{2\sigma^2\tau(S)(1+C_{\rm lem})}{m^2}.
        \)
        \\

        \(v_1,v_3\)
        & \(\cref{thm:exp-decay-centered-moments}\)
        & \(m,\gamma,K_\sigma\)
        & \(\displaystyle
            v_1:=\frac1{m\gamma}+\frac3{\gamma^3}-K_\sigma,
            \qquad
            v_3:=\frac{2\gamma}{m}-\frac3\gamma.
        \)
        \\

        \(\edecay{p}\)
        & \(\cref{thm:exp-decay-centered-moments}\)
        & {\tiny\makecell{
            \(p,m,\gamma,\sigma,\alpha,\overline f,\underline f,\)\\
            \(\tau(S), C_{\mathrm{lem}}, c_2^{(2)}\)
        }}
        & \(\displaystyle
            \edecay{p}:=\frac{p\mu_{\mathrm{gap},p}}{4\gamma}
            \quad (p>2),
            \qquad
            \edecay{2}:=\frac{\min\{v_1,v_3\}}{c_2^{(2)}}.
        \)
        \\

        \(C_{\mathcal M}\)
        & \(\cref{lem:stability_est}\)
        & \(\alpha,L_f,\overline f,\underline f\)
        & \(\displaystyle
            C_{\mathcal M}
            :=
            2\alpha L_f\exp\!\left(2\alpha(\overline f-\underline f)\right).
        \)
        \\

        \(\Lambda_{\rm noise}\)
        & \makecell{\(\cref{thm:uit-poc-kinetic}\)\\
        \(\cref{thm:stab_particle}\)}
        & \(\sigma,m,\tau(S)\)
        & \(\displaystyle
            \Lambda_{\rm noise}:=\frac{\sigma^2\tau(S)}{m^2}.
        \)
        \\

        \(C_{\rm norm}\)
        & \makecell{\(\cref{thm:uit-poc-kinetic}\)\\
        \(\cref{thm:stab_particle}\)}
        & \(c_{\mathcal E,1}\), \(m,\gamma\)
        & \(\displaystyle
            C_{\rm norm}
            :=
            \sqrt{
                \frac{2}{c_{\mathcal E,1}}
                \max\!\left(\frac4{m^2},\frac1{\gamma^2}\right)
            }.
        \)
        \\
        \bottomrule
    \end{tabularx}
    \caption{Structural constants, decay rates, and coupling coefficients}
    \label{tab:structural-constants}
\end{table}

\FloatBarrier
\subsection{High-Friction Scaling and Choice of Weights}
\label{sec:high-friction-scaling}

The admissible parameter regime is a high-friction regime. In this regime,
the centered-moment decay rates appearing in the proof scale like
\[
    \edecay{p} = O(\gamma^{-1})
\]
in unrescaled time. To be more specific, for \(p>2\),
\[
    \edecay{p}=\frac{p\mu_{\mathrm{gap},p}}{4\gamma}.
\]
The same scaling applies to the rate \(\edecay{2}\), and to the
auxiliary concentration rate \(\kappa\). This scaling enters the final uniform-in-time constant through the Gr\"onwall
multiplier
\[
    \mathcal I_\beta
    :=
    \int_0^\infty \beta(t)\,\d t .
\]
The coefficients collected in \(\beta(t)\) are built from exponentially
decaying centered-moment bounds, hence contain terms of the form
\[
    C e^{-c t/\gamma}.
\]
Consequently,
\[
    \int_0^\infty C e^{-c t/\gamma}\,\d t
    =
    C\frac{\gamma}{c}.
\]
Thus the final constant may grow exponentially in \(\gamma\) through the factor
\(e^{\mathcal I_\beta}\). This growth comes from working in the original
unrescaled time variable and from the slow \(O(\gamma^{-1})\) decay of the
spatial consensus estimates coming from $\edecay{p}$.

The choice of weights in the Lyapunov and coupling energies is designed not to
introduce an additional loss through norm equivalence. To see the issue,
consider a quadratic energy of the form
\[
    a|x|^2+|v|^2+c\langle x,v\rangle .
\]
The choice of the cross coefficient \(c = m/\gamma\) is made to keep the quadratic
energy well-scaled in the high-friction regime. In the energy expansion, the
friction term acting on the mixed part contributes
\[
    -\frac{\gamma}{m}\cdot \frac m\gamma \langle x,v\rangle
    =
    -\langle x,v\rangle .
\]
Thus the remaining position--velocity cross coefficient is \(O(1)\), and in
the coupling calculation it takes the form
\[
    2a-1-\frac2m.
\]
This is cancelled by the \(\gamma\)-independent choice
\[
    a=\frac12+\frac1m.
\]

Had we used a mixed coefficient of order one, the friction contribution would
be of order \(\gamma\), forcing \(a=O(\gamma)\) in order to cancel the cross
term. The upper norm-equivalence constant would then also grow with
\(\gamma\). The scaling \(m/\gamma\) avoids this loss and keeps the quadratic
weights \(O(1)\).

The center-of-mass correction
\[
    -\frac1m|\overline{\delta X}_t|^2
\]
has the same purpose at the level of the shifted coupling energy. With
\[
    \mathcal E_t
    =
    \frac1J\sum_{j=1}^J
    \phi_{a,2}(\delta X_t^j,\delta V_t^j)
    -
    \frac1m|\overline{\delta X}_t|^2,
    \qquad
    a=\frac12+\frac1m,
\]
the fluctuation and center-of-mass cross terms have the same coefficient, and
both vanish. Equivalently,
\[
    \mathcal E_t
    =
    \frac1J\sum_{j=1}^J
    \phi_{a,2}(\widetilde{\delta X}_t^j,\widetilde{\delta V}_t^j)
    +
    \phi_{a-1/m,2}(\overline{\delta X}_t,\overline{\delta V}_t),
\]
so the center-of-mass block has spatial weight \(a-1/m=1/2\), while the
fluctuation block has spatial weight \(a\). Both remain independent of
\(\gamma\).
\newpage


\section{Main Results}\label{sec: main-results}

\subsection{Uniform-in-Time Propagation of Chaos}
\begin{assumption}[Propagation-of-chaos conditions]
\label{assump:poc}
Let \(2r\ge 8\). Assume that \cref{assump:admissibility,assump:centered-decay}
hold for \(p=2\), \(p=8\), and \(p=2r\), and denote the corresponding decay rates by
\(\edecay{2}\), \(\edecay{8}\), and \(\edecay{2r}\). Set $\kappa:=\frac{\edecay{8}}8\,.$ In addition, for these choice of $m, \gamma, \sigma$ above, assume:

\begin{enumerate}[label=(\roman*)]
    \item \textbf{Coupling-energy coercivity.}
    \[
        \gamma>\frac{m}{\sqrt2}.
    \]

    \item \textbf{Concentration gaps.}
    We assume
    \begin{equation}\label{eq:poc-concentration-gap}
        \min\left\{
            \frac1{m\gamma}
            +
            \frac3{\gamma^3}
            -
            \frac{2\sigma^2\tau(S)(1+C_{\mathrm{lem}})}{m^2},
            \frac{2\gamma}{m}
            -
            \frac3\gamma
        \right\}
        >
        c_2^{(2)}
        \frac{\mu_{\mathrm{gap},8}}{4\gamma},
    \end{equation}
    and
    \begin{equation}\label{eq:poc-high-moment-gap}
        \mu_{\mathrm{gap},2r}
        >
        \frac12\mu_{\mathrm{gap},8}
    \end{equation}
    so that the concentration gaps satisfy
    \[
        \edecay{2}>\kappa,
        \qquad
        \edecay{2r}>r\kappa.
    \]
        
    \item \textbf{Synchronous-coupling noise absorption.}
    \begin{equation}\label{eq:stab_small_noise}
        \sigma^2
        \le
        \frac{m^2}{\gamma\tau(S)}\,.
    \end{equation}

    \item \textbf{Noise absorption in the concentration estimate.}
    \begin{equation}\label{eq:main-thm-concentration-small-noise}
        \sigma^2
        \le
        \frac{m^2 c_1^{(2)}(\edecay{2}-\kappa)}
        {2\tau(S)(1+C_{\mathrm{lem}})}.
    \end{equation}
\end{enumerate}
\end{assumption}

\begin{remark}[Compatibility of the concentration gaps]
\label{rem:concentration-gap-compatibility}
The admissibility tradeoff between inverse temperature and moment order was described
in \cref{rem:alpha-moment-admissibility}. Here we record the additional gap
comparisons in \cref{assump:poc} and
\cref{assump:stab}. The concentration gap assumptions in \cref{assump:poc} and \cref{assump:stab} are
written in terms of the system parameters rather than only in terms of the decay
rates. This is useful because the two gaps have different origins. First, since
\[
    \kappa=\frac{\edecay{8}}8
    =
    \frac{\mu_{\mathrm{gap},8}}{4\gamma},
\]
the condition \(\edecay{2}>\kappa\) is
\[
    \min\left\{
        \frac1{m\gamma}
        +
        \frac3{\gamma^3}
        -
        \frac{2\sigma^2\tau(S)(1+C_{\mathrm{lem}})}{m^2},
        \frac{2\gamma}{m}
        -
        \frac3\gamma
    \right\}
    >
    c_2^{(2)}
    \frac{\mu_{\mathrm{gap},8}}{4\gamma}.
\]
This is the quadratic concentration gap. It can be ensured by choosing the quadratic
noise level small enough, provided the corresponding deterministic gap is positive. Second, for \(p>2\),
\[
    \edecay{p}
    =
    \frac{p\mu_{\mathrm{gap},p}}{4\gamma}.
\]
Therefore, in \cref{assump:poc},
\[
    \edecay{2r}>r\kappa
    \qquad\Longleftrightarrow\qquad
    \mu_{\mathrm{gap},2r}>\frac12\mu_{\mathrm{gap},8},
\]
and in \cref{assump:stab}, since \(p_\ast=8\max\{1,q\}\),
\[
    \edecay{p_\ast}>4\max\{1,q\}\kappa
    \qquad\Longleftrightarrow\qquad
    \mu_{\mathrm{gap},p_\ast}>\frac12\mu_{\mathrm{gap},8}.
\]
These high-moment gaps are comparisons between admissibility gaps at
different moment orders; they are not consequences of taking \(\sigma\) small.

The regime is nonempty under the usual high-moment choice. Write
\[
    \Delta f:=\overline f-\underline f,
    \qquad
    \mu_{\mathrm{gap},p}
    =
    2-(p-1)^{-\frac{p-2}{p}}
    \exp\!\left(\frac{p-1}{p}\alpha\Delta f\right).
\]
For fixed \(\alpha\Delta f\), we have
\[
    \mu_{\mathrm{gap},p}\longrightarrow 2
    \qquad\text{as }p\to\infty.
\]
Thus, once the base admissibility condition \(\mu_{\mathrm{gap},8}>0\) holds, one can
choose the higher moment order large enough so that
\[
    \mu_{\mathrm{gap},p}>\frac12\mu_{\mathrm{gap},8}.
\]
In particular, the comparisons
\[
    \mu_{\mathrm{gap},2r}>\frac12\mu_{\mathrm{gap},8},
    \qquad
    \mu_{\mathrm{gap},p_\ast}>\frac12\mu_{\mathrm{gap},8}
\]
hold whenever the corresponding high moment orders are chosen sufficiently large,
subject to the remaining admissibility conditions. After fixing such moment orders
and deterministic parameters, the remaining restrictions involving \(\sigma\) are
upper bounds on the noise strength and can be satisfied by taking \(\sigma>0\) small
enough.
\end{remark}

\begin{theorem}[Uniform-in-time propagation of chaos]
\label{thm:uit-poc-kinetic}
Let the objective function \(f:\mathbb R^d\to\mathbb R\) satisfy
\cref{assumption:bounded} and \cref{assumption:lip}. Let \(2r\ge 8\), and let the
initial configurations
\[
    (\mathcal X_0^J,\mathcal V_0^J)
    =
    \{(X_0^j,V_0^j)\}_{j=1}^J
\]
be \(J\) i.i.d. samples from a probability measure
\[
    \overline\rho_0\in\mathcal P_{2r}(\mathbb R^{2d}).
\]
Consider the second-order CBO particle system \eqref{eq: k-cbo} and the associated
mean-field system \eqref{eq: k-cbo-mfl}, initialized with
\[
    \overline X_0^j=X_0^j,
    \qquad
    \overline V_0^j=V_0^j,
\]
and driven by the same Brownian motions. Assume \cref{assump:poc}.

Then there exists a finite constant \(C_{\mathrm{MFL}}>0\), independent of \(t\) and
\(J\), such that for every \(j\in[J]\),
\begin{equation}\label{eq:uit-poc-particle-conclusion}
    \sup_{t\ge0}
    \E\left[
        |X_t^j-\overline X_t^j|^2
        +
        |V_t^j-\overline V_t^j|^2
    \right]
    \le
    \frac{C_{\mathrm{MFL}}}{J}.
\end{equation}
Consequently,
\begin{equation}\label{eq:uit-poc-wasserstein-conclusion}
    \sup_{t\ge0}
    \E\left[
        \mathcal W_2^2(\mu_{X_t^J},\mu_{\overline X_t^J})
    \right]
    \le
    C_{\mathrm{MFL}}J^{-1}.
\end{equation}
\end{theorem}

\subsection{Almost Uniform-in-Time Stability}
\label{sub:uit_stability}

We next state an almost uniform-in-time stability estimate for the interacting particle system \eqref{eq: k-cbo}. 
The proof follows the same synchronous coupling and hypocoercivity argument as the propagation of chaos result, but is simpler because no empirical-to-mean-field sampling error appears.

\begin{assumption}[Stability conditions]
\label{assump:stab}
Let \(q>1/2\), and set \(p_\ast:=\max\{8,8q\}\). Assume that
\cref{assump:admissibility,assump:centered-decay} hold for \(p=2\), \(p=8\), and
\(p=p_\ast\), and denote the corresponding decay rates by
\(\edecay{2}\), \(\edecay{8}\), and \(\edecay{p_\ast}\). Set
\(\kappa:=\frac{\edecay{8}}8\,.\) In addition, for these choice of $m, \gamma, \sigma$ above, assume:

\begin{enumerate}[label=(\roman*)]
    \item \textbf{Coupling-energy coercivity.}
    \[
        \gamma>\frac{m}{\sqrt2}.
    \]

    \item \textbf{Concentration gaps.}
    We assume
    \begin{equation}\label{eq:stab-concentration-gap}
        \min\left\{
            \frac1{m\gamma}
            +
            \frac3{\gamma^3}
            -
            \frac{2\sigma^2\tau(S)(1+C_{\mathrm{lem}})}{m^2},
            \frac{2\gamma}{m}
            -
            \frac3\gamma
        \right\}
        >
        c_2^{(2)}
        \frac{\mu_{\mathrm{gap},8}}{4\gamma},
    \end{equation}
    and
    \begin{equation}\label{eq:stab-high-moment-gap}
        \mu_{\mathrm{gap},p_\ast}
        >
        \frac12\mu_{\mathrm{gap},8},
    \end{equation}
    so that the concentration gaps satisfy
    \[
        \edecay{2}>\kappa,
        \qquad
        \edecay{p_\ast}>4\max\{1,q\}\kappa.
    \]

    \item \textbf{Synchronous-coupling noise absorption.}
    \begin{equation}\label{eq:stab-small-noise-stability}
        \sigma^2
        \le
        \frac{m^2}{\gamma\tau(S)}\,.
    \end{equation}

    \item \textbf{Noise absorption in the concentration estimate.}
    \begin{equation}\label{eq:stab-concentration-small-noise}
        \sigma^2
        \le
        \frac{m^2 c_1^{(2)}(\edecay{2}-\kappa)}
        {2\tau(S)(1+C_{\mathrm{lem}})}.
    \end{equation}
\end{enumerate}
\end{assumption}

\begin{theorem}[Almost uniform-in-time stability]
\label{thm:stab_particle}
Let \(f\) satisfy \cref{assumption:bounded,assumption:lip}. 
Let \(q>1/2\), set \(p_\ast:=\max\{8,8q\}\), and let
\(\rho_0,\tilde{\rho}_0\in\mathcal P_{p_\ast}(\R^{2d})\).
Consider two copies \((X_t^j,V_t^j)_{j=1}^J\) and
\((\tilde X_t^j,\tilde V_t^j)_{j=1}^J\) of the particle system
\eqref{eq: k-cbo}, driven by the same Brownian motions \((W_t^j)_{j=1}^J\), with
i.i.d.\ initial data sampled respectively from \(\rho_0\) and \(\tilde\rho_0\).
Assume \cref{assump:stab}.

Then there exist constants \(C_{\mathrm{Stab},1},C_{\mathrm{Stab},2}>0\),
independent of \(J\) and \(t\), such that for all \(t\ge0\),
\begin{equation}\label{eq:stab-particle-conclusion}
    \E\!\left[
        \frac1J\sum_{j=1}^J
        \Bigl(
            |X_t^j-\tilde X_t^j|^2
            +
            |V_t^j-\tilde V_t^j|^2
        \Bigr)
    \right]
    \le
    C_{\mathrm{Stab},1}
    \E\!\left[
        \frac1J\sum_{j=1}^J
        \Bigl(
            |X_0^j-\tilde X_0^j|^2
            +
            |V_0^j-\tilde V_0^j|^2
        \Bigr)
    \right]
    +
    \frac{C_{\mathrm{Stab},2}}{J^q}.
\end{equation}
\end{theorem}

\begin{remark}\label{rmk:stability_better_rate}
    Contrary to the propagation of chaos estimate, the stability estimate compares two empirical swarms of the same size $J$.
    Therefore no empirical-to-mean-field sampling term appears, and the remainder comes only from the bad-set contribution in the weighted-mean stability estimate.
    By increasing the initial moment assumption from order $8$ to order $8q$, one obtains the arbitrarily improved rate $J^{-q}$.
\end{remark}

\section{Proofs of the Main Results}\label{sec: main-results-proof}

This section proves the two main results stated in \cref{sec: main-results}. The
proofs use four auxiliary ingredients: exponential decay of centered moments,
uniform-in-time raw moment bounds, concentration estimates for the empirical
Lyapunov functional, and stability/Monte Carlo estimates for the weighted mean.
These ingredients are stated and proved in \cref{sec: auxiliary-results}. We first
prove the uniform-in-time propagation-of-chaos estimate, where all four
ingredients are needed, and then prove the particle-to-particle stability
estimate, where the empirical-to-mean-field sampling error is absent.

\subsection{Proof of \cref{thm:uit-poc-kinetic}}

\begin{proof}
    Let the raw differences be $\delta X_t^j := X_t^j - \overline{X}_t^j$ and $\delta V_t^j := V_t^j - \overline{V}^j_t$. We define the empirical mean differences (the Center of Mass error) as $\overline{\delta X}_t := \frac{1}{J}\sum_{j=1}^J \delta X_t^j$ and $\overline{\delta V}_t := \frac{1}{J}\sum_{j=1}^J \delta V_t^j$. We also define the fluctuations around the CM error as $\widetilde{\delta X}_t^j := \delta X_t^j - \overline{\delta X}_t$ and $\widetilde{\delta V}_t^j := \delta V_t^j - \overline{\delta V}_t$. Note that $\sum_j \widetilde{\delta X}_t^j = \sum_j \widetilde{\delta V}_t^j = 0$.

    We define the standard Euclidean energy as $\widehat{\mathcal{E}}_t := \frac{1}{J} \sum_{j=1}^J ( |\delta X_t^j|^2 + |\delta V_t^j|^2 )$. 
    To exploit hypocoercivity while accounting for the shift-invariant nature of the model, we define the modified energy $\mathcal{E}_t$ using our universal functional $\phi_{a,2}$. Recall that for $p=2$, the polynomial form is
\[
\phi_{a,2}(x,v):=a|x|^2+\frac{m}{\gamma}\langle x,v\rangle+|v|^2,
\]
and that
\begin{equation}\label{eq:energy-functional-proof-recall}
    \mathcal{E}_t
    =
    \frac1J\sum_{j=1}^J \phi_{a,2}(\delta X_t^j,\delta V_t^j)
    -\frac1m |\overline{\delta X}_t|^2.
\end{equation}
Using the decomposition
\[
\delta X_t^j=\widetilde{\delta X}_t^j+\overline{\delta X}_t,
\qquad
\delta V_t^j=\widetilde{\delta V}_t^j+\overline{\delta V}_t,
\]
together with
\[
\frac1J\sum_{j=1}^J \widetilde{\delta X}_t^j=0,
\qquad
\frac1J\sum_{j=1}^J \widetilde{\delta V}_t^j=0,
\]
we indeed have
\[
\mathcal{E}_t
=
\frac1J\sum_{j=1}^J \phi_{a,2}(\widetilde{\delta X}_t^j,\widetilde{\delta V}_t^j)
+\phi_{a-\frac1m,2}(\overline{\delta X}_t,\overline{\delta V}_t).
\]
    
    This functional consists of two distinct quadratic forms: $\phi_{a,2}$ for the fluctuations and $\phi_{a - 1/m, 2}$ for the Center of Mass. To guarantee that $\mathcal{E}_t$ is  positive definite, both must be positive definite. By \cref{lem:hypercoercivity}, it is necessary and sufficient to enforce the stronger condition on the CM component: $a - \frac{1}{m} > \frac{m^2}{4\gamma^2}$. 
    
    By \cref{cor:coupling-energy-equivalence}, the energy satisfies
    \[
        c_{\mathcal E,1}\widehat{\mathcal E}_t
        \le
        \mathcal E_t
        \le
        c_{\mathcal E,2}\widehat{\mathcal E}_t.
    \]
    
    The dynamics of the difference process are:
    \begin{align}
        \d (\delta X_t^j) &= \delta V_t^j \d t \,, \\
        m \d (\delta V_t^j) &= -\gamma \delta V_t^j \d t - [ (X_t^j - \mathcal{M}_\alpha(\mu_{X_t^J})) - (\overline{X}_t^j - \mathcal{M}_\alpha(\overline{\rho}_t)) ] \d t + \sigma \Delta S^j \d W_t^j,
    \end{align}
    where $\Delta S^j := S(X_t^j - \mathcal{M}_\alpha(\mu_{X_t^J})) - S(\overline{X}_t^j - \mathcal{M}_\alpha(\overline{\rho}_t))$.

    \proofstep{Drift Decomposition and Pivoting}
    We decompose the deterministic force difference $F_{diff}^j$ by pivoting around the empirical measure of the mean-field particles, $\mu_{\overline{X}_t^J}$:
    \begin{align} 
        F_{diff}^j &= - \delta X_t^j + \underbrace{\left( \mathcal{M}_\alpha(\mu_{X_t^J}) - \mathcal{M}_\alpha(\mu_{\overline{X}_t^J}) \right)}_{:= \Delta_\alpha^{coup}} + \underbrace{\left( \mathcal{M}_\alpha(\mu_{\overline{X}_t^J}) - \mathcal{M}_\alpha(\overline{\rho}_t) \right)}_{:= \mathfrak{D}_t}.
    \end{align}
    To isolate the linear feedback acting on the CM, we pivot the coupling shift $\Delta_\alpha^{coup}$ around $\overline{\delta X}_t$:
    \begin{equation} \label{eq:pivot_rho}
        \Delta_\alpha^{coup} = \overline{\delta X}_t + \Delta_\alpha^{stab}(t),
    \end{equation}
    where $\Delta_\alpha^{stab}(t) = \mathcal{M}_\alpha(\mu_{X_t^J}) - \mathcal{M}_\alpha(\mu_{\overline{X}_t^J}) - \overline{\delta X}_t$ is the purely nonlinear stability error. Defining the total remainder $\mathcal{R}_t := \Delta_\alpha^{stab}(t) + \mathfrak{D}_t$, the deterministic driving force on $\delta V_t^j$ simplifies to:
    \begin{equation}
        F_{diff}^j = -\widetilde{\delta X}_t^j + \mathcal{R}_t.
    \end{equation}

    \proofstep{Energy Expansion and Cross-Term Cancellation}
    We now expand the derivative of the quadratic energy in full detail. We differentiate each term in $\phi_{a,2}$.

\medskip
\noindent
\textbf{Step 1: the $a|\delta X_t^j|^2$ term.}
Since $d(\delta X_t^j)=\delta V_t^j\,dt$,
\[
d\bigl(a|\delta X_t^j|^2\bigr)
=
2a\langle \delta X_t^j,\delta V_t^j\rangle\,dt.
\]

\medskip
\noindent
\textbf{Step 2: the mixed term $\frac{m}{\gamma}\langle \delta X_t^j,\delta V_t^j\rangle$.}
Because $d(\delta X_t^j)$ has bounded variation, there is no quadratic covariation term, and hence
\[
d\Bigl(\frac{m}{\gamma}\langle \delta X_t^j,\delta V_t^j\rangle\Bigr)
=
\frac{m}{\gamma}\langle d(\delta X_t^j),\delta V_t^j\rangle
+
\frac{m}{\gamma}\langle \delta X_t^j,d(\delta V_t^j)\rangle.
\]
Substituting the dynamics gives
\begin{align}
d\Bigl(\frac{m}{\gamma}\langle \delta X_t^j,\delta V_t^j\rangle\Bigr)
&=
\frac{m}{\gamma}|\delta V_t^j|^2\,dt
-\langle \delta X_t^j,\delta V_t^j\rangle\,dt
+\frac1\gamma \langle \delta X_t^j,F_{\mathrm{diff}}^j(t)\rangle\,dt
+\frac{\sigma}{\gamma}\langle \delta X_t^j,\Delta S_t^j\,dW_t^j\rangle.
\end{align}

\medskip
\noindent
\textbf{Step 3: the $|\delta V_t^j|^2$ term.}
By It\^o's formula,
\begin{align}
d|\delta V_t^j|^2
&=
2\langle \delta V_t^j,d(\delta V_t^j)\rangle
+d\langle \delta V^j\rangle_t \\
&=
-\frac{2\gamma}{m}|\delta V_t^j|^2\,dt
+\frac{2}{m}\langle \delta V_t^j,F_{\mathrm{diff}}^j(t)\rangle\,dt
+\frac{\sigma^2}{m^2}\|\Delta S_t^j\|_{HS}^2\,dt  +\frac{2\sigma}{m}\langle \delta V_t^j,\Delta S_t^j\,dW_t^j\rangle.
\end{align}

\medskip
Summing the three contributions, we obtain
\begin{align}
d\phi_{a,2}(\delta X_t^j,\delta V_t^j)
&=
\Bigl[
(2a-1)\langle \delta X_t^j,\delta V_t^j\rangle
-\Bigl(\frac{2\gamma}{m}-\frac{m}{\gamma}\Bigr)|\delta V_t^j|^2
\notag\\
&\qquad
+\frac1\gamma \langle \delta X_t^j,-\widetilde{\delta X}_t^j+\mathcal{R}_t\rangle
+\frac2m \langle \delta V_t^j,-\widetilde{\delta X}_t^j+\mathcal{R}_t\rangle
+\frac{\sigma^2}{m^2}\|\Delta S_t^j\|_{HS}^2
\Bigr]dt
+ dM_t^j,
\label{eq:phi-expansion-full}
\end{align}
where
\[
dM_t^j
=
\frac{\sigma}{\gamma}\langle \delta X_t^j,\Delta S_t^j\,dW_t^j\rangle
+
\frac{2\sigma}{m}\langle \delta V_t^j,\Delta S_t^j\,dW_t^j\rangle
\]
is a local martingale term.

Averaging \eqref{eq:phi-expansion-full} over $j=1,\dots,J$ and taking expectations removes the martingale part. Using that $\mathcal{R}_t$ is independent of $j$, together with
\[
\frac1J\sum_{j=1}^J \delta X_t^j=\overline{\delta X}_t,
\qquad
\frac1J\sum_{j=1}^J \delta V_t^j=\overline{\delta V}_t,
\]
and the identities
\[
\frac1J\sum_{j=1}^J \langle \delta X_t^j,\widetilde{\delta X}_t^j\rangle = \frac1J\sum_{j=1}^J |\widetilde{\delta X}_t^j|^2,
\qquad
\frac1J\sum_{j=1}^J \langle \delta V_t^j,\widetilde{\delta X}_t^j\rangle = \frac1J\sum_{j=1}^J \langle \widetilde{\delta V}_t^j,\widetilde{\delta X}_t^j\rangle,
\]
we find that the restoring force separates the cross-term coefficients:
\begin{align}
\frac{d}{dt}\E\Bigl[\frac1J\sum_{j=1}^J \phi_{a,2}(\delta X_t^j,\delta V_t^j)\Bigr]
&=
\Bigl(2a-1-\frac2m\Bigr) \E\Bigl[ \frac1J\sum_{j=1}^J \langle \widetilde{\delta X}_t^j,\widetilde{\delta V}_t^j\rangle \Bigr]
+ (2a-1) \E \langle \overline{\delta X}_t,\overline{\delta V}_t\rangle
\notag\\
&\quad
-\frac1\gamma \E\Bigl[\frac1J\sum_{j=1}^J |\widetilde{\delta X}_t^j|^2\Bigr]
-\Bigl(\frac{2\gamma}{m}-\frac{m}{\gamma}\Bigr) \E\Bigl[\frac1J\sum_{j=1}^J |\delta V_t^j|^2\Bigr]
\notag\\
&\quad
+\E\Bigl\langle \frac1\gamma \overline{\delta X}_t+\frac2m\overline{\delta V}_t,\,\mathcal{R}_t\Bigr\rangle
+\frac{\sigma^2}{m^2J}\E\sum_{j=1}^J \|\Delta S_t^j\|_{HS}^2.
\label{eq:pre-cancel-energy}
\end{align}

It is now evident why the center-of-mass correction term in \eqref{eq:energy-functional-proof-recall} is necessary. Since $\frac{d}{dt}\overline{\delta X}_t=\overline{\delta V}_t$, differentiating the correction yields:
\begin{equation}
\frac{d}{dt}\E\Bigl[-\frac1m|\overline{\delta X}_t|^2\Bigr]
=
-\frac2m \E\langle \overline{\delta X}_t,\overline{\delta V}_t\rangle.
\label{eq:cm-correction-derivative}
\end{equation}
Combining \eqref{eq:pre-cancel-energy} and \eqref{eq:cm-correction-derivative}, the correction term perfectly aligns the CM cross-term coefficient with the fluctuation cross-term coefficient:
\begin{align}
\frac{d}{dt}\E[\mathcal E_t]
&=
\Bigl(2a-1-\frac2m\Bigr)
\E\Bigl[
\frac1J\sum_{j=1}^J \langle \widetilde{\delta X}_t^j,\widetilde{\delta V}_t^j\rangle
+
\langle \overline{\delta X}_t,\overline{\delta V}_t\rangle
\Bigr]
\notag\\
&\quad
-\frac1\gamma
\E\Bigl[\frac1J\sum_{j=1}^J |\widetilde{\delta X}_t^j|^2\Bigr]
-\Bigl(\frac{2\gamma}{m}-\frac{m}{\gamma}\Bigr)
\E\Bigl[\frac1J\sum_{j=1}^J |\delta V_t^j|^2\Bigr]
\notag\\
&\quad
+\E\Bigl\langle \frac1\gamma \overline{\delta X}_t+\frac2m\overline{\delta V}_t,\,\mathcal{R}_t\Bigr\rangle
+\frac{\sigma^2}{m^2J}\E\sum_{j=1}^J \|\Delta S_t^j\|_{HS}^2.
\label{eq:energy-before-choice-of-a}
\end{align}

We now choose
\[
a:=\frac12+\frac1m.
\]
Then both cross-term coefficients vanish simultaneously:
\[
2a-1-\frac2m=0.
\]
Therefore, using the decomposition $\mathcal{R}_t=\Delta_{\alpha}^{\mathrm{stab}}(t)+\mathfrak{D}_t$, we arrive at
\[
\frac{d}{dt}\E[\mathcal E_t]
= 
-\E[\mathcal D_{\mathrm{lin}}]
+\mathcal A_1(t)+\mathcal A_2(t)+\mathcal A_3(t),
\]
where
\[
\mathcal D_{\mathrm{lin}}
:=
\frac1J\sum_{j=1}^J
\left[
\frac1\gamma |\widetilde{\delta X}_t^j|^2
+
\left(\frac{2\gamma}{m}-\frac{m}{\gamma}\right)|\delta V_t^j|^2
\right],
\]
and the terms corresponding to Coupling Stability ($\mathcal{A}_1$), Sampling Error ($\mathcal{A}_2$), and Noise Trace ($\mathcal{A}_3$) simplify to:
    \begin{align}
        \mathcal{A}_1(t) &= \E \left\langle \frac{2}{m}\overline{\delta V}_t + \frac{1}{\gamma}\overline{\delta X}_t, \Delta_\alpha^{stab}(t) \right\rangle \,,  \\
        \mathcal{A}_2(t) &= \E \left\langle \frac{2}{m}\overline{\delta V}_t + \frac{1}{\gamma}\overline{\delta X}_t, \mathfrak{D}_t \right\rangle \,, \\
        \mathcal{A}_3(t) &= \frac{\sigma^2 }{m^2 J}\E \sum_{j=1}^J \left|S(X_t^j - \mathcal{M}_\alpha(\mu_{X_t^J})) - S(\overline{X}_t^j - \mathcal{M}_{\alpha}(\overline{\rho}_t)) \right|_{\text{HS}}^2 \,.
    \end{align}

    As noted previously, our choices yield $\frac{2\gamma}{m} - \frac{m}{\gamma} > 0$ under the same condition $\gamma > m/\sqrt{2}$ used to guarantee positive definiteness.

    \proofstep{Bounding the Terms}
    To bound the terms $\mathcal{A}_1, \mathcal{A}_2,$ and $\mathcal{A}_3$, we first define the norm equivalence constant
    \[
        C_{norm}
        :=
        \sqrt{
            \frac{2}{c_{\mathcal E,1}}
            \max\left(\frac4{m^2},\frac1{\gamma^2}\right)
        }.
    \]
    By the norm equivalence \(c_{\mathcal E,1}\widehat{\mathcal E}_t\le \mathcal E_t\), the linear combination of the Center of Mass errors satisfies:
    \begin{equation}
        \left| \frac{2}{m}\overline{\delta V}_t + \frac{1}{\gamma}\overline{\delta X}_t \right| \le C_{norm} \sqrt{\mathcal{E}_t}.
    \end{equation}
    
    \textbf{1. Coupling Stability ($\mathcal{A}_1$):}
    Because the modified energy completely eliminated the Center of Mass cross-term $\langle \overline{\delta X}_t, \overline{\delta V}_t \rangle$ during the energy expansion, there is no longer any cross-term to estimate. The coupling stability term depends solely on the nonlinear shift error.
    Applying the Cauchy-Schwarz inequality and the norm equivalence bound:
    \begin{align}
        \mathcal{A}_1(t) &= \E \left\langle \frac{2}{m}\overline{\delta V}_t + \frac{1}{\gamma}\overline{\delta X}_t, \Delta_\alpha^{stab}(t) \right\rangle \nonumber \\
        &\le \E \left[ \left| \frac{2}{m}\overline{\delta V}_t + \frac{1}{\gamma}\overline{\delta X}_t \right| |\Delta_\alpha^{stab}(t)| \right] \nonumber \\
        &\le C_{norm} \E \left[ \sqrt{\mathcal{E}_t} |\Delta_\alpha^{stab}(t)| \right] \le C_{norm} \sqrt{\E[\mathcal{E}_t]} \sqrt{\E[|\Delta_\alpha^{stab}(t)|^2]}.
    \end{align}
    
    \textbf{2. Sampling Error ($\mathcal{A}_2$):}
    Using the Cauchy-Schwarz inequality, \cref{lem:mc_est}, Young's inequality with weight $\zeta \in (0, \edecay{2})$, and \cref{thm:mfl-decay-centered-moment}:
    \begin{align}
        \mathcal{A}_2(t) \le C_{norm} \sqrt{\E[\mathcal{E}_t]} \sqrt{\E[|\mathfrak{D}_t|^2]} \le \frac{C_{norm}}{2} \e^{-\zeta t} \E[\mathcal{E}_t] + \frac{C_{norm} \cwm{2}}{2J} \e^{-(\edecay{2} - \zeta)t} \momentp{2}{\mfldis_0} \,.
    \end{align}

    \textbf{3. Noise Trace ($\mathcal{A}_3$):}
    Using Lipschitz continuity of $S$ and substituting the pivot $F_{diff}^j = -\widetilde{\delta X}_t^j + \Delta_\alpha^{stab} + \mathfrak{D}_t$:
    \begin{align}
        \mathcal{A}_3(t) \le \Lambda_{noise} \E \left[ \frac{1}{J}\sum_{j=1}^J |\widetilde{\delta X}_t^j|^2 + |\Delta_\alpha^{stab} + \mathfrak{D}_t|^2 \right], \quad \text{where } \Lambda_{noise} := \frac{\sigma^2 \chi_S}{m^2}.
    \end{align}
    To close the energy estimate, we require the noise trace generated by the spatial fluctuations to be  absorbed by the linear fluctuation dissipation $-\frac{1}{\gamma}\E[\frac{1}{J}\sum |\widetilde{\delta X}_t^j|^2]$ from $\mathcal{D}_{lin}$. This requires $\Lambda_{noise} \le \frac{1}{\gamma}$, which holds by the assumption $\sigma^2 \le \frac{m^2}{\gamma \chi_S}$ from \eqref{eq:stab_small_noise}. Provided this holds, the sum $-\E[\mathcal{D}_{lin}^{fluc}] + \mathcal{A}_3^{fluc} \le 0$ can be dropped. Applying the elementary inequality $|x+y|^2 \le 2|x|^2 + 2|y|^2$ to the macroscopic shifts, the remaining noise feedback is bounded by:
    \begin{equation}
        \mathcal{A}_3^{rem}(t) \le 2 \Lambda_{noise} \left( \E[|\Delta_\alpha^{stab}|^2] + \E[|\mathfrak{D}_t|^2] \right).
    \end{equation}

    \proofstep{Summary of the Differential Inequality}
    We drop the  negative linear velocity dissipation from $\mathcal{D}_{lin}$ (which is guaranteed since $\frac{2\gamma}{m} - \frac{m}{\gamma} > 0$). Substituting the bounds for $\mathcal{A}_1$, $\mathcal{A}_2$, and $\mathcal{A}_3^{rem}$ into the drift equation, and substituting the Monte Carlo bound for $\E[|\mathfrak{D}_t|^2]$, we obtain:
    \begin{align} 
        \frac{\d}{\d t} \E[\mathcal{E}_t] &\le C_{norm} \sqrt{\E[\mathcal{E}_t]} \sqrt{\E[|\Delta_\alpha^{stab}|^2]} + 2 \Lambda_{noise} \E[|\Delta_\alpha^{stab}|^2] \nonumber \\
        &\quad + \frac{C_{norm}}{2} \e^{-\zeta t} \E[\mathcal{E}_t] \nonumber \\
        &\quad + \frac{1}{J} \left( \frac{C_{norm} \cwm{2}}{2} \momentp{2}{\mfldis_0} \e^{-(\edecay{2} - \zeta)t} + 2 \Lambda_{noise} \cwm{2} \momentp{2}{\mfldis_0} \e^{-\edecay{2}t} \right). \label{eq:diff_ineq_final}
    \end{align}

    \proofstep{Substituting the Stability Error}
    As derived in \eqref{eq:technical-inequality-kinetic} below, the nonlinear stability error satisfies:
    \begin{equation}
        \E[|\Delta_\alpha^{stab}(t)|^2] \le C_Q \e^{-\kappa t} \E[\mathcal{E}_t] + \frac{C_{\Delta}}{J} \e^{-\frac{\edecay{8}}{4} t},
    \end{equation}
    where \(C_Q:=4C_{\mathcal M}^2K_0c_{\mathcal E,1}^{-1}\) and \(C_{\Delta}\) is defined in \eqref{eq:C_Q_kinetic}.
    Using $\sqrt{A+B} \le \sqrt{A} + \sqrt{B}$ and applying Young's inequality ($xy \le \frac{1}{2}x^2 + \frac{1}{2}y^2$) to the mixed product in \eqref{eq:diff_ineq_final}:
    \begin{align}
        C_{norm} \sqrt{\E[\mathcal{E}_t]} \sqrt{\E[|\Delta_\alpha^{stab}|^2]} 
        &\le \left( C_{norm} \sqrt{C_Q} \e^{-\frac{\kappa}{2} t} + \frac{C_{norm} \sqrt{C_{\Delta} } }{2} \e^{-\frac{\edecay{8}}{8} t} \right) \E[\mathcal{E}_t] + \frac{ C_{norm} \sqrt{C_{\Delta} }}{2 J} \e^{-\frac{\edecay{8}}{8} t}.
    \end{align}
    Similarly, the noise feedback is bounded by:
    \begin{equation}
        2 \Lambda_{noise} \E[|\Delta_\alpha^{stab}|^2] \le 2 \Lambda_{noise} C_Q \e^{-\kappa t} \E[\mathcal{E}_t] + \frac{2 \Lambda_{noise} C_{\Delta} }{J} \e^{-\frac{\edecay{8}}{4} t}.
    \end{equation}

    \proofstep{Gr\"onwall Estimate and Final Constant $C_{\text{MFL}}$}
    Collecting all terms, the system resolves to the Gr\"onwall form:
    \begin{equation}
        \frac{\d}{\d t} \E[\mathcal{E}_t] \le \beta(t) \E[\mathcal{E}_t] + \frac{1}{J} \alpha(t).
    \end{equation}
    Here, the integrating multiplier $\beta(t)$ groups the coefficients of $\E[\mathcal{E}_t]$:
    \begin{align}
        \beta(t) &= C_{norm} \sqrt{C_Q} \e^{-\frac{\kappa}{2} t} + \frac{C_{norm} \sqrt{C_{\Delta} }} {2} \e^{-\frac{\edecay{8}}{8} t} + \frac{C_{norm}}{2} \e^{-\zeta t} + 2 \Lambda_{noise} C_Q \e^{-\kappa t},
    \end{align}
    and the term $\alpha(t)$ groups the $\mathcal{O}(J^{-1})$ contributions:
    \begin{align}
        \alpha(t) &= \frac{C_{norm} \sqrt{C_{\Delta }}}{2} \e^{-\frac{\edecay{8}}{8} t} + \frac{C_{norm} \cwm{2} \momentp{2}{\mfldis_0}}{2} \e^{-(\edecay{2} - \zeta)t}  + 2 \Lambda_{noise} C_{\Delta } \e^{-\frac{\edecay{8}}{4} t} + 2 \Lambda_{noise} \cwm{2} \momentp{2}{\mfldis_0} \e^{-\edecay{2}t}.
    \end{align}
    
    Since all exponent decay rates ($\kappa, \edecay{8}, \edecay{2}, \zeta$) are  positive (by choosing $\zeta < \edecay{2}$), both functions are integrable on $[0, \infty)$. We define the bounded integrating factor exponent $\mathcal{I}_\beta := \int_0^\infty \beta(t) \d t$. 
    By direct integration, the total integral $C_\alpha := \int_0^\infty \alpha(s) \e^{\int_s^\infty \beta(\tau)\d \tau} \d s$ is bounded  by:
    \begin{align} 
        C_\alpha \le \e^{\mathcal{I}_\beta} \Bigg[ &\frac{4 C_{norm} \sqrt{C_{\Delta}}}{\edecay{8}} 
        + \frac{C_{norm} \cwm{2} \momentp{2}{\mfldis_0}}{2(\edecay{2} - \zeta)} + \frac{8 \Lambda_{noise} C_{\Delta}}{\edecay{8}} 
        + \frac{2 \Lambda_{noise} \cwm{2} \momentp{2}{\mfldis_0}}{\edecay{2}} \Bigg] < \infty. \label{eq:C_alpha}
    \end{align}
    
    Applying the integral form of Gr\"onwall's inequality yields:
    \begin{equation}
        \E[\mathcal{E}_t] \le \E[\mathcal{E}_0] \e^{\int_0^t \beta(\tau) \d\tau} + \frac{1}{J} \int_0^t \alpha(s) \e^{\int_s^t \beta(\tau) \d\tau} \d s.
    \end{equation}
    By the definition of the synchronous coupling, the mean-field particles are initialized identically to the particle system ($X_0^j = \overline{X}_0^j$ and $V_0^j = \overline{V}_0^j$), which guarantees initial matching: $\E[\mathcal{E}_0] = 0$. Thus, we conclude:
    \begin{align}
        \E[\mathcal{E}_t] \le \frac{1}{J} \int_0^\infty \alpha(s) \e^{\int_s^\infty \beta(\tau)\d \tau} \d s \le \frac{C_\alpha}{J}.
    \end{align}
    Using the norm equivalence $\E[|\delta X_t^j|^2 + |\delta V_t^j|^2] \le c_{\mathcal{E},1}^{-1} \E[\mathcal{E}_t]$, we obtain the claimed phase-space bound:
    \begin{equation}
        \sup_{t \ge 0} \E \left[ |X_t^j - \overline{X}_t^j|^2 + |V_t^j - \overline{V}_t^j|^2 \right] \le \frac{C_\alpha}{c_{\mathcal{E},1}J}.
    \end{equation}
    Finally, by the definition of the Wasserstein distance as an infimum over all valid couplings, this synchronous empirical coupling automatically bounds the spatial Wasserstein distance between the two empirical measures:
\begin{equation}
    \sup_{t \ge 0} \E \left[\mathcal{W}_2^2(\mu_{X_t^J}, \mu_{\overline{X}_t^J})\right]
    \le
    \sup_{t \ge 0} \E \left[ \frac{1}{J}\sum_{j=1}^J |X_t^j - \overline{X}_t^j|^2 \right]
    \le \frac{C_\alpha}{c_{\mathcal{E},1}J},
\end{equation}
    completing the proof with $C_{\text{MFL}} = c_{\mathcal{E},1}^{-1} C_\alpha$.
\end{proof}

\paragraph{Control of Stability Error $\E[|\Delta_\alpha^{stab}(t)|^2]$}
    We bound the nonlinear stability error using the stability estimate (\cref{lem:stability_est}) and the concentration inequalities derived in \cref{sec:concentration}.
    First, by \cref{lem:stability_est}, we have
    \begin{align}
        |\Delta_\alpha^{stab}(t)|^2 &\le 
        C_{\mathcal M}^2 \left( \sqrt{\mathfrak{M}_2(\mu_{X^J_t})} + \sqrt{\mathfrak{M}_2(\mu_{\overline{X}^J_t})} \right)^2 \mathcal{W}_2^2 (\mu_{X^J_t}, \mu_{\overline{X}^J_t}) \nonumber \\
        &\le 2 C_{\mathcal M}^2 \left( \mathfrak{M}_2(\mu_{X^J_t}) + \mathfrak{M}_2(\mu_{\overline{X}^J_t}) \right)\mathcal{W}_2^2 (\mu_{X^J_t}, \mu_{\overline{X}^J_t}).
    \end{align}
    The centered moments decay exponentially with high probability, but may be large on rare events. We use
    \[
        \kappa:=\frac{\edecay{8}}8,
    \]
    which satisfies \(\kappa<\edecay{2}\) and \(r\kappa<\edecay{2r}\) by
    \cref{assump:poc}. Set the threshold
    \[
        K_0
        :=
        \frac{
            \max\{\mathcal L_2(0),\,\E[\overline L_2^J(0)]\}+1
        }{c_1^{(2)}},
    \]
    and define the "Bad Sets" for the particle system and the mean-field system respectively:
    \begin{align}
        \Omega_{bad, X} &:= \left\{ \sup_{t \ge 0} \e^{\kappa t} \mathfrak{M}_2(\mu_{X^J_t}) \ge K_0 \right\}, \\
        \Omega_{bad, \overline{X}} &:= \left\{ \sup_{t \ge 0} \e^{\kappa t} \mathfrak{M}_2(\mu_{\overline{X}^J_t}) \ge K_0 \right\}.
    \end{align}
    Let $\Omega^* := \Omega_{bad, X} \cup \Omega_{bad, \overline{X}}$. By the concentration inequalities (\cref{lem:concentration-ineq} and \cref{lem:concentration-mf}), applied with deviation threshold \(A=1\) and using \(\kappa<\edecay{2}\), \(r\kappa<\edecay{2r}\), the probability of this set is bounded. Using \(\cref{cor:centered-lyapunov-moment-equivalence}\) at order \(2r\), we have:
    \begin{align}\label{eq: prob-omega-bound}
        \proba(\Omega^*) \le \proba(\Omega_{bad, X}) + \proba(\Omega_{bad, \overline{X}})
        \le
        \left(
            C_{Bad}\mathcal L_{2r}(0)
            +
            C_{Bad}^{MF}\overline{\mathcal L}_{2r}(0)
        \right)
        J^{-r/2}
        \le
        (C_{Bad}+C_{Bad}^{MF})C_{2r,2}\,
        J^{-r/2}\mathfrak M_{2r}(\mfldis_0).
    \end{align}
    
    We decompose the expectation of the error term:
    \begin{equation}
        \E[|\Delta_\alpha^{stab}(t)|^2] \le 2 C_{\mathcal M}^2 \underbrace{\E \left[ \mathbf{1}_{(\Omega^*)^c} \left(\mathfrak{M}_2(\mu_{X^J_t}) + \mathfrak{M}_2(\mu_{\overline{X}^J_t})\right) \mathcal{W}_2^2 \right]}_{\text{Good Set}} 
        + 2 C_{\mathcal M}^2 \underbrace{\E \left[ \mathbf{1}_{\Omega^*} \left(\mathfrak{M}_2(\mu_{X^J_t}) + \mathfrak{M}_2(\mu_{\overline{X}^J_t})\right) \mathcal{W}_2^2 \right]}_{\text{Bad Set}}.
    \end{equation}
    
    \textbf{1. On the Good Set $(\Omega^*)^c$:}
    By definition, on the complement of the bad sets, the moments satisfy $\mathfrak{M}_2(\mu_{X^J_t}) \le K_0 \e^{-\kappa t}$ and $\mathfrak{M}_2(\mu_{\overline{X}^J_t}) \le K_0 \e^{-\kappa t}$ for all $t \ge 0$. Furthermore, the empirical Wasserstein distance is bounded by the Euclidean energy: $\mathcal{W}_2^2(\mu_{X^J_t}, \mu_{\overline{X}^J_t}) \le \frac{1}{J}\sum_{j=1}^J |X_t^j - \overline{X}_t^j|^2 \le c_{\mathcal E,1}^{-1} \mathcal{E}_t$. Thus:
    \begin{align}
        \E \left[ \mathbf{1}_{(\Omega^*)^c} \left(\mathfrak{M}_2(\mu_{X^J_t}) + \mathfrak{M}_2(\mu_{\overline{X}^J_t})\right) \mathcal{W}_2^2 (\mu_{X^J_t}, \mu_{\overline{X}^J_t})  \right] 
        & \le 2 K_0 \e^{-\kappa t} \E[\mathcal{W}_2^2 (\mu_{X^J_t}, \mu_{\overline{X}^J_t}) ] \nonumber \\
        & \leq 2 K_0 c_{\mathcal E,1}^{-1} \e^{-\kappa t} \E[\mathcal{E}_t ].
    \end{align}

    \textbf{2. On the Bad Set $\Omega^*$:}
    We use the Cauchy-Schwarz inequality to separate the probability of the bad set from the moments of the error: 
    \begin{align}
        \E \left[ \mathbf{1}_{\Omega^*} \left(\mathfrak{M}_2(\mu_{X^J_t}) + \mathfrak{M}_2(\mu_{\overline{X}^J_t})\right) \mathcal{W}_2^2 \right] 
        &\le \sqrt{\proba(\Omega^*)} \sqrt{\E \left[ \left(\mathfrak{M}_2(\mu_{X^J_t}) + \mathfrak{M}_2(\mu_{\overline{X}^J_t})\right)^2 \mathcal{W}_2^4 \right]}.
    \end{align}
    To bound the term under the square root, we apply the Cauchy-Schwarz a second time to decouple the centered moments from the Wasserstein distance. By Jensen's inequality, $(x+y)^4 \le 8(x^4+y^4)$ and $(\mathfrak{M}_2)^4 \le \mathfrak{M}_8$, meaning the centered moments provide exponential decay: $\E[\mathfrak{M}_8(t)] \le \e^{-\edecay{8}t} \E[\mathfrak{M}_8(0)]$. Simultaneously, the empirical Wasserstein distance is bounded by the raw moments up to order 8 ($\E[\mathcal{W}_2^8] \le 2^7 \E[|X_t|^8 + |\overline{X}_t|^8]$), which are bounded uniformly in time by \cref{lem:uit-raw-moments-particle}. Combining this exponentially decaying term with the $\mathcal{O}(1)$ raw moment bound yields:
    \begin{equation} \label{eq: bad-exp-bound}
        \sqrt{\E \left[ \left(\mathfrak{M}_2(\mu_{X^J_t}) + \mathfrak{M}_2(\mu_{\overline{X}^J_t})\right)^2 \mathcal{W}_2^4 \right]} \le 2^5 K_{\rm{raw}, 8}^2 \e^{-\frac{\edecay{8}} {4} t} \sqrt{\momentp{8}{\mfldis_0}}  < \infty.
    \end{equation}
    Therefore, multiplying by the probability bound $\sqrt{\proba(\Omega^*)}$, which follows from \eqref{eq: prob-omega-bound}:
    \begin{align}
        \E \left[ \mathbf{1}_{\Omega^*} \left(\mathfrak{M}_2(\mu_{X^J_t}) + \mathfrak{M}_2(\mu_{\overline{X}^J_t})\right) \mathcal{W}_2^2 \right] 
        &\le \sqrt{(C_{Bad}+C_{Bad}^{MF})C_{2r,2}}\,
        J^{-r/4}
        \sqrt{\mathfrak M_{2r}(\mfldis_0)}
        2^5 K_{\rm{raw},8}^2
        \e^{-\frac{\edecay{8}}{4}t}
        \sqrt{\momentp{8}{\mfldis_0}} \,.
    \end{align}

    \textbf{Summary:}
    Combining both parts, the total nonlinear stability error satisfies:
    \begin{align}
        \E[|\Delta_\alpha^{stab}(t)|^2]  &\le 4 C_{\mathcal M}^2 K_0 c_{\mathcal E,1}^{-1} \e^{-\kappa t} \E[\mathcal{E}_t] \nonumber \\
        &\quad + 2 C_{\mathcal M}^2
        \sqrt{(C_{Bad}+C_{Bad}^{MF})C_{2r,2}}\,
        J^{-r/4}
        \sqrt{\mathfrak M_{2r}(\mfldis_0)}
        2^5K_{\rm{raw},8}^2
        \e^{-\frac{\edecay{8}}{4}t}
        \sqrt{\momentp{8}{\mfldis_0}}  \notag \\
        &\le C_Q \e^{-\kappa t} \E[\mathcal{E}_t] + \frac{C_{\Delta}}{J^{r/4}} \e^{-\frac{\edecay{8}}{4}t} \,, \label{eq:rho_bound_final}
    \end{align}
    where the constants are:
    \begin{align}\label{eq:C_Q_kinetic}
        C_Q &:= 4 C_{\mathcal M}^2 K_0 c_{\mathcal E,1}^{-1}\,, \qquad
        C_{\Delta}
        :=
        2^6 C_{\mathcal M}^2
        \sqrt{(C_{Bad}+C_{Bad}^{MF})C_{2r,2}}\,
        K_{\rm{raw},8}^2
        \sqrt{\mathfrak M_{2r}(\mfldis_0)\momentp{8}{\mfldis_0}}\,.
    \end{align}
   Since the initial moments up to order $2r \ge 8$ are finite, we have $r \ge 4$. Hence $J^{-r/4} \le J^{-1}$ for $J \ge 1$, ensuring the spatial error scales as $\mathcal{O}(J^{-1})$:
    \begin{align}\label{eq:technical-inequality-kinetic}
         \E[|\Delta_\alpha^{stab}(t)|^2] \le C_Q \e^{-\kappa t} \E[\mathcal{E}_t] + \frac{C_{\Delta}}{J} \e^{-\frac{\edecay{8}}{4}t } \,,
    \end{align}
    and substituting this back into the main energy inequality yields a closed Gr\"onwall estimate.

\subsection{Proof of \cref{thm:stab_particle}}

\begin{proof}
    Define the raw differences
    \[
        \delta X_t^j := X_t^j-\tilde X_t^j,
        \qquad
        \delta V_t^j := V_t^j-\tilde V_t^j,
    \]
    their empirical means
    \[
        \overline{\delta X}_t := \frac1J\sum_{j=1}^J \delta X_t^j,
        \qquad
        \overline{\delta V}_t := \frac1J\sum_{j=1}^J \delta V_t^j,
    \]
    and the centered fluctuations
    \[
        \widetilde{\delta X}_t^j := \delta X_t^j-\overline{\delta X}_t,
        \qquad
        \widetilde{\delta V}_t^j := \delta V_t^j-\overline{\delta V}_t.
    \]

    We also introduce the Euclidean coupling error
    \[
        \widehat{\mathcal G}_t := \frac1J\sum_{j=1}^J \Bigl(|\delta X_t^j|^2+|\delta V_t^j|^2\Bigr).
    \]
    As in the proof of \cref{thm:uit-poc-kinetic}, we define the modified energy
    \begin{equation}
        \label{eq:defGt_stab}
        \mathcal G_t
        :=
        \frac1J\sum_{j=1}^J \phi_{a,2}(\delta X_t^j,\delta V_t^j)
        - \frac1m |\overline{\delta X}_t|^2,
        \qquad
        a:=\frac12+\frac1m.
    \end{equation}
    Decomposing into fluctuations and center of mass gives
    \[
        \mathcal G_t
        =
        \frac1J\sum_{j=1}^J \phi_{a,2}(\widetilde{\delta X}_t^j,\widetilde{\delta V}_t^j)
        + \phi_{a-\frac1m,2}(\overline{\delta X}_t,\overline{\delta V}_t).
    \]
    Since $a-\frac1m=\frac12$ and $\gamma>m/\sqrt2$, both quadratic forms are positive definite. By \cref{cor:coupling-energy-equivalence},
    \begin{equation}
        \label{eq:stab_energy_equiv}
        c_{\mathcal E,1}\widehat{\mathcal G}_t
        \le
        \mathcal G_t
        \le
        c_{\mathcal E,2}\widehat{\mathcal G}_t.
    \end{equation}

    \proofstep{Step 1: Drift decomposition}
    The coupled dynamics satisfy $d(\delta X_t^j)=\delta V_t^j\,dt$ and
    \begin{align}
        m\,d(\delta V_t^j)
        &=
        -\gamma \delta V_t^j\,dt
        -
        \Bigl[
            (X_t^j-\mathcal M_\alpha(\mu_{X_t^J}))
            -
            (\tilde X_t^j-\mathcal M_\alpha(\mu_{\tilde X_t^J}))
        \Bigr]dt \\
        & 
         \qquad + \sigma \left( S(X_t^j-\mathcal M_\alpha(\mu_{X_t^J})) - S(\tilde X_t^j-\mathcal M_\alpha(\mu_{\tilde X_t^J})) \right)\,dW_t^j \,.
    \end{align}
    We decompose the nonlinear shift around the empirical mean:
    \begin{equation}
        \label{eq:stab_pivot}
        \mathcal M_\alpha(\mu_{X_t^J})-\mathcal M_\alpha(\mu_{\tilde X_t^J})
        =
        \overline{\delta X}_t + \Delta_\alpha^{\rm stab}(t),
    \end{equation}
    where $\Delta_\alpha^{\rm stab}(t) := \mathcal M_\alpha(\mu_{X_t^J}) - \mathcal M_\alpha(\mu_{\tilde X_t^J}) - \overline{\delta X}_t$.
    Because both systems are purely empirical, there is no mean-field sampling error ($\mathfrak{D}_t \equiv 0$). The deterministic forcing elegantly simplifies to $-\widetilde{\delta X}_t^j + \Delta_\alpha^{\rm stab}(t)$.

    \proofstep{Step 2: Differential inequality for the energy}
    Applying It\^o's formula to \eqref{eq:defGt_stab} and using the choice $a=\frac12+\frac1m$, all quadratic cross terms cancel. We obtain
    \begin{equation}
        \label{eq:stab_diffeq_pre}
        \frac{d}{dt}\E[\mathcal G_t]
        \le
        -\E[\mathcal D_{\rm lin}(t)]
        + \mathcal A_1(t)
        + \mathcal A_3(t),
    \end{equation}
    where the linear dissipation $\mathcal D_{\rm lin}(t)$ acts on the spatial fluctuations and the raw velocities:
    \[
        \mathcal D_{\rm lin}(t)
        :=
        \frac1J\sum_{j=1}^J
        \left[
            \frac1\gamma |\widetilde{\delta X}_t^j|^2
            +
            \left(\frac{2\gamma}{m}-\frac{m}{\gamma}\right)|\delta V_t^j|^2
        \right],
    \]
    the coupling stability term is:
    \[
        \mathcal A_1(t)
        :=
        \E\left\langle
            \frac2m \overline{\delta V}_t + \frac1\gamma \overline{\delta X}_t,
            \Delta_\alpha^{\rm stab}(t)
        \right\rangle,
    \]
    and the trace of the stochastic noise difference is:
    \[
        \mathcal A_3(t)
        :=
        \frac{\sigma^2}{m^2J}
        \E\sum_{j=1}^J
        \left\|
            S(X_t^j-\mathcal M_\alpha(\mu_{X_t^J}))
            -
            S(\tilde X_t^j-\mathcal M_\alpha(\mu_{\tilde X_t^J}))
        \right\|_{\rm HS}^2.
    \]

    By \eqref{eq:stab_energy_equiv}, $\left| \frac2m \overline{\delta V}_t + \frac1\gamma \overline{\delta X}_t \right| \le C_{\rm norm}\,\sqrt{\mathcal G_t}$, where
    \[
        C_{\rm norm}
        :=
        \sqrt{
            \frac{2}{c_{\mathcal E,1}}
            \max\left(\frac4{m^2},\frac1{\gamma^2}\right)
        }.
    \]
    Thus, by the Cauchy-Schwarz:
    \begin{equation}
        \label{eq:A1stab}
        \mathcal A_1(t)
        \le
        C_{\rm norm}\,
        \sqrt{\E[\mathcal G_t]}\,
        \sqrt{\E|\Delta_\alpha^{\rm stab}(t)|^2}.
    \end{equation}

    For the noise term $\mathcal A_3$, applying the Lipschitz trace bound for $S$ with constant $\chi_S$ yields:
    \[
        \mathcal A_3(t)
        \le
        \frac{\sigma^2\chi_S}{m^2}
        \E\left[
            \frac1J\sum_{j=1}^J
            \bigl|\widetilde{\delta X}_t^j+\Delta_\alpha^{\rm stab}(t)\bigr|^2
        \right].
    \]
    Because $\widetilde{\delta X}_t^j$ is centered by definition ($\sum_{j=1}^J \widetilde{\delta X}_t^j = 0$), and $\Delta_\alpha^{\rm stab}(t)$ is independent of the particle index $j$, their inner product vanishes when averaged over the swarm. Thus, the square expands orthogonally without requiring Young's inequality:
    \[
        \frac1J\sum_{j=1}^J \bigl|\widetilde{\delta X}_t^j+\Delta_\alpha^{\rm stab}(t)\bigr|^2 = \frac1J\sum_{j=1}^J |\widetilde{\delta X}_t^j|^2 + |\Delta_\alpha^{\rm stab}(t)|^2.
    \]
    Defining $\Lambda_{\rm noise}:=\frac{\sigma^2\chi_S}{m^2}$, we substitute this expansion back into the energy derivative \eqref{eq:stab_diffeq_pre}:
    \begin{align}
        \frac{d}{dt}\E[\mathcal G_t] &\le \E \left[ \frac1J\sum_{j=1}^J \left( \Lambda_{\rm noise} - \frac1\gamma \right) |\widetilde{\delta X}_t^j|^2 \right] - \left(\frac{2\gamma}{m}-\frac{m}{\gamma}\right) \E\left[\frac1J\sum_{j=1}^J |\delta V_t^j|^2\right] \nonumber \\
        &\quad + \mathcal A_1(t) + \Lambda_{\rm noise}\E|\Delta_\alpha^{\rm stab}(t)|^2.
    \end{align}
    By the condition \eqref{eq:stab-small-noise-stability}, $\sigma^2 \le \frac{m^2}{\gamma \chi_S} \implies \Lambda_{\rm noise} \le \frac1\gamma$. Thus, the trace noise is  and completely absorbed by the linear spatial dissipation. Dropping the remaining negative velocity dissipation gives the closed differential inequality:
    \begin{equation}
        \label{eq:stab_diffeq}
        \frac{d}{dt}\E[\mathcal G_t]
        \le
        C_{\rm norm}\,
        \sqrt{\E[\mathcal G_t]}\,
        \sqrt{\E|\Delta_\alpha^{\rm stab}(t)|^2}
        +
        \Lambda_{\rm noise}\,\E|\Delta_\alpha^{\rm stab}(t)|^2.
    \end{equation}

    \proofstep{Step 3: Estimate on the nonlinear stability remainder}
    By \cref{lem:stability_est},
    \[
        |\Delta_\alpha^{\rm stab}(t)|^2
        \le
        2C_{\mathcal M}^2
        \Bigl(
            \mathfrak M_2(\mu_{X_t^J})
            +
            \mathfrak M_2(\mu_{\tilde X_t^J})
        \Bigr)
        \mathcal{W}_2^2(\mu_{X_t^J},\mu_{\tilde X_t^J}).
    \]
    Let
    \[
        \kappa:=\frac{\edecay{8}}8.
    \]
    By \cref{assump:stab}, we have \(\kappa<\edecay{2}\). We define the corresponding bad sets:
    \[
        \Omega_{\kappa,X}
        :=
        \left\{
            \sup_{t\ge0} e^{\kappa t}\mathfrak M_2(\mu_{X_t^J})
            \ge \widetilde{K}_0
        \right\},
        \qquad
        \Omega_{\kappa,\tilde X}
        :=
        \left\{
            \sup_{t\ge0} e^{\kappa t}\mathfrak M_2(\mu_{\tilde X_t^J})
            \ge \widetilde{K}_0
        \right\},
    \]
    and $\widetilde{\Omega}_\kappa^\ast := \Omega_{\kappa,X}\cup \Omega_{\kappa,\tilde X}$. To maintain symmetry, the uniform threshold is defined via both initial distributions:
    \[
        \widetilde{K}_0
        :=
        \frac{
            \E\!\left[L_2(\mathcal X_0^J,\mathcal V_0^J)\right]
            +
            \E\!\left[L_2(\widetilde{\mathcal X}_0^J,\widetilde{\mathcal V}_0^J)\right]
        }{c_1^{(2)}}
        +1.
    \]
    To attain a convergence rate of $J^{-q}$, we define the concentration parameter
    \[
        \tilde q:=4\max\{1,q\}.
    \]
    Then \(2\tilde q=p_\ast\). By applying the concentration inequality
    (\cref{lem:concentration-ineq}) with moment order \(p_\ast\), using the gap condition
    \[
        \edecay{p_\ast}>\tilde q\,\kappa
    \]
    from \cref{assump:stab}, we bound the probability of the combined bad set:
    \[
        \proba(\widetilde{\Omega}_\kappa^\ast)
        \le
        \widetilde C_{\rm Bad}\,J^{-\tilde q/2}
        =
        \widetilde C_{\rm Bad}\,J^{-2\max\{1,q\}},
    \]
    where the \(J\)-independent constant is defined via the finite \(p_\ast\)-th initial moments:
    \[
        \widetilde C_{\rm Bad}
        :=
        C_{Bad,\tilde q,\kappa}
        C_{p_\ast,2}
        \left(
            \mathfrak M_{p_\ast}(\rho_0)
            +
            \mathfrak M_{p_\ast}(\tilde\rho_0)
        \right).
    \]

    On the Good Set $(\widetilde{\Omega}_\kappa^\ast)^c$, both spatial moments are  bounded by $\widetilde{K}_0 e^{-\kappa t}$, and the empirical Wasserstein distance is bounded by the Euclidean energy ($\mathcal{W}_2^2 \le c_{\mathcal E,1}^{-1}\mathcal G_t$). Hence:
    \[
        \E\!\left[
            \mathbf 1_{(\widetilde{\Omega}_\kappa^\ast)^c}
            \Bigl(
                \mathfrak M_2(\mu_{X_t^J})+\mathfrak M_2(\mu_{\tilde X_t^J})
            \Bigr)
            \mathcal{W}_2^2
        \right]
        \le
        2\widetilde{K}_0 c_{\mathcal E,1}^{-1} e^{-\kappa t}\E[\mathcal G_t].
    \]

    On the Bad Set $\widetilde{\Omega}_\kappa^\ast$, applying the Cauchy--Schwarz inequality:
    \begin{align}
        &\E\!\left[
            \mathbf 1_{\widetilde{\Omega}_\kappa^\ast}
            \Bigl(
                \mathfrak M_2(\mu_{X_t^J})+\mathfrak M_2(\mu_{\tilde X_t^J})
            \Bigr)
            \mathcal{W}_2^2
        \right]
        \le
        \proba(\widetilde{\Omega}_\kappa^\ast)^{1/2}
        \left(
            \E\!\left[
                \Bigl(
                    \mathfrak M_2(\mu_{X_t^J})+\mathfrak M_2(\mu_{\tilde X_t^J})
                \Bigr)^2
                \mathcal{W}_2^4(\mu_{X_t^J},\mu_{\tilde X_t^J})
            \right]
        \right)^{1/2}.
    \end{align}
    By taking the square root, the probability bound gives
    \[
        \proba(\widetilde{\Omega}_\kappa^\ast)^{1/2}
        \le
        \widetilde C_{\rm Bad}^{1/2}J^{-\max\{1,q\}}
        \le
        \widetilde C_{\rm Bad}^{1/2}J^{-q}.
    \]
    For the remaining expectation, we apply the Cauchy-Schwarz a second time to decouple the centered moments from the Wasserstein distance. By Jensen's inequality ($\E[\mathfrak{M}_2^4] \le \E[\mathfrak{M}_8]$), the moments provide exponential decay $\e^{-\edecay{8} t}$. By the uniform-in-time raw moment bounds (\cref{lem:uit-raw-moments-particle}), $\E[\mathcal{W}_2^8] \le 2^7 K_{\text{raw}, 8}^8 \le \mathcal{O}(1)$. The term under the square root is therefore  bounded uniformly in time by:
    \[
        \widetilde K_{\rm raw}\,e^{-\edecay{8} t/4}, \quad \text{where } \widetilde K_{\rm raw} := 2^5 K_{\rm{raw},8}^2 \sqrt{ \E[\mathfrak{M}_8(\mu_{X_0^J})] + \E[\mathfrak{M}_8(\mu_{\tilde X_0^J})] }.
    \]
    Combining these yields the scaling:
    \[
        \E\!\left[
            \mathbf 1_{\widetilde{\Omega}_\kappa^\ast}
            \Bigl(
                \mathfrak M_2(\mu_{X_t^J})+\mathfrak M_2(\mu_{\tilde X_t^J})
            \Bigr)
            \mathcal{W}_2^2
        \right]
        \le
        \widetilde C_{\rm Bad}^{1/2} \widetilde K_{\rm raw}\,J^{-q} e^{-\edecay{8} t/4}.
    \]

    Summing the good-set and bad-set bounds, the nonlinear stability error is bounded by:
    \begin{equation}
        \label{eq:stab_delta_bound}
        \E|\Delta_\alpha^{\rm stab}(t)|^2
        \le
        \widetilde{C}_Q e^{-\kappa t}\E[\mathcal G_t]
        +
        \widetilde C_\Delta J^{-q} e^{-\edecay{8} t/4},
    \end{equation}
    where we define the finite, completely $J$-independent constants:
    \[
        \widetilde{C}_Q := 4C_{\mathcal M}^2 \widetilde{K}_0 c_{\mathcal E,1}^{-1}, \qquad \widetilde C_\Delta := 2C_{\mathcal M}^2 \widetilde C_{\rm Bad}^{1/2} \widetilde K_{\rm raw}.
    \]

    \proofstep{Step 4: Gr\"onwall closure}
    Substituting \eqref{eq:stab_delta_bound} into \eqref{eq:stab_diffeq}, and using $\sqrt{A+B}\le \sqrt A + \sqrt B$ and $xy\le \frac12x^2+\frac12y^2$, we unpack the nested integrals. Since all exponential terms ($e^{-\kappa/2 t}, e^{-\edecay{8}/8 t}$) are  bounded from above by the slowest rate $e^{-\edecay{8} t/16}$, we can cleanly bound the entire derivative as:
    \[
        \frac{d}{dt}\E[\mathcal G_t]
        \le
        \tilde c_1 e^{-\edecay{8} t/16}\E[\mathcal G_t]
        +
        \tilde c_2 e^{-\edecay{8} t/16}J^{-q},
    \]
    for suitable finite constants $\tilde c_1,\tilde c_2>0$ constructed from $C_{\rm norm}, \widetilde{C}_Q, \Lambda_{\rm noise},$ and $\widetilde C_\Delta$.
    Integrating this inequality from $0$ to $t$ yields:
    \[
        \E[\mathcal G_t]
        \le
        \E[\mathcal G_0]
        + \frac{16\tilde c_2}{\edecay{8}}J^{-q}
        + \int_0^t \tilde c_1 e^{-\edecay{8} s/16}\E[\mathcal G_s]\,ds.
    \]
    Applying Gr\"onwall's inequality yields:
    \[
        \E[\mathcal G_t]
        \le
        \left(
            \E[\mathcal G_0] + \frac{16\tilde c_2}{\edecay{8}}J^{-q}
        \right)
        \exp\!\left(\frac{16\tilde c_1}{\edecay{8}}\right).
    \]
    Finally, using \eqref{eq:stab_energy_equiv} to un-shift the energy, we conclude:
    \[
        \E[\widehat{\mathcal G}_t]
        \le
        \frac{1}{c_{\mathcal E,1}}\E[\mathcal G_t]
        \le
        \frac{c_{\mathcal E,2}}{c_{\mathcal E,1}}
        e^{16\tilde c_1/\edecay{8}}\E[\widehat{\mathcal G}_0]
        +
        \frac{16\tilde c_2}{c_{\mathcal E,1}\edecay{8}}
        e^{16\tilde c_1/\edecay{8}}
        J^{-q},
    \]
    which completes the proof with
    \[
        C_{\rm Stab,1}
        :=
        \frac{c_{\mathcal E,2}}{c_{\mathcal E,1}}
        e^{16\tilde c_1/\edecay{8}},
        \qquad
        C_{\rm Stab,2}
        :=
        \frac{16\tilde c_2}{c_{\mathcal E,1}\edecay{8}}
        e^{16\tilde c_1/\edecay{8}}.
    \]
\end{proof}
\section{Auxiliary Results}\label{sec: auxiliary-results}
This section collects the auxiliary estimates used in the proof of the main
theorems. The results are organized according to the proof strategy in
\cref{sec:proof-overview}. We first prove exponential decay of centered moments
for both the interacting particle system and the mean-field process. We then use
this decay to derive uniform-in-time bounds on raw moments. Next, we establish
concentration inequalities for the empirical Lyapunov functional, both for the
interacting system and for the synchronously coupled independent mean-field
particles. Finally, we recall the weighted-mean stability and Monte Carlo
estimates imported from the first-order CBO analysis.

\subsection{Decay of Centered Moments}
\label{sec: centered-moments-decay}
\subsubsection{Setup and Auxiliary Lemmas for Centered Moments}

We prove exponential decay of the centered moments for the interacting particle
system and for the mean-field process. For \(p>2\), the proof uses the polynomial
Lyapunov functional built from \(\phi_{a_p,p}\) and the shifted variables
\((Y,\hat Z)\) introduced in \cref{sec:setting-transformation}, where
\[
    a_p:=\frac1p\left(1-\frac{m(p-2)}{\gamma^2}\right).
\]
We define the pathwise centered Lyapunov functional
\begin{equation}\label{eq:random-Lp-def}
    L_p(t)
    :=
    \frac1J\sum_{j=1}^J \phi_{a_p,p}(Y_t^j,\hat Z_t^j)
    =
    \frac1J\sum_{j=1}^J
    \left(
        a_p|Y_t^j|^p
        +
        |\hat Z_t^j|^p
        +
        \frac{m}{\gamma}|Y_t^j|^{p-2}
        \ip{Y_t^j}{\hat Z_t^j}
    \right),
    \qquad p>2.
\end{equation}
Its expected counterpart is denoted by
\begin{equation}\label{eq:expected-Lp-def}
    \mathcal L_p(t):=\E[L_p(t)].
\end{equation}
By \cref{lem:hypercoercivity}, \(L_p(t)\) is equivalent to the standard
\(p\)-th moment norm in the shifted variables. Since
\[
    V_t^j-m_V(\mu_{V_t^J})
    =
    \hat Z_t^j+\frac1\gamma Y_t^j,
\]
this equivalence also controls the centered position and velocity moments.

The quadratic case \(p=2\) is treated separately. In this case we do not use
the specialization \(\phi_{a,2}\) from \cref{eq:phi-ap}. Instead, we use the quadratic form defined as in \eqref{eq:psi2-centered}:
\begin{equation}
    \psi_2(y,z)
    :=
    a_2|y|^2+|z|^2+\frac5\gamma\ip{y}{z},
    \qquad
    a_2:=\frac9{2m}+\frac1{\gamma^2}.
\end{equation}
The corresponding pathwise quadratic Lyapunov functional is
\begin{equation}\label{eq:random-L2-def}
    L_2(t)
    :=
    \frac1J\sum_{j=1}^J \psi_2(Y_t^j,\hat Z_t^j)
    =
    \frac1J\sum_{j=1}^J
    \left(
        a_2|Y_t^j|^2
        +
        |\hat Z_t^j|^2
        +
        \frac5\gamma\ip{Y_t^j}{\hat Z_t^j}
    \right).
\end{equation}
We write
\begin{equation}\label{eq:expected-L2-def}
    \mathcal L_2(t):=\E[L_2(t)].
\end{equation}
This functional is used only for the quadratic centered-moment estimate. The
quadratic form \(\phi_{a,2}\) remains reserved for the coupling energy in the
proofs of \cref{thm:uit-poc-kinetic,thm:stab_particle}.

For the mean-field samples, we use the same notation after replacing
\((Y_t^j,\hat Z_t^j)\) by the corresponding centered-shifted mean-field variables
\((\overline Y_t^j,\widehat{\overline Z}_t^j)\). Since each mean-field sample has law
\(\overline\rho_t\), the same norm-equivalence estimates apply to the mean-field
centered moments.

Before analyzing the decay of the expected Lyapunov functionals
\(\mathcal L_p(t)=\E[L_p(t)]\), we first derive the corresponding pathwise
It\^o differential inequalities for \(L_p(t)\).

\begin{lemma}[Quadratic variation of centered noise] \label{lem:qvar-centered-noise}
    For all $j \in \{1, \dots, J\}$, the quadratic variation of the centered noise process $M^{j,\mathrm{center}}$ satisfies, in the sense of symmetric semi-definite matrices:
    \begin{equation} \label{eq:qvar_bound}
        \d \langle M^{j,\mathrm{center}} \rangle_t \preceq \frac{\sigma^2}{m^2} \left[ S(Y_t^j + \Delta_\alpha) S(Y_t^j + \Delta_\alpha)^\top + \frac{1}{J}\sum_{k=1}^J S(Y_t^k + \Delta_\alpha) S(Y_t^k + \Delta_\alpha)^\top \right] \d t\,.
    \end{equation}
\end{lemma}

\begin{proof}
    Let $G_j(t) := \frac{\sigma}{m} S(Y_t^j + \Delta_\alpha)$. The particle noise increment is $\d M_t^j = G_j(t) \d W_t^j$, and the averaged noise increment is $\d \overline{M}_t = \frac{1}{J} \sum_{k=1}^J G_k(t) \d W_t^k$. The centered noise is defined as $\d M_t^{j, \mathrm{center}} = \d M_t^j - \d \overline{M}_t$.
    
    For any fixed vector $u \in \R^d$, using the independence of the Brownian motions $W^k$, the quadratic variation is:
    \begin{align}
        u^\top \d \langle M^{j, \mathrm{center}} \rangle_t u &= \left( \left(1 - \frac{2}{J}\right) |G_j^\top u|^2 + \frac{1}{J^2} \sum_{k=1}^J |G_k^\top u|^2 \right) \d t \nonumber \\
        &\le \left( |G_j^\top u|^2 + \frac{1}{J} \sum_{k=1}^J |G_k^\top u|^2 \right) \d t.
    \end{align}
    Removing the test vector $u$ yields the matrix inequality.
\end{proof}

\begin{lemma}[It\^o expansion of the shifted high-moment Lyapunov functional]
\label{lem:ito-expansion-Lp}
Let \(p>2\), and recall that 
\[
    L_p(t)
    :=
    \frac1J\sum_{j=1}^J
    \phi_{a,p}(Y_t^j,\hat Z_t^j),
    \qquad
    \phi_{a,p}(y,z)
    :=
    a|y|^p+|z|^p+\frac m\gamma |y|^{p-2}\langle y,z\rangle \,,
\] and $\Delta_\alpha(t):=m_X(\mu_{X_t^J})-\mathcal M_\alpha(t)$. Hence, $ X_t^j-\mathcal M_\alpha(t)=Y_t^j+\Delta_\alpha(t)\,,$
and
\[
    \d\hat Z_t^j
    =
    (-K_Z\hat Z_t^j-K_YY_t^j)\d t+\d M_t^{j,\mathrm{center}},
\]
where
\begin{equation}\label{eq:centered-noise-martingale-diff}
    \d M_t^{j,\mathrm{center}}
    :=
    \frac{\sigma}{m}
    \left[
        \left(1-\frac1J\right)
        S(Y_t^j+\Delta_\alpha(t))\,\d W_t^j
        -
        \frac1J\sum_{k\ne j}
        S(Y_t^k+\Delta_\alpha(t))\,\d W_t^k
    \right].
\end{equation}
The stochastic differential of \(L_p(t)\) is
\begin{align}
    \d L_p(t)
    &=
    \frac1J\sum_{j=1}^J
    \Bigg\{
    \left[
        \left(\frac{ap}{\gamma}-\frac m\gamma K_Y\right)|Y_t^j|^p
        -
        pK_Z|\hat Z_t^j|^p
    \right]\d t
    \notag\\
    &\quad
    +
    \left[
        \left(
            ap-\frac m\gamma K_Z+\frac{m(p-1)}{\gamma^2}
        \right)
        |Y_t^j|^{p-2}\langle Y_t^j,\hat Z_t^j\rangle
        -
        pK_Y|\hat Z_t^j|^{p-2}\langle \hat Z_t^j,Y_t^j\rangle
    \right]\d t
    \notag\\
    &\quad
    +
    \left[
        \frac m\gamma |Y_t^j|^{p-2}|\hat Z_t^j|^2
        +
        \frac{m(p-2)}{\gamma}
        |Y_t^j|^{p-4}\langle Y_t^j,\hat Z_t^j\rangle^2
    \right]\d t
    \notag\\
    &\quad
    +
    \mathcal T_{\mathrm{noise}}^j(t)\,\d t
    +
    \d M_{p,t}^j
    \Bigg\},
    \label{eq:ito-expansion-Lp-full}
\end{align}
where
\begin{equation}\label{eq:Mpjt-definition}
    \d M_{p,t}^j
    :=
    p|\hat Z_t^j|^{p-2}
    \left\langle \hat Z_t^j,\d M_t^{j,\mathrm{center}}\right\rangle
    +
    \frac m\gamma |Y_t^j|^{p-2}
    \left\langle Y_t^j,\d M_t^{j,\mathrm{center}}\right\rangle .
\end{equation}
We denote $ M_{p,t}:=\frac1J\sum_{j=1}^J M_{p,t}^j$ as the local martingale part of \(L_p(t)\). The noise trace satisfies
\begin{align}
    \mathcal T_{\mathrm{noise}}^j(t)
    &=
    \frac12
    \operatorname{Tr}
    \left(
        \nabla^2(|\hat Z_t^j|^p)\,
        \d\langle M^{j,\mathrm{center}}\rangle_t
    \right)
    \notag\\
    &=
        \frac{p\sigma^2}{2m^2}\left(1-\frac1J\right)^2 |\hat Z_t^j|^{p-4}
        \left[
            (p-2)\left|S(X_t^j-\mathcal M_\alpha)^\top\hat Z_t^j\right|^2
            +|\hat Z_t^j|^2\norm{S(X_t^j-\mathcal M_\alpha)}_F^2
        \right] \notag \\
        &\quad
        +\frac{p\sigma^2}{2m^2J^2}\sum_{k\ne j}|\hat Z_t^j|^{p-4}
        \left[
            (p-2)\left|S(X_t^k-\mathcal M_\alpha)^\top\hat Z_t^j\right|^2
            +|\hat Z_t^j|^2\norm{S(X_t^k-\mathcal M_\alpha)}_F^2
        \right]\\
        & \le
    \frac{p\sigma^2\chi_S}{2m^2}
    |\hat Z_t^j|^{p-2}
    \left(
        |Y_t^j+\Delta_\alpha(t)|^2
        +
        \frac1J\sum_{k=1}^J |Y_t^k+\Delta_\alpha(t)|^2
    \right).
    \label{eq:trace_bound}
\end{align}
\end{lemma}

\begin{proof}
We use the centered-shifted dynamics
\[
    \d Y_t^j
    =
    \left(\hat Z_t^j+\frac1\gamma Y_t^j\right)\d t,
    \qquad
    \d\hat Z_t^j
    =
    (-K_Z\hat Z_t^j-K_YY_t^j)\d t+\d M_t^{j,\mathrm{center}}.
\]

\textbf{Term A: \(a|Y|^p\).}
Since \(Y_t^j\) has finite variation,
\begin{align}
    \d(a|Y_t^j|^p)
    &=
    ap|Y_t^j|^{p-2}
    \langle Y_t^j,\d Y_t^j\rangle =
    \left[
        \frac{ap}{\gamma}|Y_t^j|^p
        +
        ap|Y_t^j|^{p-2}\langle Y_t^j,\hat Z_t^j\rangle
    \right]\d t .
    \label{eq:ito-term-A}
\end{align}

\textbf{Term B: \(|\hat Z|^p\).}
It\^o's formula gives
\begin{align}
    \d|\hat Z_t^j|^p
    &=
    p|\hat Z_t^j|^{p-2}
    \langle \hat Z_t^j,\d\hat Z_t^j\rangle
    +
    \frac12
    \operatorname{Tr}
    \left(
        \nabla^2(|\hat Z_t^j|^p)
        \d\langle M^{j,\mathrm{center}}\rangle_t
    \right)
    \notag\\
    &=
    \left[
        -pK_Z|\hat Z_t^j|^p
        -
        pK_Y|\hat Z_t^j|^{p-2}
        \langle \hat Z_t^j,Y_t^j\rangle
    \right]\d t
    +
    \mathcal T_{\mathrm{noise}}^j(t)\,\d t
    +
    p|\hat Z_t^j|^{p-2}
    \left\langle \hat Z_t^j,\d M_t^{j,\mathrm{center}}\right\rangle .
    \label{eq:ito-term-B}
\end{align}
The trace bound \eqref{eq:trace_bound} follows from
\(\cref{lem:qvar-centered-noise}\), the definitions of \(\tau(S)\) and
\(\chi_S=\tau(S)+p-2\), and the inequalities
\[
    |S(x)^\top z|^2\le |x|^2|z|^2,
    \qquad
    \|S(x)\|_F^2\le \tau(S)|x|^2.
\]

\textbf{Term C: the cross term.}
Let
\[
    U(y):=\frac m\gamma |y|^{p-2}y.
\]
Since \(Y_t^j\) has finite variation,
\[
    \d\left(
        \frac m\gamma |Y_t^j|^{p-2}
        \langle Y_t^j,\hat Z_t^j\rangle
    \right)
    =
    \langle \d U(Y_t^j),\hat Z_t^j\rangle
    +
    \langle U(Y_t^j),\d\hat Z_t^j\rangle .
\]
First,
\[
    \d U(Y_t^j)
    =
    \frac m\gamma
    \left[
        |Y_t^j|^{p-2}\d Y_t^j
        +
        (p-2)|Y_t^j|^{p-4}
        \langle Y_t^j,\d Y_t^j\rangle Y_t^j
    \right].
\]
Substituting
\[
    \d Y_t^j
    =
    \left(\hat Z_t^j+\frac1\gamma Y_t^j\right)\d t
\]
gives
\begin{align}
    \langle \d U(Y_t^j),\hat Z_t^j\rangle
    &=
    \frac m\gamma
    \Big[
        |Y_t^j|^{p-2}|\hat Z_t^j|^2
        +
        \frac1\gamma
        |Y_t^j|^{p-2}
        \langle Y_t^j,\hat Z_t^j\rangle
    \notag\\
    &\qquad
        +
        (p-2)|Y_t^j|^{p-4}
        \langle Y_t^j,\hat Z_t^j\rangle^2
        +
        \frac{p-2}{\gamma}
        |Y_t^j|^{p-2}
        \langle Y_t^j,\hat Z_t^j\rangle
    \Big]\d t .
    \label{eq:ito-cross-dU}
\end{align}
Second,
\begin{align}
    \langle U(Y_t^j),\d\hat Z_t^j\rangle
    &=
    \frac m\gamma |Y_t^j|^{p-2}
    \langle Y_t^j,-K_Z\hat Z_t^j-K_YY_t^j\rangle\d t
    +
    \frac m\gamma |Y_t^j|^{p-2}
    \langle Y_t^j,\d M_t^{j,\mathrm{center}}\rangle
    \notag\\
    &=
    \left[
        -\frac m\gamma K_Z
        |Y_t^j|^{p-2}
        \langle Y_t^j,\hat Z_t^j\rangle
        -
        \frac m\gamma K_Y|Y_t^j|^p
    \right]\d t
    +
    \frac m\gamma |Y_t^j|^{p-2}
    \langle Y_t^j,\d M_t^{j,\mathrm{center}}\rangle .
    \label{eq:ito-cross-dZ}
\end{align}
Combining \eqref{eq:ito-cross-dU} and \eqref{eq:ito-cross-dZ}, the coefficient of
\(|Y_t^j|^{p-2}\langle Y_t^j,\hat Z_t^j\rangle\) from the cross term is
\[
    \frac{m(p-1)}{\gamma^2}-\frac m\gamma K_Z.
\]
Adding \eqref{eq:ito-term-A}, \eqref{eq:ito-term-B}, and the cross-term contribution gives
\eqref{eq:ito-expansion-Lp-full}.
\end{proof}

\subsubsection{Decay of Centered Moments: Interacting Particle Systems}

For $p>2$, the stability of $L_p(t)$ depends on the variation of the objective function being sufficiently small relative to the polynomial order. We quantify this by defining
\begin{equation}\label{eq:mu_gap}
    \mu_{\mathrm{gap},p}
    :=2-(p-1)^{-\frac{p-2}{p}}
    \exp\left(\frac{p-1}{p}\alpha(\overline f-\underline f)\right).
\end{equation}
We treat $\mu_{\mathrm{gap},p}>0$ as a strict prerequisite. Equivalently,
\[
    \alpha(\overline f-\underline f)
    <\frac{p\log 2}{p-1}+\frac{p-2}{p-1}\log(p-1).
\]
We will make use of the following lemma frequently:
\begin{lemma}\label{lem:m-malpha}
    Let $q \ge 2$ and let \cref{assumption:bounded} hold. Then for any vector norm $|\cdot|$ it holds that
    \begin{align}
        \forall \mu \in \mathcal{P}_q,
        \qquad
        \bigl|\mathcal{M}_\alpha(\mu) - m_X(\mu)\bigr|^q
        \le \e^{\alpha(\overline f - \underline f)} \int |x - m_X(\mu)|^q\,\mu(\d x).
    \end{align}
\end{lemma}
We denote $C_{\rm lem} := \e^{\alpha(\overline f - \underline f)}$ for simplicity from now on.

\begin{assumption}[Centered-moment decay conditions]
\label{assump:centered-decay}
Let \(p\ge2\). We assume \cref{assump:admissibility} for this value of \(p\).

For \(p>2\), define
\begin{equation}\label{eq:Cpos-def}
    C_{\mathrm{pos}}
    :=
    1-\frac{3\mu_{\mathrm{gap},p}}{16}
    -
    \frac{m(p-1)(p-2)}p,
    \qquad
    \eta_p
    :=
    C_{\mathrm{pos}}-\frac{\mu_{\mathrm{gap},p}}{4}.
\end{equation}
We also define
\begin{equation}\label{eq:Lambda-p-def}
    \Lambda_p
    :=
    \frac{2(p-2)\chi_S}{m^2}
    \left(
        \frac{\mu_{\mathrm{gap},p}m^2}{64\chi_S}
    \right)^{-\frac2{p-2}}
    \left[
        1+\left(\frac{C_{\mathrm{lem}}}{2}\right)^{\frac2{p-2}}
    \right].
\end{equation}
We assume \(\eta_p>0\), and impose
\begin{equation}\label{eq:gamma-Y-condition}
    \gamma
    \ge
    \Gamma_{Y,p}
    :=
    \max\left\{
        1,\,
        \left(
            \frac{2\cdot 4^{p-1}(2/m+1)}{\eta_p}
        \right)^{\frac1{p-2}},
        \frac{\mu_{\mathrm{gap},p}m(p-1)}{2\eta_p}
    \right\}.
\end{equation}
We also assume the velocity-channel closure condition
\begin{equation}\label{eq:Z-closure}
    \Lambda_p\sigma^{\frac{2p}{p-2}}\gamma^{\frac2{p-2}}
    +
    \frac{p-1}{4\gamma}
    +
    \frac{2m(p-1)}{p\gamma}
    +
    \frac{p\mu_{\mathrm{gap},p}}{4\gamma}
    \left(
        1+\frac{m}{p\gamma}
    \right)
    \le
    \frac{p+1}{2m}\gamma.
\end{equation}
When the quadratic estimate for $p =2 $ and a high-moment estimate are used simultaneously, we
assume that the above lower-bound requirements on \(\gamma\) are compatible with the
quadratic small-noise condition
\begin{equation}\label{eq:p2-high-compatibility}
    \gamma
    <
    \frac{m}{2\tau(S)(1+C_{\mathrm{lem}})\sigma^2}.
\end{equation}
Equivalently, the chosen \(\gamma\) must satisfy
\[
    \gamma
    \in
    \left[
        \max\left\{
            \sqrt{\frac{3m}{2}},
            \sqrt{m(p-2)},
            \Gamma_{Y,p}
        \right\},
        \frac{m}{2\tau(S)(1+C_{\mathrm{lem}})\sigma^2}
    \right)
\]
and also satisfy \eqref{eq:Z-closure}.
\end{assumption}
\begin{remark}[Role and compatibility of the centered-decay assumptions]
\label{rem:role-centered-decay-assumptions}
The role of \cref{assump:centered-decay} is to record the parameter conditions needed
to close the centered Lyapunov estimate. The condition
\eqref{eq:gamma-Y-condition} controls the remainders in the spatial channel,
while \eqref{eq:Z-closure} controls the remainders in the shifted-velocity
channel. These two conditions are used in
\cref{prop:ito-centered-lyapunov-differential} to pass from the componentwise
dissipation estimate to the scalar Lyapunov inequality
\[
    \d L_p(t)\le -\lambda_p L_p(t)\,\d t+\d M_{p,t}.
\]

Although \cref{thm:exp-decay-centered-moments} is stated for a fixed moment order
\(p\ge2\), the proof of uniform-in-time propagation of chaos \cref{thm:uit-poc-kinetic} later uses several moment orders at the same
time. The quadratic estimate \(p=2\) is needed for the decay of second centered
moments, which enters the concentration and collapse estimates. Higher moments, such
as \(p=8\) and \(p=2r\ge8\), are needed to control bad events and higher-order
H\"older estimates. Thus it is not enough to know that the conditions are admissible
for each \(p\) separately; we must also make sure that the \(p=2\) conditions and the
high-moment conditions can hold for the same choice of parameters.

The qualitative meaning of the velocity-channel condition
\eqref{eq:Z-closure} depends on \(p\). The noise contribution in this
condition is
\[
    \Lambda_p\sigma^{\frac{2p}{p-2}}\gamma^{\frac2{p-2}},
\]
whereas the leading friction contribution grows like \(\gamma\). Hence:
\begin{itemize}
    \item if \(2<p<4\), the noise term grows faster than the leading friction term as
    \(\gamma\to\infty\), so increasing \(\gamma\) does not by itself improve the
    velocity-channel balance;

    \item if \(p=4\), the noise term and the leading friction term have the same
    scaling in \(\gamma\), so the condition reduces to a genuine small-noise
    restriction;

    \item if \(p>4\), the noise term grows slower than the leading friction term, so
    the velocity-channel noise can be absorbed for fixed \(\sigma\) by taking
    \(\gamma\) large enough.
\end{itemize}
In the applications below we use high moments such as \(p=8\) and \(p=2r\ge8\), so the
relevant high-moment regime is \(p>4\). There is, however, a second constraint coming from the quadratic estimate. For \(p=2\),
the small-noise condition is
\[
    K_\sigma<\frac1{m\gamma},
    \qquad\text{equivalently}\qquad
    \sigma^2\gamma
    <
    \frac{m}{2\tau(S)(1+C_{\mathrm{lem}})}.
\]
This gives an upper bound on \(\gamma\). Thus the high-moment estimates favor taking
\(\gamma\) large, while the quadratic estimate prevents taking \(\gamma\) arbitrarily
large unless \(\sigma\) is chosen small enough.

For \(p>4\), the two requirements are compatible as follows. Choose \(\gamma\) so that
\[
    \gamma
    \ge
    \max\left\{
        \sqrt{\frac{3m}{2}},
        \sqrt{m(p-2)},
        \Gamma_{Y,p}
    \right\},
\]
and so that the lower-order velocity-channel terms satisfy
\begin{equation}\label{eq:Z-lower-order}
    \frac{p-1}{4\gamma}
    +
    \frac{2m(p-1)}{p\gamma}
    +
    \frac{p\mu_{\mathrm{gap},p}}{4\gamma}
    \left(
        1+\frac{m}{p\gamma}
    \right)
    \le
    \frac{p+1}{4m}\gamma.
\end{equation}
Then the high-moment velocity condition and the quadratic small-noise condition both
hold provided
\begin{equation}\label{eq:sigma-smallness}
    0<\sigma
    <
    \min\left\{
        \left(
            \frac{p+1}{4m\Lambda_p}
        \right)^{\frac{p-2}{2p}}
        \gamma^{\frac{p-4}{2p}},
        \left(
            \frac{m}{2\tau(S)(1+C_{\mathrm{lem}})\gamma}
        \right)^{1/2}
    \right\}.
\end{equation}
Indeed, the first bound in \eqref{eq:sigma-smallness} yields
\[
    \Lambda_p\sigma^{\frac{2p}{p-2}}\gamma^{\frac2{p-2}}
    \le
    \frac{p+1}{4m}\gamma,
\]
and together with \eqref{eq:Z-lower-order} this gives
\eqref{eq:Z-closure}. The second bound gives
\[
    \sigma^2\gamma
    <
    \frac{m}{2\tau(S)(1+C_{\mathrm{lem}})},
\]
which is the quadratic small-noise condition.

Therefore, after fixing a \(\gamma\) satisfying the displayed lower bounds and
\eqref{eq:Z-lower-order}, the interval of admissible \(\sigma\)'s in
\eqref{eq:sigma-smallness} is nonempty. This shows that the quadratic
estimate and the high-moment estimates used later in the proof can be imposed
simultaneously.
\end{remark}

\begin{proposition}[It\^o differential inequality for the centered Lyapunov functionals]
\label{prop:ito-centered-lyapunov-differential}
Let \(p\ge2\), and assume the conditions of \cref{assump:centered-decay}.
Let
\[
    L_p(t):=L_p(\mathcal X_t^J,\mathcal V_t^J)
\]
denote the pathwise centered Lyapunov functional defined in
\eqref{eq:random-Lp-def} for \(p>2\), and in \eqref{eq:random-L2-def} for \(p=2\).

Define the local martingale \(M_{p,t}:=J^{-1}\sum_{j=1}^J M_{p,t}^j\) by
\begin{equation}\label{eq:martingale-Mpjt}
    \d M_{p,t}^j
    :=
    \begin{cases}
        \displaystyle
        2\left\langle \hat Z_t^j,\d M_t^{j,\mathrm{center}}\right\rangle
        +
        \frac5\gamma
        \left\langle Y_t^j,\d M_t^{j,\mathrm{center}}\right\rangle,
        & p=2,\\[1.2em]
        \displaystyle
        p|\hat Z_t^j|^{p-2}
        \left\langle \hat Z_t^j,\d M_t^{j,\mathrm{center}}\right\rangle
        +
        \frac m\gamma |Y_t^j|^{p-2}
        \left\langle Y_t^j,\d M_t^{j,\mathrm{center}}\right\rangle,
        & p>2.
    \end{cases}
\end{equation}
Then there exists a local martingale \(M_{p,t}\) such that, in the sense of semimartingale differentials,
\begin{equation}\label{eq:pathwise-Lp-differential}
    \d L_p(t)
    \le
    -\lambda_p L_p(t)\,\d t+\d M_{p,t},
\end{equation}
where
\begin{equation}\label{eq:lambda-p-cases-prop}
    \lambda_p
    :=
    \begin{cases}
        \displaystyle
        \frac{\min\{v_1,v_3\}}{c_2^{(2)}},
        & p=2,\\[1.2em]
        \displaystyle
        \frac{p\mu_{\mathrm{gap},p}}{4\gamma},
        & p>2.
    \end{cases}
\end{equation}
Here, for \(p=2\),
\[
    v_1:=\frac1{m\gamma}+\frac3{\gamma^3}-K_\sigma,
    \qquad
    v_3:=\frac{2\gamma}{m}-\frac3\gamma,
\]
and \(c_2^{(2)}\) is the upper norm-equivalence constant associated with \(\psi_2\).

Moreover, when \(p>2\), the stronger componentwise estimate holds:
\begin{equation}\label{eq:componentwise-Lp-differential}
    \d L_p(t)
    \le
    -\widehat\lambda_Y
    \left(\frac1J\sum_{j=1}^J |Y_t^j|^p\right)\d t
    -\widehat\lambda_Z
    \left(\frac1J\sum_{j=1}^J |\hat Z_t^j|^p\right)\d t
    +
    \d M_{p,t},
\end{equation}
where
\begin{align}
    \widehat\lambda_Y
    &:=
    \frac1\gamma
    \left[
        C_{\mathrm{pos}}
        -
        \gamma\left(\frac2m+\frac1{\gamma^2}\right)
        \left(\frac4\gamma\right)^{p-1}
        +
        \frac{m(p-1)}{\gamma^2}
    \right],
    \label{eq:lambdaYhat-prop}\\
    \widehat\lambda_Z
    &:=
    p\left(\frac\gamma m+\frac1\gamma\right)
    -
    \left(
        \frac{p-1}{2m}\gamma+\frac{p-1}{4\gamma}
    \right)
    -
    \frac{2m(p-1)}{p\gamma}
    -
    \Lambda_p\sigma^{\frac{2p}{p-2}}\gamma^{\frac2{p-2}}\,,
    \label{eq:lambdaZhat-prop}
\end{align}
where $C_{\mathrm{pos}} = 1-\frac{3\mu_{\mathrm{gap},p}}{16} - \frac{m(p-1)(p-2)}p$ as defined from \eqref{eq:Cpos-def}. Under \cref{assump:centered-decay}, these coefficients satisfy
\begin{equation}\label{eq:Y-component-sufficient-prop}
    \widehat\lambda_Y
    \ge
    \lambda_p
    \left(
        a_p+\frac{m(p-1)}{p\gamma}
    \right),
\end{equation}
and
\begin{equation}\label{eq:Z-component-sufficient-prop}
    \widehat\lambda_Z
    \ge
    \lambda_p
    \left(
        1+\frac{m}{p\gamma}
    \right).
\end{equation}
\end{proposition}

\begin{proof}
We prove the estimate for the interacting particle system. The same centered-shifted
calculation applies to the synchronously coupled mean-field samples, with the corresponding
mean-field shift in the noise coefficient.

\proofstep{1. The high-moment case \(p>2\)}

Recall that
\[
    L_p(t)
    =
    \frac1J\sum_{j=1}^J
    \phi_{a_p,p}(Y_t^j,\hat Z_t^j),
    \qquad
    a_p:=\frac1p\left(1-\frac{m(p-2)}{\gamma^2}\right).
\]
By \cref{lem:ito-expansion-Lp}, with \(a=a_p\),
\begin{align}
    \d L_p(t)
    &=
    \frac1J\sum_{j=1}^J
    \Bigg\{
    \left[
        \left(\frac{a_pp}{\gamma}-\frac m\gamma K_Y\right)|Y_t^j|^p
        -
        pK_Z|\hat Z_t^j|^p
    \right]\d t
    \notag\\
    &\quad
    +
    \left[
        \left(
            a_pp-\frac m\gamma K_Z+\frac{m(p-1)}{\gamma^2}
        \right)
        |Y_t^j|^{p-2}\langle Y_t^j,\hat Z_t^j\rangle
        -
        pK_Y|\hat Z_t^j|^{p-2}\langle \hat Z_t^j,Y_t^j\rangle
    \right]\d t
    \notag\\
    &\quad
    +
    \left[
        \frac m\gamma |Y_t^j|^{p-2}|\hat Z_t^j|^2
        +
        \frac{m(p-2)}{\gamma}
        |Y_t^j|^{p-4}\langle Y_t^j,\hat Z_t^j\rangle^2
    \right]\d t
    \notag\\
    &\quad
    +
    \mathcal T_{\mathrm{noise}}^j(t)\,\d t
    +
    \d M_{p,t}^j
    \Bigg\}.
    \label{eq:Lp-full-before-bounds}
\end{align}

\textbf{Cancellation of the first mixed drift.}
The coefficient of
\[
    |Y_t^j|^{p-2}\langle Y_t^j,\hat Z_t^j\rangle
\]
is
\[
    a_pp-\frac m\gamma K_Z+\frac{m(p-1)}{\gamma^2}.
\]
Since
\[
    K_Z=\frac\gamma m+\frac1\gamma,
\]
our choice of \(a_p\) gives
\[
    a_pp
    =
    1-\frac{m(p-2)}{\gamma^2}
    =
    \frac m\gamma K_Z-\frac{m(p-1)}{\gamma^2}.
\]
Therefore
\begin{equation}\label{eq:mixed-drift-cancellation-prop}
    a_pp-\frac m\gamma K_Z+\frac{m(p-1)}{\gamma^2}=0.
\end{equation}

\textbf{The remaining \(pK_Y\) mixed term.}
Using Young's inequality with conjugates
\[
    r=\frac p{p-1},
    \qquad
    s=p,
\]
and the choice \(\epsilon_1=\gamma/4\), we obtain
\begin{align}
    pK_Y|\hat Z|^{p-1}|Y|
    &\le
    pK_Y
    \left[
        \frac{p-1}{p}\frac\gamma4|\hat Z|^p
        +
        \frac1p\left(\frac\gamma4\right)^{-(p-1)}|Y|^p
    \right].
    \label{eq:cross-B-young-prop}
\end{align}
Thus this term subtracts
\[
    K_Y\left(\frac\gamma4\right)^{-(p-1)}
\]
from the \(Y\)-coefficient and
\[
    (p-1)K_Y\frac\gamma4
    =
    \frac{p-1}{2m}\gamma+\frac{p-1}{4\gamma}
\]
from the \(\hat Z\)-coefficient.

\textbf{The quadratic terms.}
The quadratic mixed terms satisfy
\begin{align}
    \frac m\gamma |Y|^{p-2}|\hat Z|^2
    +
    \frac{m(p-2)}{\gamma}
    |Y|^{p-4}\langle Y,\hat Z\rangle^2
    &\le
    \frac{m(p-1)}{\gamma}|Y|^{p-2}|\hat Z|^2 .
    \label{eq:corrected-bound-prop}
\end{align}
Applying Young's inequality with conjugates
\[
    r=\frac p{p-2},
    \qquad
    s=\frac p2,
\]
and \(\epsilon_2=1\), gives
\begin{equation}\label{eq:young-prop}
    \frac{m(p-1)}{\gamma}|Y|^{p-2}|\hat Z|^2
    \le
    \frac{m(p-1)}{\gamma}
    \left[
        \frac{p-2}{p}|Y|^p+\frac2p|\hat Z|^p
    \right].
\end{equation}
Thus this term subtracts
\[
    \frac{m(p-1)(p-2)}{p\gamma}
\]
from the \(Y\)-coefficient and
\[
    \frac{2m(p-1)}{p\gamma}
\]
from the \(\hat Z\)-coefficient.

\textbf{The noise trace.}
By \eqref{eq:trace_bound} and the elementary inequality
\[
    |Y_t^j+\Delta_\alpha|^2
    \le
    2|Y_t^j|^2+2|\Delta_\alpha|^2,
\]
the averaged noise trace is bounded by
\begin{align}
    \frac1J\sum_{j=1}^J \mathcal T_{\mathrm{noise}}^j(t)
    &\le
    \frac{p\sigma^2\chi_S}{m^2}
    \left[
        \frac1J\sum_{j=1}^J |\hat Z_t^j|^{p-2}|Y_t^j|^2
        +
        \left(\frac1J\sum_{j=1}^J |\hat Z_t^j|^{p-2}\right)
        \left(\frac1J\sum_{k=1}^J |Y_t^k|^2\right)
        +
        2|\Delta_\alpha|^2
        \frac1J\sum_{j=1}^J |\hat Z_t^j|^{p-2}
    \right].
    \label{eq:noise-trace-expanded-prop}
\end{align}
We estimate the three terms separately.

For the self-interaction term, Young's inequality with conjugates
\[
    r=\frac p{p-2},
    \qquad
    s=\frac p2
\]
gives, for any \(\epsilon_3>0\),
\begin{align}
    \frac1J\sum_{j=1}^J |\hat Z_t^j|^{p-2}|Y_t^j|^2
    &\le
    \frac{p-2}{p}\epsilon_3^{-\frac2{p-2}}
    \frac1J\sum_{j=1}^J |\hat Z_t^j|^p
    +
    \frac2p\epsilon_3
    \frac1J\sum_{j=1}^J |Y_t^j|^p.
    \label{eq:noise-self-young-prop}
\end{align}
For the cross-particle product, Jensen's inequality gives
\[
    \left(
        \frac1J\sum_{j=1}^J |\hat Z_t^j|^{p-2}
    \right)^{\frac p{p-2}}
    \le
    \frac1J\sum_{j=1}^J |\hat Z_t^j|^p,
\]
and
\[
    \left(
        \frac1J\sum_{j=1}^J |Y_t^j|^2
    \right)^{\frac p2}
    \le
    \frac1J\sum_{j=1}^J |Y_t^j|^p.
\]
Applying Young's inequality with the same \(\epsilon_3\), we obtain
\begin{align}
    &
    \left(\frac1J\sum_{j=1}^J |\hat Z_t^j|^{p-2}\right)
    \left(\frac1J\sum_{k=1}^J |Y_t^k|^2\right) \le
    \frac{p-2}{p}\epsilon_3^{-\frac2{p-2}}
    \frac1J\sum_{j=1}^J |\hat Z_t^j|^p
    +
    \frac2p\epsilon_3
    \frac1J\sum_{j=1}^J |Y_t^j|^p.
    \label{eq:noise-cross-young-prop}
\end{align}
Choose
\[
    \epsilon_3
    :=
    \frac{\mu_{\mathrm{gap},p}m^2}
    {64\gamma\sigma^2\chi_S}.
\]
Combining \eqref{eq:noise-self-young-prop} and
\eqref{eq:noise-cross-young-prop}, and multiplying by the prefactor
\(p\sigma^2\chi_S/m^2\), the combined spatial contribution gives the \(Y\)-penalty
\[
    \frac{\mu_{\mathrm{gap},p}}{16\gamma}
    \frac1J\sum_{j=1}^J |Y_t^j|^p,
\]
and the \(\hat Z\)-penalty
\begin{equation}\label{eq:noise-spatial-Z-penalty-prop}
    \frac{2(p-2)\sigma^2\chi_S}{m^2}
    \left(
        \frac{\mu_{\mathrm{gap},p}m^2}
        {64\gamma\sigma^2\chi_S}
    \right)^{-\frac2{p-2}}
    \frac1J\sum_{j=1}^J|\hat Z_t^j|^p.
\end{equation}

For the \(|\Delta_\alpha|^2\) contribution, Jensen's inequality gives
\[
    \frac1J\sum_{j=1}^J |\hat Z_t^j|^{p-2}
    \le
    \left(
        \frac1J\sum_{j=1}^J |\hat Z_t^j|^p
    \right)^{\frac{p-2}{p}}.
\]
Applying Young's inequality with
\[
    \epsilon_4
    :=
    \frac{\mu_{\mathrm{gap},p}m^2}
    {32\gamma\sigma^2\chi_S C_{\mathrm{lem}}},
\]
we obtain
\begin{align}
    &
    \frac{2p\sigma^2\chi_S}{m^2}
    |\Delta_\alpha|^2
    \frac1J\sum_{j=1}^J |\hat Z_t^j|^{p-2}
    \notag\\
    &\qquad\le
    \frac{2p\sigma^2\chi_S}{m^2}
    \frac{p-2}{p}
    \epsilon_4^{-\frac2{p-2}}
    \frac1J\sum_{j=1}^J|\hat Z_t^j|^p
    +
    \frac{2p\sigma^2\chi_S}{m^2}
    \frac2p\epsilon_4
    |\Delta_\alpha|^p
    \notag\\
    &\qquad=
    \frac{2(p-2)\sigma^2\chi_S}{m^2}
    \left(
        \frac{\mu_{\mathrm{gap},p}m^2}
        {32\gamma\sigma^2\chi_S C_{\mathrm{lem}}}
    \right)^{-\frac2{p-2}}
    \frac1J\sum_{j=1}^J|\hat Z_t^j|^p
    +
    \frac{\mu_{\mathrm{gap},p}}{8\gamma C_{\mathrm{lem}}}
    |\Delta_\alpha|^p .
    \label{eq:noise-Delta-young-prop}
\end{align}
By \(\cref{lem:m-malpha}\),
\[
    |\Delta_\alpha|^p
    \le
    C_{\mathrm{lem}}
    \frac1J\sum_{j=1}^J |Y_t^j|^p.
\]
Hence the last term in \eqref{eq:noise-Delta-young-prop} contributes
\[
    \frac{\mu_{\mathrm{gap},p}}{8\gamma}
    \frac1J\sum_{j=1}^J |Y_t^j|^p
\]
to the \(Y\)-penalty.

The two \(\hat Z\)-penalties in
\eqref{eq:noise-spatial-Z-penalty-prop} and \eqref{eq:noise-Delta-young-prop}
are bounded by
\begin{equation}\label{eq:noise-penalty-Lambda-p-prop}
    \Lambda_p
    \sigma^{\frac{2p}{p-2}}
    \gamma^{\frac2{p-2}}
    \frac1J\sum_{j=1}^J|\hat Z_t^j|^p,
\end{equation}
where \(\Lambda_p\) is the constant defined in \cref{assump:centered-decay}.

Combining all estimates gives the componentwise differential inequality
\eqref{eq:componentwise-Lp-differential}, with \(\widehat\lambda_Y\) and
\(\widehat\lambda_Z\) given by \eqref{eq:lambdaYhat-prop} and
\eqref{eq:lambdaZhat-prop}. By the spatial closure condition \eqref{eq:gamma-Y-condition},
\[
    \widehat\lambda_Y
    \ge
    \frac{\mu_{\mathrm{gap},p}}{4\gamma}
    +
    \frac{\mu_{\mathrm{gap},p}m(p-1)}{4\gamma^2}.
\]
Since
\[
    \lambda_p
    \left(
        a_p+\frac{m(p-1)}{p\gamma}
    \right)
    =
    \frac{\mu_{\mathrm{gap},p}}{4\gamma}
    \left(
        1-\frac{m(p-2)}{\gamma^2}
    \right)
    +
    \frac{\mu_{\mathrm{gap},p}m(p-1)}{4\gamma^2},
\]
we obtain \eqref{eq:Y-component-sufficient-prop}. Similarly, by
\eqref{eq:Z-closure},
\[
    \widehat\lambda_Z
    \ge
    \frac{p\mu_{\mathrm{gap},p}}{4\gamma}
    \left(
        1+\frac{m}{p\gamma}
    \right)
    =
    \lambda_p
    \left(
        1+\frac{m}{p\gamma}
    \right),
\]
which proves \eqref{eq:Z-component-sufficient-prop}.

Finally, Young's inequality gives
\[
    |Y|^{p-2}\langle Y,\hat Z\rangle
    \le
    \frac{p-1}{p}|Y|^p+\frac1p|\hat Z|^p.
\]
Thus
\begin{align}
    L_p(t)
    &\le
    \left(
        a_p+\frac{m(p-1)}{p\gamma}
    \right)
    \frac1J\sum_{j=1}^J |Y_t^j|^p
    +
    \left(
        1+\frac{m}{p\gamma}
    \right)
    \frac1J\sum_{j=1}^J |\hat Z_t^j|^p.
    \label{eq:Lp-upper-by-components-prop}
\end{align}
Combining \eqref{eq:componentwise-Lp-differential},
\eqref{eq:Y-component-sufficient-prop},
\eqref{eq:Z-component-sufficient-prop}, and
\eqref{eq:Lp-upper-by-components-prop}, we obtain
\[
    \d L_p(t)
    \le
    -\lambda_p L_p(t)\,\d t+\d M_{p,t}.
\]
This proves \eqref{eq:pathwise-Lp-differential}.

\proofstep{2. The quadratic case \(p=2\)}

For \(p=2\), write
\[
    L_2(t)
    =
    \left(\frac9{2m}+\frac1{\gamma^2}\right)Y_2(t)
    +
    \frac5\gamma C_2(t)
    +
    Z_2(t),
\]
where
\[
    Y_2(t):=\frac1J\sum_{j=1}^J |Y_t^j|^2,
    \qquad
    C_2(t):=\frac1J\sum_{j=1}^J \langle Y_t^j,\hat Z_t^j\rangle,
    \qquad
    Z_2(t):=\frac1J\sum_{j=1}^J |\hat Z_t^j|^2.
\]

First,
\[
    \d Y_2(t)
    =
    \left(
        2C_2(t)+\frac2\gamma Y_2(t)
    \right)\d t.
\]
Second, since \(Y_t^j\) has finite variation,
\begin{align}
    \d C_2(t)
    &=
    \left[
        Z_2(t)
        -
        \frac\gamma m C_2(t)
        -
        K_YY_2(t)
    \right]\d t
    +
    \frac1J\sum_{j=1}^J
    \left\langle Y_t^j,\d M_t^{j,\mathrm{center}}\right\rangle .
    \label{eq:C2-pathwise-diff}
\end{align}
Third, It\^o's formula and \(\cref{lem:qvar-centered-noise}\) give
\begin{align}
    \d Z_2(t)
    &\le
    \left[
        -2K_ZZ_2(t)-2K_YC_2(t)+K_\sigma Y_2(t)
    \right]\d t
    +
    \frac2J\sum_{j=1}^J
    \left\langle \hat Z_t^j,\d M_t^{j,\mathrm{center}}\right\rangle .
    \label{eq:Z2-pathwise-diff}
\end{align}
Here we used
\[
    \frac1J\sum_{j=1}^J |Y_t^j+\Delta_\alpha|^2
    =
    \frac1J\sum_{j=1}^J |Y_t^j|^2+|\Delta_\alpha|^2
    \le
    (1+C_{\mathrm{lem}})Y_2(t),
\]
so that
\[
    K_\sigma
    :=
    \frac{2\sigma^2\tau(S)(1+C_{\mathrm{lem}})}{m^2}.
\]

Combining the three differential identities,
\begin{align}
    \d L_2(t)
    &\le
    -
    \bigl(
        v_1Y_2(t)+v_2C_2(t)+v_3Z_2(t)
    \bigr)\d t
    +
    \d M_{2,t},
    \label{eq:L2-pathwise-before-v2}
\end{align}
where
\[
    \d M_{2,t}
    =
    \frac1J\sum_{j=1}^J
    \left[
        2\left\langle \hat Z_t^j,\d M_t^{j,\mathrm{center}}\right\rangle
        +
        \frac5\gamma
        \left\langle Y_t^j,\d M_t^{j,\mathrm{center}}\right\rangle
    \right],
\]
and
\begin{align}
    v_1
    &=
    \frac5\gamma K_Y
    -
    \frac2\gamma
    \left(
        \frac9{2m}+\frac1{\gamma^2}
    \right)
    -
    K_\sigma
    =
    \frac1{m\gamma}+\frac3{\gamma^3}-K_\sigma,
    \notag\\
    v_2
    &=
    -2\left(
        \frac9{2m}+\frac1{\gamma^2}
    \right)
    +
    \frac5m
    +
    2K_Y,
    \notag\\
    v_3
    &=
    2K_Z-\frac5\gamma
    =
    \frac{2\gamma}{m}-\frac3\gamma .
    \label{eq:p2-v123-prop}
\end{align}
Since \(K_Y=2/m+1/\gamma^2\), we have
\[
    v_2
    =
    -2\left(\frac9{2m}+\frac1{\gamma^2}\right)
    +
    \frac5m
    +
    2\left(\frac2m+\frac1{\gamma^2}\right)
    =
    0.
\]
Moreover, \(\gamma^2>3m/2\) gives \(v_3>0\), and
\(K_\sigma<1/(m\gamma)\) gives \(v_1>0\). Therefore
\[
    \d L_2(t)
    \le
    -
    v_1Y_2(t)\,\d t
    -
    v_3Z_2(t)\,\d t
    +
    \d M_{2,t}.
\]
By the upper norm-equivalence estimate for \(\psi_2\),
\[
    L_2(t)
    \le
    c_2^{(2)}
    \bigl(Y_2(t)+Z_2(t)\bigr).
\]
Hence
\[
    v_1Y_2(t)+v_3Z_2(t)
    \ge
    \frac{\min\{v_1,v_3\}}{c_2^{(2)}}L_2(t).
\]
Thus
\[
    \d L_2(t)
    \le
    -
    \frac{\min\{v_1,v_3\}}{c_2^{(2)}}L_2(t)\,\d t
    +
    \d M_{2,t}.
\]
This proves \eqref{eq:pathwise-Lp-differential}, and completes the proof.
\end{proof}

\begin{theorem}[Global exponential decay of centered moments]
\label{thm:exp-decay-centered-moments}
Let \(p\ge2\), and assume \cref{assump:centered-decay} for this value of \(p\).
Assume also that
\[
    \E\left[
        \momentp{p}{\emp{0}}
        +
        \momentp{p}{\empv{0}}
    \right]
    <\infty.
\]
Define
\begin{equation}\label{eq:lambda-p-cases}
    \edecay{p}
    :=
    \begin{cases}
        \displaystyle
        \frac{\min\{v_1,v_3\}}{c_2^{(2)}},
        & p=2,\\[1.2em]
        \displaystyle
        \frac{p\mu_{\mathrm{gap},p}}{4\gamma},
        & p>2.
    \end{cases}
\end{equation}
Then the expected centered Lyapunov functional decays exponentially:
\begin{equation}\label{eq:Lp-decay-main}
    \frac{\d}{\d t}\mathcal L_p(t)
    \le
    -\edecay{p}\mathcal L_p(t).
\end{equation}
Consequently,
\begin{equation}\label{eq:Lp-integrated-decay}
    \mathcal L_p(t)
    \le
    e^{-\edecay{p}t}\mathcal L_p(0).
\end{equation}
Moreover, there exists a finite constant \(C_p>0\), depending only on the parameters
and on the norm-equivalence constants, such that
\begin{equation}\label{eq:centered-moment-decay-conclusion}
    \E\left[
        \mathfrak M_p(\mu_{X_t^J})
        +
        \mathfrak M_p(\mu_{V_t^J})
    \right]
    \le
    C_p e^{-\edecay{p}t}
    \E\left[
        \mathfrak M_p(\mu_{X_0^J})
        +
        \mathfrak M_p(\mu_{V_0^J})
    \right].
\end{equation}
\end{theorem}

\begin{proof}
Let
\[
    L_p(t):=L_p(\mathcal X_t^J,\mathcal V_t^J),
    \qquad
    \mathcal L_p(t):=\E[L_p(t)].
\]
By \cref{prop:ito-centered-lyapunov-differential}, there exists a local martingale
\(M_{p,t}\) such that, in the sense of semimartingale differentials,
\[
    \d L_p(t)
    \le
    -\lambda_p L_p(t)\,\d t+\d M_{p,t},
\]
where
\[
    \lambda_p
    =
    \begin{cases}
        \displaystyle
        \frac{\min\{v_1,v_3\}}{c_2^{(2)}},
        & p=2,\\[1.2em]
        \displaystyle
        \frac{p\mu_{\mathrm{gap},p}}{4\gamma},
        & p>2.
    \end{cases}
\]
This is the decay rate \(\edecay{p}\) in \eqref{eq:lambda-p-cases}.

We first justify taking expectations. For \(R>0\), define the stopping time
\[
    \tau_R:=\inf\{t\ge0:L_p(t)\ge R\}.
\]
Applying the preceding differential inequality to the stopped process
\(L_p(t\wedge \tau_R)\), integrating from \(0\) to \(t\), and taking expectations, we get
\[
    \E[L_p(t\wedge\tau_R)]
    \le
    \E[L_p(0)]
    -
    \lambda_p
    \int_0^t
    \E\left[
        L_p(s)\mathbf 1_{\{s\le \tau_R\}}
    \right]\d s.
\]
Here the stopped martingale has zero expectation. Equivalently,
\[
    \E[L_p(t\wedge\tau_R)]
    +
    \lambda_p
    \int_0^t
    \E\left[
        L_p(s)\mathbf 1_{\{s\le \tau_R\}}
    \right]\d s
    \le
    \E[L_p(0)].
\]
Letting \(R\to\infty\), and using Fatou's lemma together with monotone convergence, yields
\[
    \mathcal L_p(t)
    +
    \lambda_p
    \int_0^t
    \mathcal L_p(s)\,\d s
    \le
    \mathcal L_p(0).
\]
Therefore,
\[
    \mathcal L_p(t)
    \le
    \mathcal L_p(0)
    -
    \lambda_p
    \int_0^t
    \mathcal L_p(s)\,\d s,
\]
or, equivalently in differential form,
\begin{equation}\label{eq:Lp-decay-main-proof}
    \frac{\d}{\d t}\mathcal L_p(t)
    \le
    -\lambda_p\mathcal L_p(t).
\end{equation}
This proves \eqref{eq:Lp-decay-main}. By Gr\"onwall's inequality,
\begin{equation}\label{eq:Lp-integrated-decay-proof}
    \mathcal L_p(t)
    \le
    e^{-\edecay{p}t}\mathcal L_p(0).
\end{equation}

It remains to translate the Lyapunov decay into decay of centered moments. By
\cref{cor:centered-lyapunov-moment-equivalence}, there exist constants
\(C_{p,1},C_{p,2}>0\) such that
\[
    C_{p,1}
    \E\left[
        \mathfrak M_p(\mu_{X_t^J})
        +
        \mathfrak M_p(\mu_{V_t^J})
    \right]
    \le
    \mathcal L_p(t)
    \le
    C_{p,2}
    \E\left[
        \mathfrak M_p(\mu_{X_t^J})
        +
        \mathfrak M_p(\mu_{V_t^J})
    \right].
\]
Combining this equivalence with \eqref{eq:Lp-integrated-decay-proof}, we obtain
\begin{align}
    \E\left[
        \mathfrak M_p(\mu_{X_t^J})
        +
        \mathfrak M_p(\mu_{V_t^J})
    \right]
    &\le
    C_{p,1}^{-1}\mathcal L_p(t) \le
    C_{p,1}^{-1}
    e^{-\edecay{p}t}
    \mathcal L_p(0)
    \notag\\
    &\le
    \frac{C_{p,2}}{C_{p,1}}
    e^{-\edecay{p}t}
    \E\left[
        \mathfrak M_p(\mu_{X_0^J})
        +
        \mathfrak M_p(\mu_{V_0^J})
    \right].
\end{align}
The conclusion follows with
\[
    C_p:=\frac{C_{p,2}}{C_{p,1}}.
\]
\end{proof}

\subsubsection{Decay of Centered Moments: Mean-Field Process}
\label{sec: dec-centered-moments-mf}

Here we prove the counterpart of \cref{thm:exp-decay-centered-moments} for the
mean-field process. Let \((\overline X_t,\overline V_t)\) solve
\begin{align}
    \d \overline X_t &= \overline V_t\,\d t, \\
    m\,\d \overline V_t
    &=
    -\gamma \overline V_t\,\d t
    -
    \bigl(\overline X_t-\mathcal M_\alpha(\overline\rho_t^X)\bigr)\d t
    +
    \sigma S\bigl(\overline X_t-\mathcal M_\alpha(\overline\rho_t^X)\bigr)\d W_t,
\end{align}
where \(\overline\rho_t=\Law(\overline X_t,\overline V_t)\) and
\(\overline\rho_t^X\) denotes its position marginal.

Define the centered-shifted variables
\[
    \overline Y_t
    :=
    \overline X_t-m_{\overline X}(\overline\rho_t),
    \qquad
    \widehat{\overline Z}_t
    :=
    \bigl(\overline V_t-m_{\overline V}(\overline\rho_t)\bigr)
    -
    \frac1\gamma \overline Y_t.
\]
For \(p>2\), set
\[
    a_p:=\frac1p\left(1-\frac{m(p-2)}{\gamma^2}\right),
\]
and define the pathwise mean-field Lyapunov functional
\begin{equation}\label{eq:mfl-random-Lp-def}
    \overline L_p(t)
    :=
    \phi_{a_p,p}\bigl(\overline Y_t,\widehat{\overline Z}_t\bigr),
    \qquad p>2.
\end{equation}
For \(p=2\), define
\begin{equation}\label{eq:mfl-random-L2-def}
    \overline L_2(t)
    :=
    \psi_2\bigl(\overline Y_t,\widehat{\overline Z}_t\bigr),
\end{equation}
where \(\psi_2\) is the dedicated quadratic form from \eqref{eq:psi2-centered}.
We denote the expected functionals by
\begin{equation}\label{eq:mfl-expected-Lp-def}
    \overline{\mathcal L}_p(t):=\E[\overline L_p(t)].
\end{equation}

For later use in the concentration estimate, if
\((\overline X_t^j,\overline V_t^j)_{j=1}^J\) are independent mean-field samples,
we write
\[
    \overline Y_t^j
    :=
    \overline X_t^j-m_{\overline X}(\overline\rho_t),
    \qquad
    \widehat{\overline Z}_t^j
    :=
    \bigl(\overline V_t^j-m_{\overline V}(\overline\rho_t)\bigr)
    -
    \frac1\gamma \overline Y_t^j.
\]
In the quadratic case, define
\begin{equation}\label{eq:mfl-sample-L2J-def}
    \overline L_2^j(t)
    :=
    \psi_2(\overline Y_t^j,\widehat{\overline Z}_t^j),
    \qquad
    \overline L_2^J(t)
    :=
    \frac1J\sum_{j=1}^J\overline L_2^j(t).
\end{equation}

\begin{proposition}[Mean-field It\^o inequalities for centered Lyapunov functionals]
\label{prop:ito-mfl-centered-lyapunov-differential}
Let \(p\ge2\), and assume \cref{assump:admissibility,assump:centered-decay} for this
value of \(p\). Then the expected mean-field centered Lyapunov functional satisfies
\begin{equation}\label{eq:mfl-expected-Lp-differential}
    \frac{\d}{\d t}\overline{\mathcal L}_p(t)
    \le
    -\edecay{p}\,\overline{\mathcal L}_p(t).
\end{equation}

Moreover, for the quadratic case \(p=2\), there exists a local martingale
\(\overline M_{2,t}^J\) such that the averaged pathwise functional
\(\overline L_2^J(t)\) satisfies
\begin{equation}\label{eq:mfl-averaged-L2-trace-remainder}
    \d \overline L_2^J(t)
    \le
    -\edecay{2}\,\overline L_2^J(t)\,\d t
    +
    K_{\mathrm{tr}}
    \left(
        \frac1J\sum_{j=1}^J
        \left|
            \overline X_t^j-\mathcal M_\alpha(\overline\rho_t^X)
        \right|^2
    \right)\d t
    +
    \d \overline M_{2,t}^J,
\end{equation}
where
\begin{equation}\label{eq:mfl-Ktr-def}
    K_{\mathrm{tr}}
    :=
    \frac{\sigma^2\tau(S)}{m^2}.
\end{equation}
\end{proposition}

\begin{proof}
Set
\[
    \Delta_\alpha^{\rm MF}(t)
    :=
    m_{\overline X}(\overline\rho_t)
    -
    \mathcal M_\alpha(\overline\rho_t^X).
\]
Since
\[
    \frac{\d}{\d t}m_{\overline X}(\overline\rho_t)
    =
    m_{\overline V}(\overline\rho_t),
\]
and
\[
    m\frac{\d}{\d t}m_{\overline V}(\overline\rho_t)
    =
    -\gamma m_{\overline V}(\overline\rho_t)
    -
    \bigl(
        m_{\overline X}(\overline\rho_t)
        -
        \mathcal M_\alpha(\overline\rho_t^X)
    \bigr),
\]
subtracting the mean equations from the mean-field SDE gives
\begin{equation}\label{eq:mfl-Y-dynamics}
    \d\overline Y_t
    =
    \left(
        \widehat{\overline Z}_t+\frac1\gamma\overline Y_t
    \right)\d t,
\end{equation}
and
\begin{equation}\label{eq:mfl-Zhat-dynamics}
    \d\widehat{\overline Z}_t
    =
    \left(
        -K_Z\widehat{\overline Z}_t
        -
        K_Y\overline Y_t
    \right)\d t
    +
    \frac\sigma m
    S\bigl(\overline X_t-\mathcal M_\alpha(\overline\rho_t^X)\bigr)\d W_t.
\end{equation}
Equivalently,
\[
    \overline X_t-\mathcal M_\alpha(\overline\rho_t^X)
    =
    \overline Y_t+\Delta_\alpha^{\rm MF}(t).
\]
Thus the deterministic part is the same centered-shifted linear drift as in the
particle case. The only difference is that the mean-field noise is not recentered by
subtracting an empirical average.

We first derive the pathwise quadratic estimate needed later in the concentration
argument. For each mean-field sample \(j\), applying It\^o's formula to
\[
    \overline L_2^j(t)
    =
    \psi_2(\overline Y_t^j,\widehat{\overline Z}_t^j)
\]
gives the same deterministic quadratic drift as in the \(p=2\) centered-moment
calculation. The quadratic variation contributes only through the
\(\widehat{\overline Z}\)-variable, since \(\overline Y\) has finite variation.
Hence
\begin{align}
    \d \overline L_2^j(t)
    &\le
    -
    v_1^{(0)}|\overline Y_t^j|^2\,\d t
    -
    v_3|\widehat{\overline Z}_t^j|^2\,\d t
    +
    \frac{\sigma^2}{m^2}
    \left\|
        S\!\left(
            \overline X_t^j-\mathcal M_\alpha(\overline\rho_t^X)
        \right)
    \right\|_{\mathrm{HS}}^2
    \d t
    +
    \d\overline M_{2,t}^j,
    \label{eq:mfl-single-L2-before-trace-bound}
\end{align}
where
\[
    v_1^{(0)}
    :=
    \frac1{m\gamma}+\frac3{\gamma^3},
    \qquad
    v_3
    :=
    \frac{2\gamma}{m}-\frac3\gamma,
\]
and
\begin{equation}\label{eq:mfl-single-L2-martingale}
    \d\overline M_{2,t}^j
    =
    \left\langle
        2\widehat{\overline Z}_t^j+\frac5\gamma\overline Y_t^j,
        \frac\sigma m
        S\!\left(
            \overline X_t^j-\mathcal M_\alpha(\overline\rho_t^X)
        \right)\d W_t^j
    \right\rangle .
\end{equation}
Using the trace bound
\[
    \|S(x)\|_{\mathrm{HS}}^2\le \tau(S)|x|^2,
\]
we get
\[
    \frac{\sigma^2}{m^2}
    \left\|
        S\!\left(
            \overline X_t^j-\mathcal M_\alpha(\overline\rho_t^X)
        \right)
    \right\|_{\mathrm{HS}}^2
    \le
    K_{\mathrm{tr}}
    \left|
        \overline X_t^j-\mathcal M_\alpha(\overline\rho_t^X)
    \right|^2.
\]
Moreover, with
\[
    v_1
    :=
    \frac1{m\gamma}+\frac3{\gamma^3}-K_\sigma,
    \qquad
    \edecay{2}
    =
    \frac{\min\{v_1,v_3\}}{c_2^{(2)}},
\]
we have \(v_1\le v_1^{(0)}\), and therefore
\[
    v_1^{(0)}|\overline Y_t^j|^2
    +
    v_3|\widehat{\overline Z}_t^j|^2
    \ge
    \min\{v_1,v_3\}
    \left(
        |\overline Y_t^j|^2
        +
        |\widehat{\overline Z}_t^j|^2
    \right)
    \ge
    \edecay{2}\,
    \psi_2(\overline Y_t^j,\widehat{\overline Z}_t^j).
\]
Consequently,
\[
    \d\overline L_2^j(t)
    \le
    -\edecay{2}\,\overline L_2^j(t)\,\d t
    +
    K_{\mathrm{tr}}
    \left|
        \overline X_t^j-\mathcal M_\alpha(\overline\rho_t^X)
    \right|^2\d t
    +
    \d\overline M_{2,t}^j.
\]
Averaging over \(j=1,\dots,J\) gives \eqref{eq:mfl-averaged-L2-trace-remainder}, with
\[
    \overline M_{2,t}^J
    :=
    \frac1J\sum_{j=1}^J\overline M_{2,t}^j.
\]

It remains to prove the expected decay estimate
\eqref{eq:mfl-expected-Lp-differential}. For \(p=2\), take expectations in
\eqref{eq:mfl-single-L2-before-trace-bound}. Since
\[
    \overline X_t-\mathcal M_\alpha(\overline\rho_t^X)
    =
    \overline Y_t+\Delta_\alpha^{\rm MF}(t),
\]
and \(\E[\overline Y_t]=0\), we have
\[
    \E\left|
        \overline X_t-\mathcal M_\alpha(\overline\rho_t^X)
    \right|^2
    =
    \E|\overline Y_t|^2
    +
    |\Delta_\alpha^{\rm MF}(t)|^2.
\]
By \cref{lem:m-malpha},
\[
    |\Delta_\alpha^{\rm MF}(t)|^2
    \le
    C_{\mathrm{lem}}\E|\overline Y_t|^2.
\]
Hence the trace contribution is bounded by
\[
    \frac{\sigma^2\tau(S)}{m^2}
    \E\left|
        \overline X_t-\mathcal M_\alpha(\overline\rho_t^X)
    \right|^2
    \le
    \frac{\sigma^2\tau(S)(1+C_{\mathrm{lem}})}{m^2}
    \E|\overline Y_t|^2
    \le
    K_\sigma\E|\overline Y_t|^2,
\]
where \(K_\sigma=\frac{2\sigma^2\tau(S)(1+C_{\mathrm{lem}})}{m^2}\) is the
quadratic noise constant from \eqref{eq:Ksigma-def}. Therefore the trace is absorbed
by the \(K_\sigma\)-budget in \(v_1\), and
\[
    \frac{\d}{\d t}\overline{\mathcal L}_2(t)
    \le
    -\edecay{2}\,\overline{\mathcal L}_2(t).
\]

For \(p>2\), applying It\^o's formula to
\(\phi_{a_p,p}(\overline Y_t,\widehat{\overline Z}_t)\) gives the same deterministic
drift expansion as in the particle proof. The martingale part is
\[
    \d\overline M_{p,t}
    =
    p|\widehat{\overline Z}_t|^{p-2}
    \left\langle
        \widehat{\overline Z}_t,
        \frac\sigma m
        S\bigl(\overline X_t-\mathcal M_\alpha(\overline\rho_t^X)\bigr)\d W_t
    \right\rangle
    +
    \frac m\gamma |\overline Y_t|^{p-2}
    \left\langle
        \overline Y_t,
        \frac\sigma m
        S\bigl(\overline X_t-\mathcal M_\alpha(\overline\rho_t^X)\bigr)\d W_t
    \right\rangle.
\]
The choice
\[
    a_p=\frac1p\left(1-\frac{m(p-2)}{\gamma^2}\right)
\]
cancels the mixed drift coefficient
\[
    a_pp-\frac m\gamma K_Z+\frac{m(p-1)}{\gamma^2}=0.
\]
The remaining mixed drift and the quadratic terms are estimated as in
\cref{prop:ito-centered-lyapunov-differential}. The mean-field noise trace satisfies
\[
    \overline{\mathcal T}_{\rm noise}(t)
    \le
    \frac{p\sigma^2\chi_S}{2m^2}
    |\widehat{\overline Z}_t|^{p-2}
    |\overline X_t-\mathcal M_\alpha(\overline\rho_t^X)|^2.
\]
Using
\[
    \overline X_t-\mathcal M_\alpha(\overline\rho_t^X)
    =
    \overline Y_t+\Delta_\alpha^{\rm MF}(t),
\]
and
\[
    |\overline Y_t+\Delta_\alpha^{\rm MF}(t)|^2
    \le
    2|\overline Y_t|^2+2|\Delta_\alpha^{\rm MF}(t)|^2,
\]
we obtain, after taking expectations, the same type of Young estimates as in the
particle centered-moment proof. The only mean-field-specific term is
\(\Delta_\alpha^{\rm MF}(t)\), which is deterministic at time \(t\) and satisfies
\[
    |\Delta_\alpha^{\rm MF}(t)|^p
    \le
    C_{\mathrm{lem}}\E|\overline Y_t|^p
\]
by \cref{lem:m-malpha}. Hence the same closure assumptions as in
\cref{prop:ito-centered-lyapunov-differential} yield
\[
    \frac{\d}{\d t}\overline{\mathcal L}_p(t)
    \le
    -\edecay{p}\,\overline{\mathcal L}_p(t),
    \qquad p>2.
\]
This proves \eqref{eq:mfl-expected-Lp-differential}.
\end{proof}

\begin{theorem}[Exponential decay of mean-field centered moments]
\label{thm:mfl-decay-centered-moment}
Let \(p\ge2\). Suppose that \(f\) satisfies \cref{assumption:bounded} and that
\(\overline\rho_0\in\mathcal P_p(\R^{2d})\). Assume
\cref{assump:admissibility,assump:centered-decay} for this value of \(p\). Then there
exists a finite constant \(C_p>0\), independent of \(t\), such that
\begin{equation}\label{eq:mfl-centered-moment-decay}
    \E\left[
        |\overline X_t-m_{\overline X}(\overline\rho_t)|^p
        +
        |\overline V_t-m_{\overline V}(\overline\rho_t)|^p
    \right]
    \le
    C_p e^{-\edecay{p}t}
    \E\left[
        |\overline X_0-m_{\overline X}(\overline\rho_0)|^p
        +
        |\overline V_0-m_{\overline V}(\overline\rho_0)|^p
    \right].
\end{equation}
\end{theorem}

\begin{proof}
By \cref{prop:ito-mfl-centered-lyapunov-differential},
\[
    \frac{\d}{\d t}\overline{\mathcal L}_p(t)
    \le
    -\edecay{p}\,\overline{\mathcal L}_p(t).
\]
Hence, Gr\"onwall's inequality gives
\begin{equation}\label{eq:mfl-Lp-integrated-decay-proof}
    \overline{\mathcal L}_p(t)
    \le
    e^{-\edecay{p}t}\overline{\mathcal L}_p(0).
\end{equation}

The norm-equivalence estimate from
\cref{cor:centered-lyapunov-moment-equivalence} applies identically to the
mean-field centered-shifted variables. Therefore there exist constants
\(C_{p,1},C_{p,2}>0\), depending only on the parameters, such that
\[
    C_{p,1}
    \E\left[
        |\overline X_t-m_{\overline X}(\overline\rho_t)|^p
        +
        |\overline V_t-m_{\overline V}(\overline\rho_t)|^p
    \right]
    \le
    \overline{\mathcal L}_p(t),
\]
and
\[
    \overline{\mathcal L}_p(0)
    \le
    C_{p,2}
    \E\left[
        |\overline X_0-m_{\overline X}(\overline\rho_0)|^p
        +
        |\overline V_0-m_{\overline V}(\overline\rho_0)|^p
    \right].
\]
Combining these two estimates with
\eqref{eq:mfl-Lp-integrated-decay-proof}, we obtain
\begin{align}
    \E\left[
        |\overline X_t-m_{\overline X}(\overline\rho_t)|^p
        +
        |\overline V_t-m_{\overline V}(\overline\rho_t)|^p
    \right] &\le
    C_{p,1}^{-1}\overline{\mathcal L}_p(t)
    \le
    C_{p,1}^{-1}e^{-\edecay{p}t}\overline{\mathcal L}_p(0)
    \notag\\
    &\le
    \frac{C_{p,2}}{C_{p,1}}
    e^{-\edecay{p}t}
    \E\left[
        |\overline X_0-m_{\overline X}(\overline\rho_0)|^p
        +
        |\overline V_0-m_{\overline V}(\overline\rho_0)|^p
    \right].
\end{align}
The conclusion follows with \(C_p:=C_{p,2}/C_{p,1}\).
\end{proof}

\begin{remark}[Single-sample form]
\label{rem:mfl-single-sample-centered-decay}
Since each mean-field sample \((\overline X_t^j,\overline V_t^j)\) has law
\(\overline\rho_t\), \cref{thm:mfl-decay-centered-moment} also gives, for every
\(j\in[J]\),
\[
    \E\left[
        |\overline X_t^j-m_{\overline X}(\overline\rho_t)|^p
        +
        |\overline V_t^j-m_{\overline V}(\overline\rho_t)|^p
    \right]
    \le
    C_p e^{-\edecay{p}t}\mathfrak M_p(\overline\rho_0).
\]
In particular,
\[
    \E\left[
        \frac1J\sum_{j=1}^J
        |\overline X_t^j-m_{\overline X}(\overline\rho_t)|^p
    \right]
    \le
    C_p e^{-\edecay{p}t}\mathfrak M_p(\overline\rho_0).
\]
\end{remark}


\subsection{Uniform-in-Time Boundedness of Raw Moments}\label{sec: raw-moments-bound}
While \cref{thm:exp-decay-centered-moments} establishes the exponential decay of the centered moments (and consequently the collapse of the particle cloud toward the mean), it does not a priori guarantee that the cloud itself does not drift to infinity. In this section, we prove that the raw moments $\E|X_t|^p$ and $\E|V_t|^p$ remain bounded uniformly in time.

To bound the position, we must first establish the exponential decay of the raw velocity. We achieve this by observing that the force driving the raw velocity depends entirely on the internal centered variables, which we already know decay exponentially.
\begin{proposition}[Exponential Decay of Raw Velocity]\label{prop:raw-velocity-decay}
    Under the assumptions of \cref{thm:exp-decay-centered-moments} for \(p\geq 2\), 
    assume in addition that \(\E|V_0^j|^p<\infty\). 
    Set $\eta_p := \frac{p\gamma}{2m}.$
    Then there exists a constant \(C_V>0\) such that, for all \(t\geq 0\),
    \[
        \E|V_t^j|^p
        \le
        \begin{cases}
            C_V \e^{-\edecay{p} t},
            & \text{if } \edecay{p} < \eta_p, \\[0.4em]
            C_V(1+t)\e^{-\eta_p t},
            & \text{if } \edecay{p} = \eta_p, \\[0.4em]
            C_V \e^{-\eta_p t},
            & \text{if } \edecay{p} > \eta_p.
        \end{cases}
    \]
    Equivalently, away from the critical case \(\edecay{p}=\eta_p\), the raw velocity decays at rate
    \[
        \edecay{p}^*:= \min\left\{\edecay{p},\eta_p\right\}
        =
        \min\left\{\edecay{p},\frac{p\gamma}{2m}\right\}.
    \]
\end{proposition}

\begin{proof}
    The dynamics of the raw velocity are given by:
    \begin{equation}
        m \d V_t^j = -\gamma V_t^j \d t - (X_t^j - \mathcal{M}_\alpha(\mu_{X_t^J})) \d t + \sigma S(X_t^j - \mathcal{M}_\alpha(\mu_{X_t^J})) \d W_t^j.
    \end{equation}
    We rewrite the driving force using the centered spatial variable $Y_t^j = X_t^j - m_X(\emp{t})$ and the drift error $\Delta_\alpha(t) = m_X(\emp{t}) - \mathcal{M}_\alpha(\mu_{X_t^J})$:
    \begin{equation}
        X_t^j - \mathcal{M}_\alpha(\mu_{X_t^J}) = Y_t^j + \Delta_\alpha(t) =: F_t^j.
    \end{equation}
    Applying It\^o's formula to $|V_t^j|^p$, we obtain:
    \begin{equation}
        \d |V_t^j|^p = p|V_t^j|^{p-2}\ip{V_t^j}{-\frac{\gamma}{m}V_t^j - \frac{1}{m}F_t^j} \d t + \mathcal{T}_{V}^j \d t + \d M_t^V,
    \end{equation}
    where the martingale term is $\d M_t^V = \frac{p\sigma}{m}|V_t^j|^{p-2}\langle V_t^j, S(F_t^j)\d W_t^j \rangle$. Using the trace property of $S$, the It\^o correction is bounded by $\mathcal{T}_V^j \le \frac{p \sigma^2 \chi_S}{2m^2} |V_t^j|^{p-2} |F_t^j|^2$. 
    
    Taking expectations to eliminate the martingale, we arrive at the differential inequality:
    \begin{equation} \label{eq:raw_vel_diff}
        \frac{\d}{\d t} \E|V_t^j|^p \le -\frac{p\gamma}{m} \E|V_t^j|^p + \frac{p}{m} \E\left[ |V_t^j|^{p-1}|F_t^j| \right] + \frac{p \sigma^2 \chi_S}{2m^2} \E\left[ |V_t^j|^{p-2}|F_t^j|^2 \right].
    \end{equation}
    We decouple the mixed terms using Young's inequality. For the drift term (using $r=\frac{p}{p-1}$, $s=p$) and the trace term (using $r=\frac{p}{p-2}$, $s=\frac{p}{2}$), we choose weights such that each penalty on $|V_t^j|^p$ absorbs $\frac{p\gamma}{4m}$ of the friction:
    \begin{align}
        \frac{p}{m} |V_t^j|^{p-1}|F_t^j| &\le \frac{p\gamma}{4m} |V_t^j|^p + K_1 |F_t^j|^p \,, \qquad  
        \frac{p \sigma^2 \chi_S}{2m^2} |V_t^j|^{p-2}|F_t^j|^2 \leq \frac{p\gamma}{4m} |V_t^j|^p + K_2 |F_t^j|^p,
    \end{align}
    where $K_1 := \frac{1}{m}\left(\frac{4(p-1)}{p\gamma}\right)^{p-1}$ and $K_2 := \frac{\sigma^2\chi_S}{m^2} \left(\frac{2(p-2)\sigma^2\chi_S}{p m\gamma}\right)^{\frac{p-2}{2}}$.
    Substituting these into \eqref{eq:raw_vel_diff} yields a net friction of $\eta := \frac{p\gamma}{2m}$:
    \begin{equation}
        \frac{\d}{\d t} \E|V_t^j|^p \le -\eta \E|V_t^j|^p + (K_1 + K_2) \E|F_t^j|^p.
    \end{equation}

    Because $F_t^j = Y_t^j + \Delta_\alpha$, we can bound its moments by the centered spatial moments. Using $|a+b|^p \le 2^{p-1}(|a|^p+|b|^p)$ and $\E|\Delta_\alpha (t)|^p \le C_{\mathrm{lem}}\E Y_p(t)$ (\cref{lem:m-malpha}), and $\E|Y_t^j|^p = \E [Y_p(t)]$ by exchangeability:
    \begin{equation}
        \E|F_t^j|^p \le 2^{p-1} \left( \E|Y_t^j|^p + \E|\Delta_\alpha|^p \right) \le 2^{p-1}(1 + C_{\mathrm{lem}}) \E[Y_p(t)].
    \end{equation}
    By \cref{thm:exp-decay-centered-moments} and \cref{cor:centered-lyapunov-moment-equivalence}, $\E[Y_p(t)] \le C_{p,1}^{-1}\mathcal L_p(0) \e^{-\edecay{p} t}$. Defining $D_F := (K_1 + K_2) 2^{p-1}(1 + C_{\mathrm{lem}}) C_{p,1}^{-1}\mathcal L_p(0)$, the inequality simplifies to:
    \begin{equation}
        \frac{\d}{\d t} \E|V_t^j|^p \le -\eta \E|V_t^j|^p + D_F \e^{-\edecay{p} t}.
    \end{equation}
    Applying Gr\"onwall's inequality gives
\[
    \E|V_t^j|^p
    \le 
    \E|V_0^j|^p \e^{-\eta_p t}
    +
    D_F \int_0^t \e^{-\eta_p(t-s)}\e^{-\edecay{p} s}\,ds .
\]
If \(\edecay{p}\neq \eta_p\), this becomes
\[
    \E|V_t^j|^p
    \le 
    \E|V_0^j|^p \e^{-\eta_p t}
    +
    D_F\frac{\e^{-\edecay{p} t}-\e^{-\eta_p t}}{\eta_p-\edecay{p}}.
\]
Therefore, if \(\edecay{p}<\eta_p\),
\[
    \E|V_t^j|^p
    \le 
    \left(
        \E|V_0^j|^p+\frac{D_F}{\eta_p-\edecay{p}}
    \right)\e^{-\edecay{p} t},
\]
whereas if \(\edecay{p}>\eta_p\),
\[
    \E|V_t^j|^p
    \le 
    \left(
        \E|V_0^j|^p+\frac{D_F}{\edecay{p}-\eta_p}
    \right)\e^{-\eta_p t}.
\]
In the critical case \(\edecay{p}=\eta_p\), we instead obtain
\[
    \E|V_t^j|^p
    \le 
    \left(\E|V_0^j|^p+D_Ft\right)\e^{-\eta_p t}
    \le C_V(1+t)\e^{-\eta_p t}.
\]
This proves the claimed estimate.
\end{proof}

\begin{proposition}[Exponential Decay of Mean-Field Raw Velocity]\label{prop:raw-velocity-decay-mf}
    Under the assumptions of \cref{thm:mfl-decay-centered-moment}, the raw velocity of the mean-field process decays exponentially. Specifically, using the same constants $C_V$ and $\edecay{p}$ from \cref{prop:raw-velocity-decay}, it holds that:
    \[
        \E|\overline{V}_t|^p \le C_V \e^{-\edecay{p}^* t}\,, \qquad \text{ where } \edecay{p}^*= \min\{ \edecay{p}, \eta_p\}\,.
    \]
\end{proposition}

\begin{proof}
    The dynamics of the mean-field raw velocity are governed by the SDE:
    \begin{equation}
        m \d \overline{V}_t = -\gamma \overline{V}_t \d t - (\overline{X}_t - \mathcal{M}_\alpha(\overline{\rho}_t)) \d t + \sigma S(\overline{X}_t - \mathcal{M}_\alpha(\overline{\rho}_t)) \d W_t.
    \end{equation}
    We define the mean-field driving force using the centered spatial variable $\overline{Y}_t = \overline{X}_t - m(\overline{\rho}_t)$ and the mean-field drift error $\overline{\Delta}_\alpha(t) = m(\overline{\rho}_t) - \mathcal{M}_\alpha(\overline{\rho}_t)$:
    \begin{equation}
        \overline{X}_t - \mathcal{M}_\alpha(\overline{\rho}_t) = \overline{Y}_t + \overline{\Delta}_\alpha(t) =: \overline{F}_t.
    \end{equation}
    Applying It\^o's formula to $|\overline{V}_t|^p$, we obtain the same structural differential as in the particle case:
    \begin{equation}
        \d |\overline{V}_t|^p = p|\overline{V}_t|^{p-2}\ip{\overline{V}_t}{-\frac{\gamma}{m}\overline{V}_t - \frac{1}{m}\overline{F}_t} \d t + \overline{\mathcal{T}}_{V} \d t + \d \overline{M}_t^V,
    \end{equation}
    where $\overline{\mathcal{T}}_V \le \frac{p \sigma^2 \chi_S}{2m^2} |\overline{V}_t|^{p-2} |\overline{F}_t|^2$. 
    
    Taking the expectation and applying the same Young's inequality decoupling steps (with identical weights $K_1(\gamma)$ and $K_2(\gamma, \sigma)$) as in \cref{prop:raw-velocity-decay}, we obtain the differential inequality:
    \begin{equation}
        \frac{\d}{\d t} \E|\overline{V}_t|^p \le -\eta_p \E|\overline{V}_t|^p + (K_1 + K_2) \E|\overline{F}_t|^p,
    \end{equation}
    where $\eta_p := \frac{p\gamma}{2m}$. 
    
    Using \cref{lem:m-malpha} for the mean-field measure $\overline{\rho}_t$, the force is bounded by the centered moments:
    \begin{equation}
        \E|\overline{F}_t|^p \le 2^{p-1} \left( \E|\overline{Y}_t|^p + \E|\overline{\Delta}_\alpha(t)|^p \right) \le 2^{p-1}(1 + C_{\mathrm{lem}}) \E|\overline{Y}_t|^p.
    \end{equation}
    By \cref{thm:mfl-decay-centered-moment} and the same norm equivalence, the mean-field centered moments decay as $\E|\overline{Y}_t|^p \le C_{p,1}^{-1}\overline{\mathcal L}_p(0)\e^{-\edecay{p} t}$. Because the particle system and the mean-field process are initialized from the identical probability measure $\overline{\rho}_0$, we have $\overline{\mathcal L}_p(0)=\mathcal L_p(0)$. 
    
    Consequently, the source term is bounded by the same exponential envelope $D_F \e^{-\edecay{p} t}$. Applying Gr\"onwall's inequality yields:
    \begin{equation}
        \E|\overline{V}_t|^p \le \E|\overline{V}_0|^p \e^{-\eta_p t} + \frac{D_F}{\eta_p - \edecay{p}} \e^{-\edecay{p} t} \le C_V \e^{-\edecay{p} t},
    \end{equation}
    which completes the proof using the identical constant $C_V$.
\end{proof}

To ensure the validity of the stability analysis and the well-posedness of the process on infinite horizons, we use this velocity decay to prove strong uniform-in-time bounds for the raw moments of both the particle system and the mean-field process. 

\begin{lemma}[Uniform-in-Time Moment Bound for Particles] \label{lem:uit-raw-moments-particle}
    Let $p \ge 2$. Under \cref{assumption:bounded}, \cref{assump:admissibility}, and the stability conditions of \cref{thm:exp-decay-centered-moments}, there exists a finite constant $K_{\mathrm{raw}} > 0$ such that:
    \begin{equation}
        \E \left[ \sup_{t \ge 0} |X_t^j|^p \right] \le K_{\mathrm{raw}} \quad \text{and} \quad \sup_{t \ge 0} \E \left[ |V_t^j|^p \right] \le K_{\mathrm{raw}}.
    \end{equation}
\end{lemma}

\begin{proof}
    \textbf{Step 1: Velocity bound.}
    From \cref{prop:raw-velocity-decay}, we have
    \[
        \E |V_t^j|^p \le C_V e^{-\edecay{p}t}
        \qquad \text{for all } t\ge 0.
    \]
    Hence
    \[
        \sup_{t\ge 0}\E |V_t^j|^p \le C_V
        =:K_{\mathrm{raw},V}.
    \]

    \textbf{Step 2: Position bound via finite-horizon localization.}
    Fix \(T>0\). Since
    \[
        X_t^j = X_0^j+\int_0^t V_s^j\,\d s,
        \qquad 0\le t\le T,
    \]
    we have the pathwise finite-horizon bound
    \[
        \sup_{0\le t\le T}|X_t^j|
        \le |X_0^j|+\int_0^T |V_s^j|\,\d s.
    \]
    Taking the \(L^p(\Omega)\)-norm and applying Minkowski's integral
    inequality on the finite interval \([0,T]\), we obtain
    \[
        \left\|\sup_{0\le t\le T}|X_t^j|\right\|_{L^p}
        \le
        \|X_0^j\|_{L^p}
        +
        \int_0^T \|V_s^j\|_{L^p}\,\d s .
    \]
    Using the exponential velocity decay,
    \[
        \|V_s^j\|_{L^p}
        \le C_V^{1/p} e^{-\frac{\edecay{p}}{p}s},
    \]
    we get the uniform-in-\(T\) estimate
    \[
        \left\|\sup_{0\le t\le T}|X_t^j|\right\|_{L^p}
        \le
        \|X_0^j\|_{L^p}
        +
        C_V^{1/p}
        \int_0^T e^{-\frac{\edecay{p}}{p}s}\,\d s
        \le
        \|X_0^j\|_{L^p}
        +
        \frac{p}{\edecay{p}}C_V^{1/p}.
    \]
    Therefore,
    \[
        \E\left[\sup_{0\le t\le T}|X_t^j|^p\right]
        \le
        \left(
            (\E|X_0^j|^p)^{1/p}
            +
            \frac{p}{\edecay{p}}C_V^{1/p}
        \right)^p
        =:K_{\mathrm{raw},X}
    \]
    for every \(T>0\).

    Finally, since
    \[
        \sup_{0\le t\le T}|X_t^j|
        \uparrow
        \sup_{t\ge 0}|X_t^j|
        \qquad \text{as } T\to\infty,
    \]
    the monotone convergence theorem yields
    \[
        \E\left[\sup_{t\ge 0}|X_t^j|^p\right]
        =
        \lim_{T\to\infty}
        \E\left[\sup_{0\le t\le T}|X_t^j|^p\right]
        \le K_{\mathrm{raw},X}.
    \]
    Setting
    \[
        K_{\mathrm{raw}}
        :=
        \max\{K_{\mathrm{raw},X},K_{\mathrm{raw},V}\},
    \]
    completes the proof.
\end{proof}

\begin{lemma}[Uniform-in-time raw moment bounds for the mean-field process]
\label{lem:uit-raw-moments-mf}
Let \(p\ge 2\). Under the same assumptions as in
\cref{lem:uit-raw-moments-particle}, there exists a finite constant
\(K_{\mathrm{raw}}>0\), possibly enlarged from the constant in
\cref{lem:uit-raw-moments-particle}, such that the mean-field process
\((\overline X_t,\overline V_t)\) satisfies
\begin{equation}
    \E\left[\sup_{t\ge 0} |\overline X_t|^p\right]
    \le K_{\mathrm{raw}},
    \qquad
    \sup_{t\ge 0}\E|\overline V_t|^p
    \le K_{\mathrm{raw}}.
\end{equation}
\end{lemma}

\begin{proof}
    From \cref{prop:raw-velocity-decay-mf}, there exists a constant
    \(\overline C_V>0\) such that
    \[
        \E|\overline V_t|^p
        \le \overline C_V e^{-\edecay{p}t}
        \qquad \text{for all } t\ge 0.
    \]
    Hence
    \[
        \sup_{t\ge 0}\E|\overline V_t|^p
        \le \overline C_V
        =: K_{\mathrm{raw},V}^{\mathrm{mf}}.
    \]

    We next prove the spatial bound. Fix \(T>0\). By the mean-field
    kinematics,
    \[
        \overline X_t
        =
        \overline X_0+\int_0^t \overline V_s\,\d s,
        \qquad 0\le t\le T.
    \]
    Therefore,
    \[
        \sup_{0\le t\le T}|\overline X_t|
        \le
        |\overline X_0|
        +
        \int_0^T |\overline V_s|\,\d s.
    \]
    Taking the \(L^p(\Omega)\)-norm and applying Minkowski's integral
    inequality on the finite interval \([0,T]\), we obtain
    \[
        \left\|
            \sup_{0\le t\le T}|\overline X_t|
        \right\|_{L^p}
        \le
        \|\overline X_0\|_{L^p}
        +
        \int_0^T \|\overline V_s\|_{L^p}\,\d s.
    \]
    Using the exponential decay estimate for \(\overline V_s\),
    \[
        \|\overline V_s\|_{L^p}
        \le
        \overline C_V^{1/p}
        e^{-\frac{\edecay{p}}{p}s},
    \]
    we get the uniform-in-\(T\) bound
    \[
        \left\|
            \sup_{0\le t\le T}|\overline X_t|
        \right\|_{L^p}
        \le
        \|\overline X_0\|_{L^p}
        +
        \overline C_V^{1/p}
        \int_0^T e^{-\frac{\edecay{p}}{p}s}\,\d s
        \le
        \|\overline X_0\|_{L^p}
        +
        \frac{p}{\edecay{p}}\overline C_V^{1/p}.
    \]
    Raising both sides to the power \(p\), we have for every \(T>0\),
    \[
        \E\left[
            \sup_{0\le t\le T}|\overline X_t|^p
        \right]
        \le
        \left(
            \|\overline X_0\|_{L^p}
            +
            \frac{p}{\edecay{p}}\overline C_V^{1/p}
        \right)^p
        =:K_{\mathrm{raw},X}^{\mathrm{mf}}.
    \]
    Since
    \[
        \sup_{0\le t\le T}|\overline X_t|^p
        \uparrow
        \sup_{t\ge 0}|\overline X_t|^p
        \qquad \text{as } T\to\infty,
    \]
    the monotone convergence theorem yields
    \[
        \E\left[
            \sup_{t\ge 0}|\overline X_t|^p
        \right]
        =
        \lim_{T\to\infty}
        \E\left[
            \sup_{0\le t\le T}|\overline X_t|^p
        \right]
        \le
        K_{\mathrm{raw},X}^{\mathrm{mf}}.
    \]

    Finally, after enlarging \(K_{\mathrm{raw}}\) if necessary, we take
    \[
        K_{\mathrm{raw}}
        :=
        \max\left\{
            K_{\mathrm{raw}}^{\mathrm{particle}},
            K_{\mathrm{raw},X}^{\mathrm{mf}},
            K_{\mathrm{raw},V}^{\mathrm{mf}}
        \right\},
    \]
    where \(K_{\mathrm{raw}}^{\mathrm{particle}}\) denotes the constant from
    \cref{lem:uit-raw-moments-particle}. This gives the claimed bounds.
\end{proof}

\begin{remark}[Sufficiency of Eighth-Order Spatial Bounds and Role of Velocity]
    While the lemmas above establish uniform-in-time bounds for raw moments of any arbitrary order $p \ge 2$, the final propagation of chaos proof  requires these bounds up to the spatial eighth order ($p=8$). One might initially expect that only $4$-th order spatial moments are needed, since the geometric error $\mathfrak{M}_2 \mathcal{W}_2^2$ is only of degree 4 in $X$. However, to isolate the probability of the rare "Bad Set" in Section 3.1, we must apply the Cauchy-Schwarz inequality. This squares the 4th-order geometric error, yielding $\E[(\mathfrak{M}_2^X + \mathfrak{M}_2^{\overline{X}})^2 \mathcal{W}_2^4]$. Bounding this squared term uniformly in time thus  requires $\sup_{t \ge 0}\E|X_t|^8 < \infty$. 
    
    Furthermore, the propagation of chaos theorem does not  depend on the raw velocity bounds. It is important to emphasize that \cref{prop:raw-velocity-decay} does not represent an additional physical constraint on the system; rather, it emerges naturally as a byproduct of the centered moment decay. Because the position process lacks direct dissipation ($\d X_t = V_t \d t$), we must mathematically route our analysis through the velocity. We establish the exponential decay of the raw velocity as a necessary intermediate analytical tool, which then allows us to  bound the spatial excursions via Minkowski's integral inequality.
\end{remark}

\subsection{Concentration Inequalities}\label{sec:concentration}

\subsubsection{Concentration Inequalities for the Interacting Particle Systems}

To close the stability estimate for the weighted mean, we require a high-probability
bound on the particle system staying within a compact region. We leverage the
structure of the pathwise quadratic functional \(L_2(t)\) to prove a concentration
inequality analogous to the first-order case.

\begin{lemma}[Bound on probability of large excursions]
\label{lem:concentration-ineq}
    Let the objective function \(f\) satisfy
    \cref{assumption:bounded,assumption:lip}, and let \(q>2\). Consider the
    second-order CBO dynamics where the initial positions and velocities are sampled
    i.i.d.\ from some \(\overline\rho_0\in\mathcal P_{2q}(\mathbb R^{2d})\).
    Assume that the parameter conditions in
    \cref{prop:ito-centered-lyapunov-differential} hold for \(p=2\) and \(p=2q\).
    Let
    \[
        0<\kappa<\min\left\{\edecay{2},\frac{\edecay{2q}}q\right\}.
    \]
    For brevity, write
    \[
        c_{1,2q}:=c_1(a_{2q},2q).
    \]
    Then there exists a finite constant \(C_{\mathrm{Bad}}>0\) such that, for any
    threshold \(A>0\) and \(J\in\mathbb N_+\), the pathwise empirical Lyapunov
    functional \(L_2(t)\) satisfies
    \begin{equation}\label{eq:particle-concentration-L2}
        \mathbb P\left[
            \sup_{t\ge0}
            \e^{\kappa t} L_2(t)
            \ge
            \mathcal L_2(0)+A
        \right]
        \le
        \frac{C_{\mathrm{Bad}}}{A^q}
        J^{-\frac q2}\,
        \mathcal L_{2q}(0).
    \end{equation}
    Consequently, by norm equivalence result \eqref{eq:psi2-norm-equivalence},
    \[
        c_1^{(2)}\mathfrak M_2(\mu_{X_t^J})\le L_2(t),
    \]
    the empirical spatial moments satisfy
    \begin{equation}\label{eq:particle-concentration-spatial-moment}
        \mathbb P\left[
            \sup_{t\ge0}
            \e^{\kappa t}
            \mathfrak M_2(\mu_{X_t^J})
            \ge
            \frac{\mathcal L_2(0)+A}{c_1^{(2)}}
        \right]
        \le
        \frac{C_{\mathrm{Bad}}}{A^q}
        J^{-\frac q2}\,
        \mathcal L_{2q}(0).
    \end{equation}
    In the above, \(C_{\mathrm{Bad}}:=C_{\mathrm{Bad},q,\kappa}\) is given by
    \begin{align}
        C_{\mathrm{Bad},q,\kappa}
        &:=
        2^q\widetilde C_{\mathrm{MZ},q}
        +
        C_{\mathrm{Mart}},
        \\
        \widetilde C_{\mathrm{MZ},q}
        &:=
        \frac{
            2^{3q-2}C_\phi^q(\cmz{q}+\cmz{2q})
        }{
            c_{1,2q}
        },
        \\
        C_{\mathrm{Mart}}
        &:=
        \frac{
            2^{q+1}C_{\mathrm{BDG},q}C_{\sigma,q}
        }{
            \edecay{2q}-q\kappa
        }
        \left(
            \frac{q-2}{\edecay{2q}-q\kappa}
        \right)^{\frac q2-1}.
    \end{align}
\end{lemma}

\begin{proof}
    We define the exponentially weighted process
    \[
        U_t:=\e^{\kappa t}L_2(t).
    \]
    Recall that \(L_2(t)\) is the pathwise centered quadratic Lyapunov functional
    defined by
    \[
        L_2(t)
        =
        \frac1J\sum_{j=1}^J
        \psi_2(Y_t^j,\hat Z_t^j).
    \]
    By \cref{prop:ito-centered-lyapunov-differential}, there exists a local
    martingale \(M_{2,t}\) such that
    \begin{equation}\label{eq:particle-L2-pathwise-ineq}
        \d L_2(t)
        \le
        -\edecay{2}L_2(t)\,\d t
        +
        \d M_{2,t}.
    \end{equation}
    More explicitly,
    \begin{equation}\label{eq:particle-concentration-martingale}
        \d M_{2,t}
        =
        \frac1J\sum_{j=1}^J
        \left\langle
            2\hat Z_t^j+\frac5\gamma Y_t^j,
            \frac{\sigma}{m}
            S\!\left(
                X_t^j-\mathcal M_\alpha(\mu_{X_t^J})
            \right)\d W_t^j
        \right\rangle.
    \end{equation}
    Here the centered noise subtraction drops out because
    \[
        \sum_{j=1}^J
        \left(
            2\hat Z_t^j+\frac5\gamma Y_t^j
        \right)
        =
        0.
    \]
    Applying It\^o's formula to \(U_t\), we get
    \begin{align}
        \d U_t
        &=
        \kappa\e^{\kappa t}L_2(t)\,\d t
        +
        \e^{\kappa t}\d L_2(t)
        \notag\\
        &\le
        \e^{\kappa t}
        (\kappa-\edecay{2})
        L_2(t)\,\d t
        +
        \e^{\kappa t}\d M_{2,t}.
    \end{align}
    Since \(\kappa<\edecay{2}\), the drift term is non-positive. Integrating gives
    \[
        U_t
        \le
        U_0+\int_0^t\e^{\kappa s}\d M_{2,s}
        =:
        U_0+\mathcal M_t^{\exp}.
    \]
    We bound the supremum of the process in two parts: the initial deviation and the
    martingale fluctuation:
    \begin{equation}\label{eq:proba-split-concent}
        \mathbb P\left[
            \sup_{t\ge0}U_t
            \ge
            \mathcal L_2(0)+A
        \right]
        \le
        \mathbb P\left[
            U_0-\mathcal L_2(0)\ge\frac A2
        \right]
        +
        \mathbb P\left[
            \sup_{t\ge0}\mathcal M_t^{\exp}\ge\frac A2
        \right].
    \end{equation}

    \proofstep{Part 1: Initial data concentration}
    The initial value $U_0=L_2(0)$ is defined using the empirical centered variables
    \[
        Y_0^j=X_0^j-m_X(\mu_{X_0^J}),
        \qquad
        \hat Z_0^j
        =
        \bigl(V_0^j-m_V(\mu_{V_0^J})\bigr)
        -
        \frac1\gamma Y_0^j.
    \]
    Let
    \[
        Z_0^j:=V_0^j-\frac1\gamma X_0^j.
    \]
    Then \(\hat Z_0^j=Z_0^j-m_Z(\mu_{Z_0^J})\). Let
    \(\tilde\phi\) be the quadratic form associated with \(\psi_2\):
    \[
        \tilde\phi(y,z)
        :=
        a_2|y|^2+|z|^2+\frac5\gamma\langle y,z\rangle.
    \]
    Using the generalized Huygens identity, we decompose \(U_0\) relative to the
    true expectations
    \[
        \bar x:=\E[X_0^1],
        \qquad
        \bar z:=\E[Z_0^1].
    \]
    Thus
    \begin{align}
        U_0
        &=
        \frac1J\sum_{j=1}^J
        \tilde\phi\bigl(
            X_0^j-m_X(\mu_{X_0^J}),
            Z_0^j-m_Z(\mu_{Z_0^J})
        \bigr)
        \notag\\
        &=
        \underbrace{
        \frac1J\sum_{j=1}^J
        \tilde\phi(X_0^j-\bar x,Z_0^j-\bar z)
        }_{=:\widetilde U_0}
        -
        \tilde\phi\bigl(
            m_X(\mu_{X_0^J})-\bar x,
            m_Z(\mu_{Z_0^J})-\bar z
        \bigr).
    \end{align}
    We bound the \(L^q\)-norm of the centered variable
    \(U_0-\mathcal L_2(0)\). Since $\mathcal L_2(0)=\E[U_0]\,,$ and
    \[
        \E[U_0]
        =
        \E[\widetilde U_0]
        -
        \E\left[
            \tilde\phi\bigl(
                m_X(\mu_{X_0^J})-\bar x,
                m_Z(\mu_{Z_0^J})-\bar z
            \bigr)
        \right],
    \]
    we have
    \begin{align}
        \E\left[
            |U_0-\mathcal L_2(0)|^q
        \right]
        &\le
        2^{q-1}
        \E\left[
            |\widetilde U_0-\E[\widetilde U_0]|^q
        \right]
        +
        2^{q-1}
        \E\left[
            \left|
                \mathcal R_0-\E[\mathcal R_0]
            \right|^q
        \right],
    \end{align}
    where
    \[
        \mathcal R_0
        :=
        \tilde\phi\bigl(
            m_X(\mu_{X_0^J})-\bar x,
            m_Z(\mu_{Z_0^J})-\bar z
        \bigr).
    \]

    \textbf{Term 1 \((\widetilde U_0)\).}
    Since \(\widetilde U_0\) is an average of i.i.d.\ random variables
    \[
        \xi_j
        :=
        \tilde\phi(X_0^j-\bar x,Z_0^j-\bar z),
    \]
    the Marcinkiewicz--Zygmund inequality gives
    \begin{align}
        \E\left[
            |\widetilde U_0-\E[\widetilde U_0]|^q
        \right]
        &\le
        \cmz{q}J^{-q/2}
        \E\left[
            |\xi_1-\E[\xi_1]|^q
        \right] \le
        2^q\cmz{q}J^{-q/2}
        \E[|\xi_1|^q].
    \end{align}
    Since \(\xi_1\) is a quadratic form,
    \[
        |\xi_1|^q
        \le
        C_\phi^q2^{q-1}
        \left(
            |X_0^1-\bar x|^{2q}
            +
            |Z_0^1-\bar z|^{2q}
        \right).
    \]
    These true-centered moments are controlled by the initial \(2q\)-moment entering
    \(\mathcal L_{2q}(0)\). Hence
    \[
        \E[|\xi_1|^q]
        \le
        \frac{C_\phi^q2^{q-1}}{c_{1,2q}}\,
        \mathcal L_{2q}(0).
    \]

    \textbf{Term 2 (empirical mean error).}
    Using
    \[
        |\tilde\phi(u,v)|
        \le
        C_\phi(|u|^2+|v|^2),
    \]
    we obtain
    \begin{align}
        \E\left[
            |\mathcal R_0-\E[\mathcal R_0]|^q
        \right]
        &\le
        2^q\E[|\mathcal R_0|^q]
        \notag\\
        &\le
        2^qC_\phi^q2^{q-1}
        \left(
            \E\left[
                |m_X(\mu_{X_0^J})-\bar x|^{2q}
            \right]
            +
            \E\left[
                |m_Z(\mu_{Z_0^J})-\bar z|^{2q}
            \right]
        \right).
    \end{align}
    Applying the MZ inequality to the empirical means and using the same coercivity
    control by \(\mathcal L_{2q}(0)\), we get
    \[
        \E\left[
            |\mathcal R_0-\E[\mathcal R_0]|^q
        \right]
        \le
        \frac{
            2^{2q-1}C_\phi^q\cmz{2q}
        }{
            c_{1,2q}
        }
        J^{-q}
        \mathcal L_{2q}(0).
    \]
    Since \(q\ge2\), \(J^{-q}\le J^{-q/2}\). Combining Term 1 and Term 2 gives
    \begin{equation}\label{eq:L2_initial_MZ_final}
        \E\left[
            |U_0-\mathcal L_2(0)|^q
        \right]
        \le
        \widetilde C_{\mathrm{MZ},q}
        J^{-q/2}
        \mathcal L_{2q}(0),
    \end{equation}
    where
    \begin{equation}
        \widetilde C_{\mathrm{MZ},q}
        :=
        \frac{
            2^{3q-2}C_\phi^q(\cmz{q}+\cmz{2q})
        }{
            c_{1,2q}
        }.
    \end{equation}
    By Markov's inequality at threshold \(A/2\),
    \begin{equation}
        \mathbb P\left[
            U_0-\mathcal L_2(0)\ge\frac A2
        \right]
        \le
        \frac{
            2^q\widetilde C_{\mathrm{MZ},q}
        }{
            A^q
        }
        J^{-q/2}\mathcal L_{2q}(0).
    \end{equation}

    \proofstep{Part 2: Martingale concentration via BDG}
    We bound the moments of the supremum of \(\mathcal M_t^{\exp}\) using the
    Burkholder--Davis--Gundy inequality. Define the effective volatility process for
    particle \(j\):
    \[
        \sigma_j(s)
        :=
        \e^{\kappa s}
        \frac{\sigma}{m}
        S\!\left(
            X_s^j-\mathcal M_\alpha(\mu_{X_s^J})
        \right)^\top
        \left(
            2\hat Z_s^j+\frac5\gamma Y_s^j
        \right).
    \]
    Then
    \[
        \mathcal M_t^{\exp}
        =
        \frac1J\sum_{j=1}^J
        \int_0^t
        \langle \sigma_j(s),\d W_s^j\rangle.
    \]
    By \cite[Lemma~5.8]{uit_cbo},
    \begin{equation}\label{eq:bdg_application}
        \E\left[
            \sup_{s\in[0,t]}
            |\mathcal M_s^{\exp}|^q
        \right]
        \le
        \frac{C_{\mathrm{BDG},q}}{J^{q/2}\ell^{\frac q2-1}}
        \frac1J
        \sum_{j=1}^J
        \int_0^t
        \e^{(\frac q2-1)\ell s}
        \E[|\sigma_j(s)|^q]\,\d s.
    \end{equation}
    For both isotropic and anisotropic interactions \(S\), the vector norm satisfies
    \[
        |S(x)^\top v|^2\le |x|^2|v|^2.
    \]
    Hence
    \[
        |\sigma_j(s)|^q
        \le
        \e^{q\kappa s}
        \left(\frac{\sigma}{m}\right)^q
        \left|
            2\hat Z_s^j+\frac5\gamma Y_s^j
        \right|^q
        \left|
            X_s^j-\mathcal M_\alpha(\mu_{X_s^J})
        \right|^q.
    \]
    By H\"older's inequality with exponents \((2,2)\),
    \begin{align}
        \E[|\sigma_j(s)|^q]
        &\le
        \e^{q\kappa s}
        \left(\frac{\sigma}{m}\right)^q
        \left(
            \E\left[
                \left|
                    2\hat Z_s^j+\frac5\gamma Y_s^j
                \right|^{2q}
            \right]
        \right)^{1/2} 
        \left(
            \E\left[
                \left|
                    X_s^j-\mathcal M_\alpha(\mu_{X_s^J})
                \right|^{2q}
            \right]
        \right)^{1/2}.
    \end{align}
    Using
    \[
        |u+v|^{2q}
        \le
        2^{2q-1}(|u|^{2q}+|v|^{2q})
    \]
    and norm equivalence, the first factor is bounded by
    \[
        \E\left[
            \left|
                2\hat Z_s^j+\frac5\gamma Y_s^j
            \right|^{2q}
        \right]
        \le
        \frac{
            2^{2q-1}\max\left\{2^{2q},(5/\gamma)^{2q}\right\}
        }{
            c_{1,2q}
        }
        \mathcal L_{2q}(s).
    \]
    Using \cref{lem:m-malpha}, the second factor is bounded by
    \[
        \E\left[
            \left|
                X_s^j-\mathcal M_\alpha(\mu_{X_s^J})
            \right|^{2q}
        \right]
        \le
        \frac{
            2^{2q-1}
            \left(
                1+\e^{\alpha(\overline f-\underline f)}
            \right)
        }{
            c_{1,2q}
        }
        \mathcal L_{2q}(s).
    \]
    Substituting these bounds into the H\"older estimate and using
    \[
        \mathcal L_{2q}(s)
        \le
        \e^{-\edecay{2q}s}\mathcal L_{2q}(0)
    \]
    from \cref{thm:exp-decay-centered-moments}, we get
    \begin{equation}
        \E[|\sigma_j(s)|^q]
        \le
        C_{\sigma,q}
        \e^{(q\kappa-\edecay{2q})s}
        \mathcal L_{2q}(0),
    \end{equation}
    where
    \begin{equation}\label{eq:C_sigma_def}
        C_{\sigma,q}
        :=
        \left(\frac{\sigma}{m}\right)^q
        \frac{
            2^{2q-1}
            \max\left\{2^q,(5/\gamma)^q\right\}
        }{
            c_{1,2q}
        }
        \left(
            1+\e^{\alpha(\overline f-\underline f)}
        \right)^{1/2}.
    \end{equation}
    Since \(q\kappa<\edecay{2q}\), choose
    \[
        \ell
        =
        \frac{\edecay{2q}-q\kappa}{q-2}
    \]
    so that
    \[
        \left(\frac q2-1\right)\ell
        =
        \frac{\edecay{2q}-q\kappa}{2}.
    \]
    The time integral resolves to
    \begin{align}
        \E\left[
            \sup_{s\in[0,t]}
            |\mathcal M_s^{\exp}|^q
        \right]
        &\le
        \frac{
            C_{\mathrm{BDG},q}
        }{
            J^{q/2}\ell^{\frac q2-1}
        }
        C_{\sigma,q}\mathcal L_{2q}(0)
        \int_0^t
        \e^{\frac{q\kappa-\edecay{2q}}2s}\,\d s
        \notag\\
        &\le
        \frac{
            2C_{\mathrm{BDG},q}C_{\sigma,q}
        }{
            J^{q/2}(\edecay{2q}-q\kappa)
        }
        \left(
            \frac{q-2}{\edecay{2q}-q\kappa}
        \right)^{\frac q2-1}
        \mathcal L_{2q}(0).
    \end{align}
    Applying Markov's inequality to the martingale event at threshold \(A/2\),
    \[
        \mathbb P\left[
            \sup_{t\ge0}
            \mathcal M_t^{\exp}
            \ge
            \frac A2
        \right]
        \le
        \frac{C_{\mathrm{Mart}}}{J^{q/2}A^q}
        \mathcal L_{2q}(0),
    \]
    where
    \[
        C_{\mathrm{Mart}}
        :=
        \frac{
            2^{q+1}C_{\mathrm{BDG},q}C_{\sigma,q}
        }{
            \edecay{2q}-q\kappa
        }
        \left(
            \frac{q-2}{\edecay{2q}-q\kappa}
        \right)^{\frac q2-1}.
    \]

    \proofstep{Conclusion}
    Combining the initial data concentration and the martingale bound, we obtain
    \[
        \mathbb P\left[
            \sup_{t\ge0}U_t
            \ge
            \mathcal L_2(0)+A
        \right]
        \le
        \frac{C_{\mathrm{Bad}}}{J^{q/2}A^q}
        \mathcal L_{2q}(0),
    \]
    where
    \[
        C_{\mathrm{Bad}}
        :=
        2^q\widetilde C_{\mathrm{MZ},q}
        +
        C_{\mathrm{Mart}}.
    \]
\end{proof}
\subsubsection{Concentration for the Synchronously Coupled Mean-Field System}

We also require a concentration inequality for the empirical measure of \(J\)
independent copies of the mean-field process to bound the probability of the ``bad
set'' in the main propagation of chaos proof.

\begin{lemma}[Concentration of independent mean-field particles]
\label{lem:concentration-mf}
    Fix \(q>2\), and assume that the objective function \(f\) satisfies
    \cref{assumption:bounded}. Consider \(J\) independent mean-field processes
    \((\overline X_t^j,\overline V_t^j)\), initialized with i.i.d.\ samples from
    \(\overline\rho_0\in\mathcal P_{2q}(\mathbb R^{2d})\).

    Define
    \[
        \overline Y_t^j
        :=
        \overline X_t^j-m_{\overline X}(\overline\rho_t),
        \qquad
        \widehat{\overline Z}_t^j
        :=
        \bigl(\overline V_t^j-m_{\overline V}(\overline\rho_t)\bigr)
        -
        \frac1\gamma\overline Y_t^j,
    \]
    and let
    \[
        \overline L_2^j(t)
        :=
        \psi_2(\overline Y_t^j,\widehat{\overline Z}_t^j),
        \qquad
        \overline L_2^J(t)
        :=
        \frac1J\sum_{j=1}^J \overline L_2^j(t).
    \]
    Assume the parameter conditions needed for
    \cref{prop:ito-mfl-centered-lyapunov-differential} at \(p=2\) and \(p=2q\).
    Let \(\kappa\) satisfy
    \[
        0<\kappa<\min\left\{\edecay{2},\frac{\edecay{2q}}q\right\},
    \]
    and assume the concentration noise absorption condition
    \begin{equation}\label{eq:mf-small-noise-concentration}
        \sigma^2
        \le
        \frac{m^2 c_1^{(2)}(\edecay{2}-\kappa)}
        {2\tau(S)(1+C_{\mathrm{lem}})}.
    \end{equation}
    Then, for every threshold \(A>0\) and every \(J\in\mathbb N_+\),
    \begin{equation}\label{eq:mf-concentration-L2J}
        \proba\left[
            \sup_{t\ge0}
            \e^{\kappa t}\overline L_2^J(t)
            \ge
            \E[\overline L_2^J(0)]+A
        \right]
        \le
        \frac{C_{\mathrm{Bad}}^{\mathrm{MF}}}{A^q}
        J^{-\frac q2}\,
        \overline{\mathcal L}_{2q}(0).
    \end{equation}
    Consequently, by the coercivity bound for \(\psi_2\),
    \begin{equation}\label{eq:mf-concentration-spatial-moment}
        \proba\left[
            \sup_{t\ge0}
            \e^{\kappa t}
            \mathfrak M_2(\mu_{\overline X_t^J})
            \ge
            \frac{\E[\overline L_2^J(0)]+A}{c_1^{(2)}}
        \right]
        \le
        \frac{C_{\mathrm{Bad}}^{\mathrm{MF}}}{A^q}
        J^{-\frac q2}\,
        \overline{\mathcal L}_{2q}(0).
    \end{equation}
    The constant is
    \[
        C_{\mathrm{Bad}}^{\mathrm{MF}}
        :=
        3^q
        \left(
            \overline C_{\mathrm{MZ}}
            +
            C_{Z,q}
            +
            C_{\mathrm{Mart}}
        \right),
    \]
    where
    \[
        \overline C_{\mathrm{MZ}}
        :=
        \frac{\cmz{q} C_\phi^q 2^{q-1}}{c_{1,2q}},
    \]
    \[
        C_{Z,q}
        :=
        \frac{
            2(2K_{\mathrm{tr}})^q \cwm{2q}
        }{
            c_{1,2q}\ell^{q-1}(\edecay{2q}-q\kappa)
        },
        \qquad
        \ell
        :=
        \frac{\edecay{2q}-q\kappa}{2(q-1)},
    \]
    and
    \[
        C_{\mathrm{Mart}}
        :=
        \frac{
            2^{q+1}C_{\mathrm{BDG},q}C_{\sigma,q}
        }{
            \edecay{2q}-q\kappa
        }
        \left(
            \frac{q-2}{\edecay{2q}-q\kappa}
        \right)^{\frac q2-1}.
    \]
\end{lemma}

\begin{proof}
    We use the pathwise quadratic functional
    \[
        \overline L_2^j(t)
        =
        \psi_2(\overline Y_t^j,\widehat{\overline Z}_t^j),
        \qquad
        \overline L_2^J(t)
        =
        \frac1J\sum_{j=1}^J\overline L_2^j(t).
    \]
    From the same It\^o expansion underlying
    \cref{prop:ito-mfl-centered-lyapunov-differential}, but before absorbing the
    quadratic variation term into the drift, each \(\overline L_2^j(t)\) satisfies
    \begin{equation}\label{eq:mf-single-L2-preabsorb}
        \d \overline L_2^j(t)
        \le
        -\edecay{2}\,\overline L_2^j(t)\,\d t
        +
        K_{\mathrm{tr}}
        \left|
            \overline X_t^j-\mathcal M_\alpha(\overline\rho_t^X)
        \right|^2\d t
        +
        \d \overline M_{2,t}^j,
    \end{equation}
    where
    \[
        K_{\mathrm{tr}}
        :=
        \frac{\sigma^2\tau(S)}{m^2},
    \]
    and
    \[
        \d \overline M_{2,t}^j
        =
        \left\langle
            2\widehat{\overline Z}_t^j
            +
            \frac5\gamma\overline Y_t^j,
            \frac{\sigma}{m}
            S\!\left(
                \overline X_t^j-\mathcal M_\alpha(\overline\rho_t^X)
            \right)\d W_t^j
        \right\rangle.
    \]
    Averaging \eqref{eq:mf-single-L2-preabsorb} over \(j=1,\dots,J\), we obtain
    \begin{align}\label{eq:mf-L2J-preabsorb}
        \d \overline L_2^J(t)
        &\le
        -\edecay{2}\,\overline L_2^J(t)\,\d t
        +
        K_{\mathrm{tr}}
        \left(
            \frac1J\sum_{j=1}^J
            \left|
                \overline X_t^j-\mathcal M_\alpha(\overline\rho_t^X)
            \right|^2
        \right)\d t
        +
        \d \overline M_{2,t}^J.
    \end{align}

    Consider the exponentially weighted process
    \[
        U_t:=\e^{\kappa t}\overline L_2^J(t).
    \]
    By It\^o's formula,
    \begin{align}
        \d U_t
        &=
        \e^{\kappa t}
        \left(
            \kappa \overline L_2^J(t)\,\d t
            +
            \d \overline L_2^J(t)
        \right)
        \notag\\
        &\le
        \e^{\kappa t}
        (\kappa-\edecay{2})
        \overline L_2^J(t)\,\d t
        +
        \e^{\kappa t}
        K_{\mathrm{tr}}
        \left(
            \frac1J\sum_{j=1}^J
            \left|
                \overline X_t^j-\mathcal M_\alpha(\overline\rho_t^X)
            \right|^2
        \right)\d t
        +
        \d\mathcal M_t^{\exp}.
    \end{align}

    \proofstep{1. Splitting the noise trace term}
    We bound the trace contribution by pivoting around the empirical weighted mean
    \(\mathcal M_\alpha(\mu_{\overline X_t^J})\). Using
    \((a+b)^2\le2a^2+2b^2\),
    \begin{align}
        \frac1J\sum_{j=1}^J
        \left|
            \overline X_t^j-\mathcal M_\alpha(\overline\rho_t^X)
        \right|^2
        &\le
        \frac2J\sum_{j=1}^J
        \left|
            \overline X_t^j-\mathcal M_\alpha(\mu_{\overline X_t^J})
        \right|^2
        +
        2
        \left|
            \mathcal M_\alpha(\mu_{\overline X_t^J})
            -
            \mathcal M_\alpha(\overline\rho_t^X)
        \right|^2.
    \end{align}
    To bound the first term, we apply the orthogonal variance decomposition around
    the empirical mean
    \[
        m_{\overline X}(\mu_{\overline X_t^J})
        :=
        \frac1J\sum_{j=1}^J\overline X_t^j.
    \]
    Since
    \[
        \sum_{j=1}^J
        \left(
            \overline X_t^j-m_{\overline X}(\mu_{\overline X_t^J})
        \right)
        =
        0,
    \]
    the cross term vanishes, and
    \begin{align}
        \frac1J\sum_{j=1}^J
        \left|
            \overline X_t^j-\mathcal M_\alpha(\mu_{\overline X_t^J})
        \right|^2
        &=
        \frac1J\sum_{j=1}^J
        \left|
            \overline X_t^j-m_{\overline X}(\mu_{\overline X_t^J})
        \right|^2
        +
        \left|
            m_{\overline X}(\mu_{\overline X_t^J})
            -
            \mathcal M_\alpha(\mu_{\overline X_t^J})
        \right|^2.
    \end{align}
    Applying \cref{lem:m-malpha} for \(p=2\), the second term is bounded by
    \[
        C_{\mathrm{lem}}
        \frac1J\sum_{j=1}^J
        \left|
            \overline X_t^j-m_{\overline X}(\mu_{\overline X_t^J})
        \right|^2.
    \]
    Thus
    \begin{equation}
        \frac1J\sum_{j=1}^J
        \left|
            \overline X_t^j-\mathcal M_\alpha(\mu_{\overline X_t^J})
        \right|^2
        \le
        (1+C_{\mathrm{lem}})
        \frac1J\sum_{j=1}^J
        \left|
            \overline X_t^j-m_{\overline X}(\mu_{\overline X_t^J})
        \right|^2.
    \end{equation}
    Since the empirical mean minimizes the sum of squared distances,
    \[
        \frac1J\sum_{j=1}^J
        \left|
            \overline X_t^j-m_{\overline X}(\mu_{\overline X_t^J})
        \right|^2
        \le
        \frac1J\sum_{j=1}^J
        \left|
            \overline X_t^j-m_{\overline X}(\overline\rho_t)
        \right|^2
        =
        \frac1J\sum_{j=1}^J|\overline Y_t^j|^2.
    \]
    By the coercivity of \(\psi_2\),
    \[
        |\overline Y_t^j|^2
        \le
        \frac1{c_1^{(2)}}\overline L_2^j(t).
    \]
    Substituting this bound into the inequality for \(U_t\), we obtain
    \begin{align}
        \d U_t
        &\le
        \e^{\kappa t}
        \left[
            \kappa-\edecay{2}
            +
            2K_{\mathrm{tr}}
            \frac{1+C_{\mathrm{lem}}}{c_1^{(2)}}
        \right]
        \overline L_2^J(t)\,\d t
        +
        2\e^{\kappa t}K_{\mathrm{tr}}
        \left|
            \mathcal M_\alpha(\mu_{\overline X_t^J})
            -
            \mathcal M_\alpha(\overline\rho_t^X)
        \right|^2\d t
        +
        \d\mathcal M_t^{\exp}.
    \end{align}
    By \eqref{eq:mf-small-noise-concentration}, the coefficient multiplying
    \(\overline L_2^J(t)\) is non-positive. We therefore drop it and integrate:
    \begin{equation}\label{eq:proof-mfl-concentration-Y-bound}
        U_t
        \le
        U_0+Z_t+\mathcal M_t^{\exp},
    \end{equation}
    where
    \begin{equation}
        Z_t
        :=
        2K_{\mathrm{tr}}
        \int_0^t
        \e^{\kappa s}
        \left|
            \mathcal M_\alpha(\mu_{\overline X_s^J})
            -
            \mathcal M_\alpha(\overline\rho_s^X)
        \right|^2\d s.
    \end{equation}

    By a union bound, we separate the probability of large deviations into three
    components evaluated at threshold \(A/3\):
    \begin{align}
        \proba\left[
            \sup_{t\ge0}U_t
            \ge
            \E[U_0]+A
        \right]
        &\le
        \proba\left[
            U_0-\E[U_0]\ge\frac A3
        \right]
        +
        \proba\left[
            \sup_{t\ge0}Z_t\ge\frac A3
        \right]
        +
        \proba\left[
            \sup_{t\ge0}\mathcal M_t^{\exp}\ge\frac A3
        \right].
    \end{align}

    \proofstep{2. Bounding the initial data \(U_0\)}
    Since
    \[
        U_0
        =
        \frac1J\sum_{j=1}^J\overline L_2^j(0)
    \]
    is an empirical average of i.i.d.\ random variables with mean \(\E[U_0]\), the
    Marcinkiewicz--Zygmund inequality gives
    \[
        \E\left[
            |U_0-\E[U_0]|^q
        \right]
        \le
        \cmz{q}J^{-q/2}
        \E\left[
            |\overline L_2^1(0)|^q
        \right].
    \]
    As in the particle concentration estimate,
    \[
        \E\left[
            |\overline L_2^1(0)|^q
        \right]
        \le
        \frac{C_\phi^q2^{q-1}}{c_{1,2q}}\,
        \overline{\mathcal L}_{2q}(0).
    \]
    Therefore,
    \[
        \E\left[
            |U_0-\E[U_0]|^q
        \right]
        \le
        \frac{\cmz{q}C_\phi^q2^{q-1}}{c_{1,2q}}
        J^{-q/2}\overline{\mathcal L}_{2q}(0).
    \]
    By Markov's inequality,
    \begin{equation}
        \proba\left[
            U_0-\E[U_0]\ge\frac A3
        \right]
        \le
        \frac{3^q\overline C_{\mathrm{MZ}}}{A^qJ^{q/2}}
        \overline{\mathcal L}_{2q}(0),
        \qquad
        \overline C_{\mathrm{MZ}}
        :=
        \frac{\cmz{q}C_\phi^q2^{q-1}}{c_{1,2q}}.
    \end{equation}

    \proofstep{3. Bounding the drift error \(Z_t\)}
    Let
    \[
        E_s
        :=
        \left|
            \mathcal M_\alpha(\mu_{\overline X_s^J})
            -
            \mathcal M_\alpha(\overline\rho_s^X)
        \right|.
    \]
    Applying H\"older's inequality to the time integral with
    \[
        \ell
        =
        \frac{\edecay{2q}-q\kappa}{2(q-1)}
    \]
    gives
    \begin{align}
        \E\left[
            \sup_{t\ge0}Z_t^q
        \right]
        &\le
        (2K_{\mathrm{tr}})^q
        \frac1{\ell^{q-1}}
        \int_0^\infty
        \e^{(q-1)\ell s+q\kappa s}
        \E[E_s^{2q}]\,\d s.
    \end{align}
    By the \(L^{2q}\) law of large numbers for the empirical weighted mean
    \(\cref{lem:mc_est}\),
    \[
        \E[E_s^{2q}]
        \le
        \frac{\cwm{2q}}{J^q}
        \E\left[
            \left|
                \overline X_s^1-m_{\overline X}(\overline\rho_s)
            \right|^{2q}
        \right].
    \]
    By \cref{rem:mfl-single-sample-centered-decay},
    \[
        \E\left[
            \left|
                \overline X_s^1-m_{\overline X}(\overline\rho_s)
            \right|^{2q}
        \right]
        \le
        C\,\e^{-\edecay{2q}s}
        \mathfrak M_{2q}(\overline\rho_0)
        \le
        \frac{C}{c_{1,2q}}
        \e^{-\edecay{2q}s}
        \overline{\mathcal L}_{2q}(0).
    \]
    Absorbing the harmless constant \(C\) into \(\cwm{2q}\), we get
    \[
        \E[E_s^{2q}]
        \le
        \frac{\cwm{2q}}{c_{1,2q}J^q}
        \e^{-\edecay{2q}s}
        \overline{\mathcal L}_{2q}(0).
    \]
    Since $(q-1)\ell+q\kappa-\edecay{2q}
        =
        -\frac{\edecay{2q}-q\kappa}{2}\,,$
    the time integral gives $\frac{2}{\edecay{2q}-q\kappa}\,.$
    Therefore
    \begin{align}
        \E\left[
            \sup_{t\ge0}Z_t^q
        \right]
        &\le
        \frac{
            (2K_{\mathrm{tr}})^q\cwm{2q}
        }{
            c_{1,2q}\ell^{q-1}J^q
        }
        \left(
            \frac{2}{\edecay{2q}-q\kappa}
        \right)
        \overline{\mathcal L}_{2q}(0).
    \end{align}
    Since \(q\ge2\), \(J^{-q}\le J^{-q/2}\). Applying Markov's inequality at
    threshold \(A/3\),
    \begin{equation}
        \proba\left[
            \sup_{t\ge0}Z_t\ge\frac A3
        \right]
        \le
        \frac{3^qC_{Z,q}}{A^qJ^{q/2}}
        \overline{\mathcal L}_{2q}(0),
    \end{equation}
    where
    \[
        C_{Z,q}
        :=
        \frac{
            2(2K_{\mathrm{tr}})^q\cwm{2q}
        }{
            c_{1,2q}\ell^{q-1}(\edecay{2q}-q\kappa)
        }.
    \]

    \proofstep{4. Bounding the martingale \(\mathcal M_t^{\exp}\)}
    The martingale term is bounded identically to Part 2 of
    \cref{lem:concentration-ineq}. The independent Brownian motions and uniformly
    bounded effective volatility through the true weighted mean
    \(\mathcal M_\alpha(\overline\rho_t^X)\) yield the BDG bound
    \begin{equation}
        \proba\left[
            \sup_{t\ge0}\mathcal M_t^{\exp}\ge\frac A3
        \right]
        \le
        \frac{3^qC_{\mathrm{Mart}}}{A^qJ^{q/2}}
        \overline{\mathcal L}_{2q}(0),
    \end{equation}
    where
    \[
        C_{\mathrm{Mart}}
        =
        \frac{
            2^{q+1}C_{\mathrm{BDG},q}C_{\sigma,q}
        }{
            \edecay{2q}-q\kappa
        }
        \left(
            \frac{q-2}{\edecay{2q}-q\kappa}
        \right)^{\frac q2-1}.
    \]

    \proofstep{5. Conclusion}
    Summing the three probability bounds gives
    \[
        \proba\left[
            \sup_{t\ge0}U_t
            \ge
            \E[U_0]+A
        \right]
        \le
        \frac{C_{\mathrm{Bad}}^{\mathrm{MF}}}{A^qJ^{q/2}}
        \overline{\mathcal L}_{2q}(0),
    \]
    where
    \[
        C_{\mathrm{Bad}}^{\mathrm{MF}}
        :=
        3^q
        \left(
            \overline C_{\mathrm{MZ}}
            +
            C_{Z,q}
            +
            C_{\mathrm{Mart}}
        \right).
    \]
\end{proof}


\subsection{Weighted-mean estimates imported from \cite{uit_cbo}}
\label{sec:weighted_mean_estimates}

The following two auxiliary estimates are taken directly from \cite[Lemmas~5.13 and~5.14]{uit_cbo}.
We restate them here in the present notation for the reader's convenience and refer to \cite{uit_cbo} for the proofs.

\begin{lemma}[Local Lipschitz continuity of $\mu \mapsto \mathcal{M}_\alpha(\mu) - m_X(\mu)$]
    \label{lem:stability_est}
    Suppose that \cref{assumption:bounded,assumption:lip} are satisfied.
    Then, for all $\mu, \nu \in \mathcal P_2(\mathbb R^d)$,
    \[
        \left\lvert
        \mathcal{M}_\alpha(\mu) - m_X(\mu)
        - \mathcal{M}_\alpha(\nu) + m_X(\nu)
        \right\rvert
        \le
        C_{\mathcal M}
        \left(
            \sqrt{\momentp{2}{\mu}} + \sqrt{\momentp{2}{\nu}}
        \right)
        \wasserstein_2(\mu,\nu),
    \]
    where
    \[
        C_{\mathcal M}
        := 2 \alpha L_f \e^{2\alpha(\overline f - \underline f)}.
    \]
\end{lemma}

\begin{proof}
    One can refer to \cite[Lemma~5.13]{uit_cbo}.
\end{proof}

\begin{lemma}[Convergence of the weighted mean for i.i.d.\ samples]
    \label{lem:mc_est}
    Fix~$p \ge 2$.
    Suppose that $f$ satisfies \cref{assumption:bounded},
    and that $\overline{\rho} \in \mathcal P_p(\real^d)$ has finite moments up to order~$p$.
    Then there exists a constant $\cwm{p}\bigl(\alpha,\overline f,\underline f\bigr)$ such that
    \begin{align}
        \expect \left|
            \mathcal{M}_\alpha(\mu_{\overline{X}^J})
            - \mathcal{M}_\alpha(\overline{\rho})
        \right|_p^p
        \le
        \cwm{p}\,
        \expect \Bigl|
            \overline{X}^1 - m_{\overline X}(\overline{\rho})
        \Bigr|_p^p
        J^{-p/2},
        \qquad
        \mu_{\overline{X}^J}
        := \frac{1}{J}\sum_{j=1}^{J}\delta_{\overline X^j},
        \qquad
        \{\overline X^j\}_{j\in\mathbb N}
        \stackrel{\rm i.i.d.}{\sim}
        \overline \rho,
    \end{align}
    where
    \[
        \cwm{p}\bigl(\alpha,\underline f,\overline f\bigr)
        :=
        \cmz{p}\,
        \e^{p\alpha(\overline f-\underline f)}
        \Bigl(1+\e^{\frac{\alpha}{p}(\overline f-\underline f)}\Bigr)^p.
    \]
\end{lemma}

\begin{proof}
    One can refer to \cite[Lemma~5.14]{uit_cbo}.
\end{proof}

\newpage

\appendix
\appendix
\section{Structural Analysis of the Centered Hypocoercive Formulation}

This appendix explains the analytical necessity of the coordinate transformation introduced in \cref{sec:setting-transformation}. The main point is that the functional based on the internal variables $(Y,\hat Z)$ yields a closed dissipation estimate, whereas the more na\"ive centered functional based on $(X^j-m_X(\emp{t}),V^j)$ does not provide enough coercivity to obtain a uniform-in-time bound in the presence of the weighted-mean error
\[
    \Delta_\alpha(t):=m_X(\emp{t})-\mathcal M_\alpha(\mu_{X_t^J}).
\]

\subsection{Limitations of the Basic Centered Functional}

Consider the centered energy functional acting on the physical variables $(Y_t^j,V_t^j)$, where
\[
    Y_t^j := X_t^j - m_X(\emp{t}).
\]
Define
\[
    {L}_{\mathrm{std}, p}(t)
    :=
    \frac{1}{J}\sum_{j=1}^J
    \Big(
        a |Y_t^j|^p
        + b |V_t^j|^p
        + c |Y_t^j|^{p-2}\langle Y_t^j, V_t^j \rangle
    \Big).
\]
Here, we refer to such a functional as the na\"ive or the basic centered functional, in contrast to the ``shifted functional" that we used in the main proof of the centered moments decay in \cref{thm:exp-decay-centered-moments}, and \cref{thm:mfl-decay-centered-moment}. The evolution of the centered position is given by
\[
    \d Y_t^j = (V_t^j - m_V(\empv{t})) \d t.
\]
Consequently,
\[
    \d |Y_t^j|^p
    =
    p |Y_t^j|^{p-2} \langle Y_t^j, V_t^j - m_V(\empv{t}) \rangle \d t.
\]
While the term involving $V_t^j$ contributes to the cross term cancellation, the term involving $m_V(\empv{t})$ persists. Although $\sum_{j=1}^J Y_t^j=0$ implies that linear terms vanish upon summation, the nonlinear coupling in the $p$-norm produces terms of the form
\[
    |Y_t^j|^{p-2}\langle Y_t^j,m_V(\empv{t})\rangle,
\]
which do not cancel for $p>2$. Estimating these terms requires Young's inequality, thereby splitting the dissipation budget between the internal mode and the bulk quantity $|m_V(\empv{t})|^p$.

The basic issue is that the pair $(Y,V)$ is not aligned with the intrinsic linear subspace of the internal dynamics: the position fluctuation $Y$ is driven by $V-m_V$, not by $V$ itself. Thus the naive centered functional still mixes internal dissipation with bulk transport.

\subsection{Decoupling via the Auxiliary Variable \texorpdfstring{$\hat{Z}$}{Z hat}}

The transformation to the centered auxiliary variable $\hat Z_t^j$ resolves this coupling issue algebraically. Recall that
\[
    Z_t^j := V_t^j - \frac{1}{\gamma}Y_t^j,
    \qquad
    \hat Z_t^j := Z_t^j - m_Z(t)
    =
    \bigl(V_t^j-m_V(\empv{t})\bigr)-\frac{1}{\gamma}Y_t^j,
\]
where we used that $m_Z(t)=m_V(\empv{t})$ since $\frac1J\sum_{j=1}^J Y_t^j=0$.

The internal dynamics derived above form a closed linear system:
\begin{align}\label{eq:internal_linear_structure}
    \d Y_t^j &= \left(\hat Z_t^j + \frac{1}{\gamma}Y_t^j\right)\d t, \\
    \d \hat Z_t^j &= \left(-K_Z \hat Z_t^j - K_Y Y_t^j\right)\d t + \d M_t^{j,\mathrm{center}}.
\end{align}
This formulation has two decisive advantages.

\paragraph{1. Elimination of the Bulk Transport}
The mean velocity $m_V(\empv{t})$ no longer appears in the drift of the internal system \eqref{eq:internal_linear_structure}. Hence the evolution of the shape of the swarm and its internal energy is algebraically decoupled from the bulk motion. In particular, the cross-term cancellation in the Lyapunov derivative is encoded in the coefficients of the nonlinear polynomial form, rather than recovered through Young-type estimates involving $m_V(\empv{t})$.

\paragraph{2. Isolation of the Weighted-Mean Error}
The quantity
\[
    \Delta_\alpha(t)=m_X(\emp{t})-\mathcal M_\alpha(\mu_{X_t^J})
\]
is common to all particles. After centering, it disappears from the linear drift of the internal variables and reappears only through the centered noise term and the separate equation for the bulk velocity. This moves the main difficulty from a drift-level obstruction to a forcing term that can be treated perturbatively.

Centering alone, however, is not enough. The additional shift by $-\gamma^{-1}Y$ is what strengthens the restoring force in the internal dynamics and creates the margin needed to absorb the weighted-mean error.

\subsection{Hypocoercive Design of the Variable $Z$}

The definition
\[
    Z = V - \gamma^{-1}Y
\]
is motivated by the overdamped structure, but the sign choice is also crucial for the coercivity mechanism.

\paragraph{Basic Functional vs. Shifted Functional}
Consider the deterministic linearized internal dynamics
\[
    \dot Y = V - m_V(\empv{t}),
    \qquad
    m\dot V = -\gamma V - Y - \Delta_\alpha.
\]

\begin{enumerate}
    \item \textbf{Basic functional.}
    If one works with the pair $(Y,V)$, then the restoring contribution in the spatial channel comes only from the original force term $-\frac1m Y$. Thus the leading spatial coercivity is of order $1/m$. The weighted-mean error enters through the same force term, and therefore also appears at order $1/m$. At this level, the dissipation and the error act on the same scale.

    \item \textbf{Shifted functional.}
    If instead one sets
    \[
        Z = V - \frac1\gamma Y,
        \qquad\text{so that}\qquad
        V = Z + \frac1\gamma Y,
    \]
    then substituting into the friction term gives
    \[
        -\frac{\gamma}{m}V
        =
        -\frac{\gamma}{m}\left(Z+\frac1\gamma Y\right)
        =
        -\frac{\gamma}{m}Z - \frac1m Y.
    \]
    This contributes an additional restoring term $-\frac1m Y$ on top of the original potential force. Accordingly, the effective coefficient in the spatial channel becomes
    \[
        K_Y = \frac{2}{m} + \frac{1}{\gamma^2},
    \]
    while the weighted-mean error remains at order $1/m$.
\end{enumerate}

\paragraph{The gain from the shifted variable}
The shift improves the balance between dissipation and error from the level of $1:1$ to $2:1$ at the level of the leading coefficients. This is the key margin used in the main decay estimate.

\paragraph{Reason for not using \texorpdfstring{$W=V+\gamma^{-1}Y$}{W = V + gamma\string^(-1)Y}}
For notational simplicity, we suppress the bulk-centering in the following heuristic calculation and write $V$ for a single internal velocity mode. One might consider a ``physically natural'' coordinate that measures the deviation from the overdamped manifold, namely
\[
    W := V + \gamma^{-1}Y.
\]
However, evaluating the drift of $W$ reveals a near-complete cancellation between the friction and the potential forces. Using
\[
    \dot{V} = -\frac{\gamma}{m}V - \frac{1}{m}Y
\]
and substituting
\[
    V = W - \gamma^{-1}Y,
\]
we obtain
\begin{align}
    \dot{W}
    &= \dot{V} + \frac{1}{\gamma}\dot{Y} \nonumber\\
    &= \left( -\frac{\gamma}{m}V - \frac{1}{m}Y \right) + \frac{1}{\gamma}V \nonumber\\
    &= -\left( \frac{\gamma}{m} - \frac{1}{\gamma} \right)\left( W - \frac{1}{\gamma}Y \right) - \frac{1}{m}Y \nonumber\\
    &= -\left( \frac{\gamma}{m} - \frac{1}{\gamma} \right) W - \frac{1}{\gamma^2}Y .
\end{align}
The main spatial restoring term of size $1/m$ disappears, leaving only the much weaker contribution $-\gamma^{-2}Y$. In particular, this choice destroys the cross-term mechanism needed to dissipate the position energy.

Thus, the variable
\[
    Z = V - \gamma^{-1}Y
\]
is not an arbitrary reparametrization. The negative sign aligns the restoring contributions so that they add rather than cancel, which enlarges the coercive margin in the spatial channel.

\section{Analysis of the Standard Energy Functional}

For completeness, we record the calculation for the na\"ive centered functional without the variable shift. This makes precise why the estimate does not close for the parameter range relevant to the main theorem.

For readability, write
\[
    \mathcal M_\alpha(t):=\mathcal M_\alpha(\mu_{X_t^J}),
    \qquad
    \Delta_\alpha(t):=m_X(\emp{t})-\mathcal M_\alpha(t).
\]
We define
\begin{align}
    {L}_{\mathrm{std},p}(t)
    :=
    \frac{1}{J}\sum_{j=1}^J
    \Big(
        a |X_t^j-m_X(\emp{t})|^p
        + b |V_t^j|^p
        + c |X_t^j-m_X(\emp{t})|^{p-2}\scp{X_t^j-m_X(\emp{t})}{V_t^j}
    \Big).
\end{align}

By It\^o's formula,
\begin{align}
    & \d \left \langle |X_t^j - m_X(\emp{t})|^{p-2}( X_t^j - m_X(\emp{t}) ), V_t^j \right \rangle
    \\
    & = (p-2) |X_t^j - m_X(\emp{t})|^{p-4}
    \left \langle X_t^j - m_X(\emp{t}), V_t^j - m_V(\empv{t}) \right \rangle
    \left \langle X_t^j - m_X(\emp{t}), V_t^j \right \rangle \d t \nonumber\\
    & \quad
    + |X_t^j - m_X(\emp{t})|^{p-2}
    \Bigg(
        \left \langle V_t^j - m_V(\empv{t}), V_t^j\right \rangle \d t
        \nonumber\\
    & \qquad\qquad
        + \left  \langle X_t^j - m_X(\emp{t}),
            -\frac{\gamma}{m}V_t^j \d t
            - \frac{ X_t^j - \mathcal M_\alpha(t)} {m} \d t
            + \frac{\sigma S(X_t^j - \mathcal M_\alpha(t))} {m} \d W_t^j
        \right \rangle
    \Bigg) .
    \nonumber
\end{align}
Hence
\begin{align}
    \d {L}_{\mathrm{std},p}(t)
    &=
    \frac{a}{J} \sum_{j =1}^J
    p |X_t^j - m_X(\emp{t})|^{p-2}
    \langle X_t^j - m_X(\emp{t}), V_t^j - m_V(\empv{t})\rangle \d t
    \\
    &\quad
    +
    \frac{b}{J}
    \sum_{j = 1}^J
    \Bigg\{
        p |V_t^j|^{p-2}
        \left(
            - \frac{\gamma}{m}|V_t^j|^2 \d t
            - \frac{(X_t^j - \mathcal M_\alpha(t)) \cdot V_t^j}{m} \d t
            - \frac{\sigma S(X_t^j - \mathcal M_\alpha(t)) \cdot V_t^j}{m} \d W_t^j
        \right)
        \nonumber\\
    &\qquad\qquad
        + \frac{p(p-2)\sigma^2}{2m^2} |V_t^j|^{p-4}
        \bigl\langle V_t^j, S(X_t^j - \mathcal M_\alpha(t)) \bigr\rangle^2 \d t
        + \frac{p\sigma^2}{2m^2}|V_t^j|^{p-2} |S(X_t^j - \mathcal M_\alpha(t))|^2 \d t
    \Bigg\}
    \nonumber\\
    &\quad
    + \frac{c}{J}\sum_{j = 1}^J
    \Bigg\{
        (p-2)|X_t^j - m_X(\emp{t})|^{p-4}
        \langle X_t^j - m_X(\emp{t}), V_t^j - m_V(\empv{t}) \rangle
        \langle X_t^j - m_X(\emp{t}), V_t^j \rangle \d t
        \nonumber\\
    &\qquad\qquad
        + |X_t^j - m_X(\emp{t})|^{p-2}
        \Bigg(
            \langle V_t^j - m_V(\empv{t}), V_t^j\rangle \d t
            \nonumber\\
    &\qquad\qquad\qquad
            + \left\langle X_t^j - m_X(\emp{t}),
                - \frac{\gamma}{m} V_t^j \d t
                - \frac{X_t^j - \mathcal M_\alpha(t)}{m} \d t
                - \frac{\sigma S(X_t^j - \mathcal M_\alpha(t))}{m} \d W_t^j
            \right\rangle
        \Bigg)
    \Bigg\} .
    \nonumber
\end{align}

Taking expectation and setting
\[
    a=\frac{\gamma}{pm}c,
\]
we obtain
\begin{align}\label{eq:Mp-decay-std}
    \frac{\d}{\d t}\E[{L}_{\mathrm{std},p}(t)]
    &=
    -\frac{ap}{J} \E \sum_{j =1}^J
    |X_t^j - m_X(\emp{t})|^{p-2}
    \langle X_t^j - m_X(\emp{t}), m_V(\empv{t}) \rangle
    \\
    &\quad
    - \frac{bp}{m J}\E \sum_{j =1}^J
    \left(
        \gamma |V_t^j|^p
        + |V_t^j|^{p-2}\langle X_t^j - \mathcal M_\alpha(t), V_t^j \rangle
    \right)
    \nonumber\\
    &\quad
    + \frac{bp \sigma^2 }{2m^2 J} \E\sum_{j = 1}^J
    \left(
        (p-2) |V_t^j|^{p-4} \langle V_t^j, S(X_t^j - \mathcal M_\alpha(t)) \rangle^2
        + |V_t^j|^{p-2}|S(X_t^j - \mathcal M_\alpha(t))|^2
    \right)
    \nonumber\\
    &\quad
    + \frac{c }{J}\E \sum_{j =1}^J
    \Bigg(
        (p-2) |X_t^j - m_X(\emp{t})|^{p-4}
        \langle X_t^j - m_X(\emp{t}), V_t^j - m_V(\empv{t})\rangle
        \langle X_t^j - m_X(\emp{t}), V_t^j \rangle
        \nonumber\\
    &\qquad\qquad
        + |X_t^j - m_X(\emp{t})|^{p-2}
        \left(
            \langle V_t^j - m_V(\empv{t}), V_t^j \rangle
            - \left\langle X_t^j - m_X(\emp{t}), \frac{X_t^j - \mathcal M_\alpha(t)}{m}\right\rangle
        \right)
    \Bigg)
    \nonumber\\
    &\leq
    - \frac{pb\gamma}{m} \E[V_p(t)]
    - \frac{c}{m} \E[X_p(t)]
    + \frac{\gamma c}{mJ} \E\sum_{j = 1}^J |X_t^j - m_X(\emp{t})|^{p-1} |m_V(\empv{t})|
    \nonumber\\
    &\quad
    + \frac{pb}{mJ}\E \sum_{j = 1}^J |V_t^j|^{p-1} |X_t^j- \mathcal M_\alpha(t)|
    + \frac{bp\sigma^2(p-2+ \tau(S))}{2m^2} \E \Big[\frac1J \sum_{j=1}^J |V_t^j|^{p-2}|X_t^j - \mathcal M_\alpha(t)|^2\Big]
    \nonumber\\
    &\quad
    + \frac{(p-1) c}{J}\E \sum_{j=1}^J |X_t^j - m_X(\emp{t})|^{p-2}|V_t^j||V_t^j - m_V(\empv{t})|
    \nonumber\\
    &\quad
    + \frac{c}{mJ}\E \sum_{j=1}^J |X_t^j - m_X(\emp{t})|^{p-1}|m_X(\emp{t}) - \mathcal M_\alpha(t)|.
    \nonumber
\end{align}
Here
\[
    X_p(t):=\frac1J\sum_{j=1}^J |X_t^j-m_X(\emp{t})|^p,
    \qquad
    V_p(t):=\frac1J\sum_{j=1}^J |V_t^j|^p.
\]

\paragraph{$(p=2)$ case.}
Since
\[
    \sum_{j=1}^J \langle X_t^j-m_X(\emp{t}),m_V(\empv{t})\rangle = 0,
\]
we obtain
\begin{align}
    \frac{\d}{\d t}\E[{L}_{\mathrm{std},2}(t)]
    &\leq \frac{2a}{J}\sum_{j =1}^J \E \langle X_t^j - m_X(\emp{t}), V_t^j \rangle
    \\
    &\qquad
    + \frac{b}{J} \sum_{j=1}^J \E\left[
        - \frac{2\gamma}{m}|V_t^j|^2
        + \frac{2|X_t^j - \mathcal M_\alpha(t)||V_t^j|}{m}
        + \tau(S) \sigma^2 m^{-2} |X_t^j - \mathcal M_\alpha(t)|^2
    \right]
    \nonumber\\
    &\qquad
    + \frac{c}{J} \sum_{j = 1}^J \E\Big[
        \langle V_t^j - m_V(\empv{t}), V_t^j \rangle
        - \frac{\gamma}{m} \langle X_t^j - m_X(\emp{t}), V_t^j \rangle
        - \frac{1}{m}\langle X_t^j - m_X(\emp{t}), X_t^j - \mathcal M_\alpha(t) \rangle
    \Big]
    \nonumber\\
    &\leq
    \frac{b}{J}\sum_{j =1}^J \E\left[
        - \frac{2\gamma}{m}|V_t^j|^2
        + \left( \frac{|V_t^j|^2}{\varepsilon m} + \frac{\varepsilon |X_t^j - \mathcal M_\alpha(t)|^2}{m}\right)
        + \tau(S) \sigma^2 m^{-2} |X_t^j - \mathcal M_\alpha(t)|^2
    \right]
    \nonumber\\
    &\qquad
    + \frac{c}{J} \sum_{j=1}^J \E|V_t^j|^2
    - c\,\E|m_V(\empv{t})|^2
    - \frac{c}{mJ} \sum_{j =1}^J \E |X_t^j - m_X(\emp{t})|^2
    \nonumber\\
    &\le
    -\Big(\frac{2b\gamma}{m}-\frac{b}{m\varepsilon}-c\Big)\,
    \E\Big[\frac1J\sum_{j=1}^J |V_t^j|^2\Big]
    - c\,\E|m_V(\empv{t})|^2
    - \frac{c}{m}\,\E\Big[\frac1J\sum_{j=1}^J |X_t^j-m_X(\emp{t})|^2\Big]
    \nonumber\\
    &\qquad
    + b\Big(\frac{\varepsilon}{m}+\frac{\tau(S)\sigma^2}{m^2}\Big)
    \E\Big[\frac1J\sum_{j=1}^J |X_t^j-\mathcal M_\alpha(t)|^2\Big]
    \nonumber\\
    &\leq
    -\Big(\frac{2b\gamma}{m}-\frac{b}{m\varepsilon}-c\Big)\,
    \E\Big[\frac1J\sum_{j=1}^J |V_t^j|^2\Big]
    \nonumber\\
    &\qquad
    + \left\{
        b\Big(\frac{\varepsilon}{m}+\frac{\tau(S)\sigma^2}{m^2}\Big)
        \left( 1 + e^{\frac{\alpha}{2}(\overline f - \underline f)} \right)^2
        - \frac{c}{m}
    \right\}
    \E\Big[\frac1J\sum_{j=1}^J |X_t^j-m_X(\emp{t})|^2\Big].
    \nonumber
\end{align}
Thus, with a suitable choice such as $b=1$ and $c=\frac{\gamma}{2m}$, one can still obtain exponential decay at the second-moment level under a restrictive smallness condition on $\sigma$.

\paragraph{$(p>2)$ case: Young-type closure of \eqref{eq:Mp-decay-std}}
Let $p\ge 2$ and recall \eqref{eq:Mp-decay-std}. We estimate each remainder term using Young's inequality with distinct parameters $\eps_i>0$.

\smallskip
\noindent
\emph{(i) The $m_V(\empv{t})$-term.}
For any $\eps_0>0$,
\begin{align}\label{eq:young-meanV}
    \frac{\gamma c}{m}\,\E\Big[\frac{1}{J}\sum_{j=1}^J |X_t^j-m_X(\emp{t})|^{p-1}|m_V(\empv{t})|\Big]
    &\le
    \frac{\gamma c}{m}\,\E\Big[
        \frac{p-1}{p}\eps_0\,X_p(t)
        + \frac{1}{p}\eps_0^{-(p-1)}\,|m_V(\empv{t})|^p
    \Big]
    \\
    &\le
    \frac{\gamma c}{m}\,\E\Big[
        \frac{p-1}{p}\eps_0\,X_p(t)
        + \frac{1}{p}\eps_0^{-(p-1)}\,V_p(t)
    \Big].
    \nonumber
\end{align}

\smallskip
\noindent
\emph{(ii) The $|V|^{p-1}|X-\mathcal M_\alpha|$ term.}
Using
\[
    X_t^j-\mathcal M_\alpha(t)
    =
    \bigl(X_t^j-m_X(\emp{t})\bigr)
    +
    \bigl(m_X(\emp{t})-\mathcal M_\alpha(t)\bigr),
\]
we obtain, for any $\eps_1,\eps_2>0$,
\begin{align}\label{eq:young-VX1}
    \frac{pb}{m}\E\Big[\frac{1}{J}\sum_{j=1}^J |V_t^j|^{p-1}|X_t^j-m_X(\emp{t})|\Big]
    &\le
    \frac{b}{m}\E\Big[
        (p-1)\eps_1^{-\frac{p}{p-1}}\,V_p(t)
        + \eps_1^p\,X_p(t)
    \Big],
\\
\label{eq:young-VX2}
    \frac{pb}{m}\E\Big[\frac{1}{J}\sum_{j=1}^J |V_t^j|^{p-1}|m_X(\emp{t})-\mathcal M_\alpha(t)|\Big]
    &\le
    \frac{b}{m}\E\Big[
        (p-1)\eps_2^{-\frac{p}{p-1}}\,V_p(t)
        + \eps_2^p\,|m_X(\emp{t})-\mathcal M_\alpha(t)|^p
    \Big].
\end{align}

\smallskip
\noindent
\emph{(iii) The mixed term $|X-m_X(\emp{t})|^{p-2}|V||V-m_V(\empv{t})|$.}
Using
\[
    |V_t^j-m_V(\empv{t})|\le |V_t^j|+|m_V(\empv{t})|
\]
and the inequality
\[
    |x|^{p-2}|v|^2
    \le
    \frac{p-2}{p}\eps |x|^p
    + \frac{2}{p}\eps^{-\frac{2}{p-2}}|v|^p,
\]
we obtain, for any $\eps_3,\eps_4>0$,
\begin{align}\label{eq:young-cross-VVbar}
    &(p-1)c\E\Big[\frac{1}{J}\sum_{j=1}^J |X_t^j-m_X(\emp{t})|^{p-2}|V_t^j||V_t^j-m_V(\empv{t})|\Big]
    \\
    &\qquad\le
    (p-1)c\,\E\Big[\frac{1}{J}\sum_{j=1}^J |X_t^j-m_X(\emp{t})|^{p-2}|V_t^j|^2\Big]
    +(p-1)c\,\E\Big[\frac{1}{J}\sum_{j=1}^J |X_t^j-m_X(\emp{t})|^{p-2}|V_t^j|\,|m_V(\empv{t})|\Big]
    \nonumber\\
    &\qquad\le
    (p-1)c\,\E\Big[
        \frac{p-2}{p}\eps_3\,X_p(t)
        + \frac{2}{p}\eps_3^{-\frac{2}{p-2}}\,V_p(t)
    \Big]
    \nonumber\\
    &\qquad\quad
    +(p-1)c\,\E\Big[
        \frac{p-2}{p}\eps_4\,X_p(t)
        + \frac{2}{p}\eps_4^{-\frac{2}{p-2}}\,V_p(t)
    \Big].
    \nonumber
\end{align}

\smallskip
\noindent
\emph{(iv) The $|X-m_X(\emp{t})|^{p-1}|m_X(\emp{t})-\mathcal M_\alpha(t)|$ term.}
For any $\eps_5>0$,
\begin{align}\label{eq:young-Xdelta}
    \frac{c}{m}\E\Big[\frac{1}{J}\sum_{j=1}^J |X_t^j-m_X(\emp{t})|^{p-1}|m_X(\emp{t})-\mathcal M_\alpha(t)|\Big]
    \le
    \frac{c}{m}\E\Big[
        \frac{p-1}{p}\eps_5^{-\frac{p}{p-1}}\,X_p(t)
        + \frac{1}{p}\eps_5^p\,|m_X(\emp{t})-\mathcal M_\alpha(t)|^p
    \Big].
\end{align}

\smallskip
\noindent
\emph{(v) The It\^o correction term.}
Fix $\chi_S$ such that, for all $x,v\in\mathbb R^d$,
\begin{align}\label{eq:chiS-def}
    (p-2)|v|^{p-4}\scp{v}{S(x)}^2 + |v|^{p-2}|S(x)|^2
    \le
    \chi_S\,|v|^{p-2}|x|^2 .
\end{align}
Then
\begin{align}\label{eq:ito-corr-chiS}
    \frac{bp\sigma^2}{2m^2}\,
    \E\Big[\frac{1}{J}\sum_{j=1}^J
    \Big(
        (p-2)|V_t^j|^{p-4}\scp{V_t^j}{S(X_t^j-\mathcal M_\alpha(t))}^2
        + |V_t^j|^{p-2}|S(X_t^j-\mathcal M_\alpha(t))|^2
    \Big)\Big]
    \\
    \le
    \frac{bp\sigma^2}{2m^2}\,\chi_S\,
    \E\Big[\frac{1}{J}\sum_{j=1}^J |V_t^j|^{p-2}|X_t^j-\mathcal M_\alpha(t)|^2\Big].
    \nonumber
\end{align}
Applying Young's inequality with $\eps_6>0$,
\begin{align}\label{eq:young-ito-close}
    |V|^{p-2}|X|^2
    \le
    \frac{p-2}{p}\eps_6^{-\frac{p-2}{2}}|V|^p
    + \frac{2}{p}\eps_6|X|^p.
\end{align}
Using
\[
    |X_t^j-\mathcal M_\alpha(t)|^p
    \le
    2^{p-1}\Big(|X_t^j-m_X(\emp{t})|^p+|m_X(\emp{t})-\mathcal M_\alpha(t)|^p\Big),
\]
we get
\begin{align}\label{eq:ito-corr-closed}
    \frac{bp\sigma^2}{2m^2}\chi_S
    \E\Big[\frac{1}{J}\sum_{j=1}^J |V_t^j|^{p-2}|X_t^j-\mathcal M_\alpha(t)|^2\Big]
    &\le
    \frac{b\sigma^2}{m^2}\chi_S\,\frac{p-2}{2}\,\eps_6^{-\frac{p-2}{2}}\,\E[V_p(t)]
    \\
    &\quad
    + \frac{b\sigma^2}{m^2}\chi_S\,2^{p-1}\,\eps_6\,
    \E\Big[X_p(t)+|m_X(\emp{t})-\mathcal M_\alpha(t)|^p\Big].
    \nonumber
\end{align}

\smallskip
\noindent
\emph{(vi) Final Young-closed differential inequality.}
Collecting \eqref{eq:young-meanV}--\eqref{eq:ito-corr-closed} into \eqref{eq:Mp-decay-std} yields
\begin{align}\label{eq:Mp-option1-young-closed}
    \frac{\d}{\d t}\,\E[{L}_{\mathrm{std},p}(t)]
    \le
    -\lambda_{X,p}^{\mathrm{std}}\,\E[X_p(t)]
    - \text{\rm (velocity dissipation terms)}
    + \mathcal R_p(t),
\end{align}
where the candidate spatial dissipation rate is
\begin{align}\label{eq:lambdaXp-def}
    \lambda_{X,p}^{\mathrm{std}}
    &:=
    \frac{c}{m}
    - \frac{\gamma c}{m}\frac{p-1}{p}\eps_0
    - \frac{b}{m}\eps_1^p
    - \frac{c}{m}\frac{p-1}{p}\eps_5^{-\frac{p}{p-1}}
    \\
    &\quad
    - (p-1)c\frac{p-2}{p}(\eps_3 + \eps_4)
    - \frac{b\sigma^2}{m^2}\chi_S 2^{p-1} \eps_6,
    \nonumber
\end{align}
and
\begin{align}\label{eq:Rp-def}
    \mathcal R_p(t)
    &:=
    \Bigg[
        \frac{b}{m}\eps_2^p
        + \frac{c}{m}\frac{1}{p}\eps_5^p
        + \frac{b\sigma^2}{m^2}\chi_S 2^{p-1}\eps_6
    \Bigg]
    \E\big[|m_X(\emp{t})-\mathcal M_\alpha(t)|^p\big].
\end{align}

\begin{remark}[The bottleneck of the standard formulation]
The preceding calculation shows that the na\"ive centered functional does not provide a robust decay mechanism for the full kinetic CBO system. The reason is structural. In the unshifted formulation, the only leading restoring contribution in the spatial channel comes from the original force term
\[
    -\frac1m\bigl(X-\mathcal M_\alpha\bigr),
\]
so the spatial dissipation is generated at scale \(1/m\). However, the same term also generates the weighted-mean error contribution involving \(m_X(\emp{t})-\mathcal M_\alpha(t)\), and after invoking the weighted-mean estimate from the main text,
\[
    \E|m_X(\emp{t})-\mathcal M_\alpha(t)|^p
    \le
    C_{\mathrm{lem}}\,\E[X_p(t)],
\]
this produces a penalty in the very same spatial channel. Thus both the available coercivity and the error term enter at the same basic scale \(1/m\). At the level of leading coefficients, this leaves no robust margin once \(C_{\mathrm{lem}}\) is of order one or larger.

By contrast, the shifted variable
\[
    \hat Z=(V-m_V)-\gamma^{-1}Y
\]
extracts an additional restoring contribution from the friction term, so that the effective spatial coefficient becomes
\[
    K_Y=\frac{2}{m}+\frac{1}{\gamma^2},
\]
while the weighted-mean error remains at order \(1/m\). This is the extra margin that allows the uniform-in-time estimate to close.
\end{remark}

\begin{remark}[Why the case \(p=2\) is not enough]
The computation above shows that a second-moment estimate may still be closed under a restrictive smallness condition. However, the full uniform-in-time propagation-of-chaos argument requires higher-moment control in order to derive the concentration bounds and to handle the weighted-mean error uniformly in time. Thus the case \(p=2\) alone is not sufficient for the main theorem.
\end{remark}

\pagebreak

\printbibliography
\end{document}